\RequirePackage{fix-cm}
\documentclass{svjour3}                     
\smartqed  
\usepackage{graphicx}
\usepackage{amsmath, amssymb}

\usepackage[colorlinks=true]{hyperref} 
\hypersetup{colorlinks = true, linkcolor = red, citecolor = blue}
\usepackage{todonotes} 
\usepackage{mathtools}   
\DeclarePairedDelimiter{\ceil}{\lceil}{\rceil}
\numberwithin{equation}{section}

\usepackage{enumerate}

\usepackage{multirow}

\usepackage{color}
\usepackage{bbm}
\usepackage{subfigure}

\usepackage{rotating}

\usepackage{placeins}

\usepackage[misc,geometry]{ifsym}

\newtheorem{algorithm}{Algorithm}

\allowdisplaybreaks

\usepackage[framemethod=tikz]{mdframed}
\usetikzlibrary{decorations.pathmorphing,calc}
\newcommand\wavydecor{%
    \draw[decoration={coil,aspect=0.1,segment length=5pt,amplitude=1.0pt},decorate,line width=1.5pt,black]
      (O|-P) -- (O);
}
\newmdenv[
hidealllines=true,
innerleftmargin=10pt,
innerrightmargin=0pt,
innertopmargin=0pt,
innerbottommargin=0pt,
leftmargin=-10pt,
skipabove=.5\baselineskip,
skipbelow=.5\baselineskip,
singleextra={\wavydecor},
firstextra={\wavydecor},
secondextra={\wavydecor},
middleextra={\wavydecor}
]{done}
\newcommand\wavydecorgreen{%
    \draw[decoration={coil,aspect=0.1,segment length=5pt,amplitude=1.0pt},decorate,line width=1.5pt,green]
      (O|-P) -- (O);
}
\newmdenv[
hidealllines=true,
innerleftmargin=10pt,
innerrightmargin=0pt,
innertopmargin=0pt,
innerbottommargin=0pt,
leftmargin=-10pt,
skipabove=.5\baselineskip,
skipbelow=.5\baselineskip,
singleextra={\wavydecorgreen},
firstextra={\wavydecorgreen},
secondextra={\wavydecorgreen},
middleextra={\wavydecorgreen}
]{final}

\newcommand{\Aq}{\alpha^{(q)}_{i,\mb{k}}}

\newcommand{\Aqb}{\bar{\alpha}^{(q)}_{i,\mb{k}}}

\newcommand{\AqbM}{\bar{\alpha}_{i,\mb{k}}^{(q,M)}}

\newcommand{\Aqo}{\alpha^{(q)}_{i,k}}

\newcommand \byqi{\bar{y}^{(q)}_i}
\newcommand \byqiM{\bar{y}^{(q,M)}_i}
\newcommand \byqM{\bar{y}^{(q,M)}}

\newcommand{\Resq}[1]{\Sq_i\left(X_{i:N}^{i#1}\right)}
\newcommand{\ResqM}[1]{S_i^{(q,M)}\left(X_{i:N}^{i#1}\right)}

\newcommand{\LG}{\mathcal{L}_\Gamma}

\newcommand \cloud{\cC}
\newcommand{\norm}[1]{\left\lVert#1\right\rVert}
\newcommand{\E}{\mathbb{E}}
\newcommand \Esp[1]{\E\left[#1\right]}

\newcommand \Var[1]{\mathbb{V}{\rm ar}\left[#1\right]}

\renewcommand{\P}{\mathbb{P}}
\newcommand{\R}{\mathbb{R}}
\newcommand{\N}{\mathbb{N}}
\renewcommand\Re[1]{{\rm Re}\left(#1\right)}

\newcommand{\1}{\mathbf{1}}

\newcommand{\mb}[1]{\mathbf{#1}}
\newcommand\mbx{\mb{x}}
\newcommand\mbu{\mb{u}}
\newcommand\mbk{\mb{k}}

\newcommand \Fnul{F_{\nu_{l}}}
\newcommand \Fnuo{F_{\nu_{1}}}

\newcommand \Fnulinv[1]{F^{-1}_{\nu_{#1}}}
\newcommand \Fnuoinv{F^{-1}_{\nu_{1}}}
\newcommand \Fnudinv{F^{-1}_{\nu_{d}}}

\newcommand \dnFnudinv[1]{{\rm d}^{#1}_u\Fnuoinv(u)}

\newcommand \dnFnulinv[2]{{\rm d}^{#1}_{u_{#2}}\Fnulinv{#2}(u_{#2})}
\newcommand \dnFnulinvapp[2]{{\rm d}^{#1}_{u}\Fnulinv{#2}(u)}

\newcommand \Fnuinv{F^{-1}_{\nu}}

\newcommand{\cA}{\mathcal{A}}

\newcommand{\cC}{\mathcal{C}}
\newcommand{\cE}{\mathcal{E}}
\newcommand{\cF}{\mathcal{F}}
\newcommand{\cG}{\mathcal{G}}

\newcommand{\cT}{\mathcal{T}}

\renewcommand\Xi{X^i}
\newcommand\Xii{X^i_{i}}
\newcommand\Xij{X^i_{j}}
\newcommand\XiN{X^i_{N}}
\newcommand\Xijp{X^i_{j+1}}

\newcommand\Sq{S^{(q)}}
\newcommand\SqM{S^{(q,M)}}

\newcommand{\DEG}{{\tt DEG}}

\newcommand{\HFG}{{$\bf (A_{f,g})$}}
\newcommand{\HBS}{{$\bf (A_{b,\sigma})$}}

\newcommand{\Hnu}{{$\bf (A_\nu)$}}

\newcommand \dyi[1]{{\rm d}_x^{#1}y_i}
\newcommand \pdyi[2]{{\partial}_{x_{#2}}^{#1}y_i}

\newcommand \yq{y^{(q)}}
\newcommand \yqi{\yq_{i}}

\newcommand \yqjp{\yq_{j+1}}

\newcommand \hq{h^{(q)}}

\newcommand \hqi{\hq_{i}}

\newcommand \dhqi{{\rm d}_u\hqi}

\newcommand \ddhqi[1]{{\rm d}_u^{#1}\hqi}

\newcommand \dnu[1]{{\rm d}_x^{#1}\nu_l}

\newcommand \pddlhqi[2]{\partial^{#1}_{u_{#2}}\hqi}

\newcommand \ddyqi[1]{{\rm d}_x^{#1}\yqi}
\newcommand \pddyqi[2]{\partial^{#1}_{x_{#2}}\yqi}

\newcommand \pdlhqi[2]{\partial^{#1}_{u_{#2}}\hqi}

\newcommand \pdhqi[2]{\partial^{#1}_{#2}\hqi}

\newcommand \LU{L^2_U([0,1]^d)}

\newcommand \LNU{L^2_\nu(\R^d)}

\newcommand{\dd}{{\rm d}}
\newcommand{\dx}{\dd x}
\newcommand{\ds}{\dd s}
\newcommand{\dt}{\dd t}

\newcommand{\du}{\dd u}

\newcommand{\dz}{\dd z}
\newcommand{\dW}{\dd W}
\newcommand{\dxx}{\dd\mathbf{x}}

\newcommand{\duu}{\dd\mathbf{u}}

\newcommand{\ti}{{t_i}}
\newcommand{\tip}{{t_{i+1}}}
\newcommand{\tj}{{t_{j}}}
\newcommand{\tjp}{{t_{j+1}}}

\begin{document}

\title{Quasi-Regression Monte-Carlo scheme for semi-linear PDEs and BSDEs with large scale parallelization on GPUs\thanks{The first author research is part of the Finance for Energy Markets (FiME) lab, of the Chair Financial Risks of the Risk Foundation and of the ANR project CAESARS (ANR-15-CE05-0024). The second author has been financially supported by the Chair Financial Risks of the Risk Foundation, the Spanish grant MTM2016-76497-R and the \textit{Xunta de Galicia} 2018 postdoctoral grant. The third author was partially supported by Spanish grant MTM2016-76497-R.}
}

\titlerunning{Quasi-Regression Monte-Carlo scheme for semi-linear PDEs and BSDEs}        

\author{E. Gobet \and J. G. L\'opez-Salas \and C. V\'azquez}


\institute{E. Gobet \at
              Centre de Math\'ematiques Appliqu\'ees, \'Ecole Polytechnique and CNRS, route de Saclay, 91128 Palaiseau cedex, France \\
              \email{emmanuel.gobet@polytechnique.edu}           
           \and
           J. G. L\'opez-Salas (\Letter \hspace{0.1cm}Corresponding author, \email{jose.lsalas@udc.es}) \at
              Centre de Math\'ematiques Appliqu\'ees, \'Ecole Polytechnique and CNRS, route de Saclay, 91128 Palaiseau cedex, France
           \and
           C. V\'azquez \at 
	       Department of Mathematics, Faculty of Informatics, Universidade da Coru\~na, Campus de Elvi\~na s/n, 15071, A Coru\~na, Spain \\
	       \email{carlosv@udc.es}
}

\date{Received: date / Accepted: date}

\maketitle
\begin{abstract}
 In this article we design a novel quasi-regression Monte Carlo algorithm in order to approximate the solution of discrete time backward stochastic differential equations (\mbox{BSDEs}), and we analyze the convergence  of the proposed method. The algorithm  also approximates the solution to the related semi-linear parabolic partial differential equation (PDE) obtained through the well known Feynman-Kac representation. For the sake of enriching the algorithm with high order convergence a weighted approximation of the solution is computed and appropriate conditions on the parameters of the method are inferred. With the challenge of tackling problems in high dimensions we propose suitable projections of the solution and efficient parallelizations of the algorithm taking advantage of powerful many core processors such as graphics processing units (GPUs).
\end{abstract}


\section{Introduction}

Since the 1940s, Feynman-Kac formula has provided a probabilistic interpretation of the solutions of linear partial differential equations (PDEs). Thence, Monte Carlo method has become an alternative to deterministic numerical algorithms for solving linear PDEs. While traditional deterministic methods like finite differences, finite elements or finite volumes suffer from the so-called \textit{curse of dimensionality} \cite{bellmann} (i.e. as the dimension grows, the complexity of the algorithms needed to achieve a given accuracy grows exponentially), the rate of convergence of Monte Carlo methods does not depend on the dimension of the state space. As a consequence, in order to solve equations of dimension greater than $4$ or $5$, the probabilistic approach often remains the only numerical method available.

In the 1990s, the theory of backward stochastic differential equations (BSDEs) provided a probabilistic interpretation for non-linear problems. BSDEs are stochastic differential equations (SDEs) for which a terminal condition is specified and their solutions are pairs of processes. Linear BSDEs were introduced by Bismut \cite{bis:73} in 1973 in order to study stochastic control problems and their connections with the Pontryagin maximum principle of optimality. A rigorous study of (non-linear) BSDEs was started by Pardoux and Peng in 1990, who proved in  their seminal article \cite{pard:peng:90} the first well-posedness result for non-linear BSDEs, see also \cite{pard:peng:92}. These equations have the general form of
 \begin{align}
 -\dd Y_t &= f(t,Y_t,Z_t) \dd t - Z_t \dd W_t, \nonumber \\
 Y_T &= \xi, \nonumber
\end{align}
where $0\leq t \leq T $, $(W_t)_{t\geq 0}$ is a Brownian motion, the terminal condition $\xi$ is a random variable, and the solution is a pair of adapted processes $(Y,Z)$.
The first applications of BSDEs in finance were conducted by El Karoui, Peng and Quenez \cite{elka:peng:quen:97} in 1997. They applied BSDEs in a European option pricing problem, where the randomness of $(Y,Z)$ comes from a forward SDE (FSDE in short) in $X$ as follows,
   \begin{align*}
     \dd X_t &= b(t,X_t) \dt + \sigma(t,X_t) \dW_t, \quad   X_0 = x_0,  \\
     -\dd Y_t &= f(t,X_t,Y_t,Z_t) \dd t - Z_t \dW_t, \quad  Y_T = g(X_T);
    \end{align*}
the function $f$ models possible imperfections and restrictions when trading in the markets.
Now the terminal condition of the BSDE is determined by the stochastic terminal value of the FSDE through some deterministic function $g$. This system of SDEs, to which we will devote this article, is known as decoupled\footnote{Fully coupled FBSDEs have the general form   \begin{align*}
     \dd X_t &= b(t,X_t,Y_t,Z_t) \dd t + \sigma(t,X_t,Y_t,Z_t) \dd W_t, \quad   X_0 = x_0,  \\
     -\dd Y_t &= f(t,X_t,Y_t,Z_t) \dd t - Z_t \dd W_t, \quad  Y_T = g(X_T).
    \end{align*}} forward-backward stochastic differential equation (FBSDE in short). It is also referred as Markovian BSDE, since the terminal condition of the BSDE and $f$ are random only through a Markov process. 
    
After the already mentioned celebrated work of El Karoui, Peng and Quenez, the interest in BSDEs has increased, mainly due to the connection of these tools with stochastic control and PDEs, connection that will be stated clearly soon. Let us start with a rigorous definition of the problem at hand.

\paragraph{The problem.}
In this work we are interested in numerically approximating the solution $(X,Y,Z)$ of a decoupled forward-backward stochastic differential equation
\begin{align}
\label{eq:fbsde}
Y_t& =  g(X_T) + \int_t^T f(s,X_s,Y_s,Z_s) \ds - \int_t^T Z_s \dW_s, \\
\label{eq:eds}
X_t & =  x + \int_{ 0} ^t b(s,X_s) \ds + \int_{0} ^t \sigma(s,X_s) \dW_s.
\end{align}
The terminal $T>0$ is fixed. These equations are considered in a filtered probability space $(\Omega,\cF,\P,(\cF_t)_{ 0 \leq t \leq T})$ supporting a $q\geq 1$  dimensional Brownian motion $W$. The filtration is assumed to be the natural filtration of $W$ augmented with the $\P$-null  sets. In this representation, $X$ is a $d$-dimensional adapted continuous process (called the forward component), $Y$ is a scalar adapted continuous process and $Z$ is a $q$-dimensional progressively measurable process. Consideration of multi-dimensional $Y$ is quite straightforward and is discarded for the sake of simplicity. Regarding terminology, $g(X_T)$ is called \emph{terminal condition} and $f$ the \emph{driver}. Under standard assumptions (like $L^2$-integrability and Lipschitz regularity), there is a unique solution $(X,Y,Z)$ (see \cite{pard:rasc:14} for broad account on existence/uniqueness assumptions). Detailed hypotheses for the current study are exposed later.

Our aim in this article is to design a new numerical scheme for approximating the solution to the above BSDE. When $f$ is linear in $(Y,Z)$, explicit representations are available \cite[Theorem 1.1]{elka:peng:quen:97}, which can be easily turned into simple and efficient Monte-Carlo algorithms. Our focus is thus more on the case of non-linear driver yielding non-linear equations.

\paragraph{Feynman-Kac representations.} Let us clearly show that FBSDEs are important tools in semi-linear PDEs and stochastic control theory. Denote by  $\cal A$  the infinitesimal generator of $X$:
$$\cA =\sum_i b_i(t,x)\partial_{x_i}+\frac{1}{2 }\sum_{i,j} [\sigma \sigma^\top]_{i,j}(t,x)\partial^2_{x_i,x_j},$$
where $^\top$ denotes the transpose operator.
Consider the parabolic partial differential equation (PDE) of the form
\begin{equation}
\label{eq:pde}
\begin{cases}
\partial_{t}u(t,x)+{\cal A}u(t,x)+f(t,x,u(t,x),(\nabla_x u\sigma)(t,x))=0, \quad t<T,\\
u(T,.)=g(.).
\end{cases}
\end{equation}
Assuming that $u$ is smooth enough, the It\^o formula leads to 
$$u(t,X_t)=u(0,x)+\int_0^t (\partial_{t}u(s,X_s)+{\cal A}u(s,X_s))\ds+\int_0^t (\nabla_x u \sigma)(s,X_s) \dW_s.$$
Using the PDE \eqref{eq:pde} and noting $u(T,X_T)=g(X_T)$, we get
$$u(t,X_t)=g(X_T)+\int_t^T f(s,X_s,u(s,X_s), (\nabla_x u \sigma)(s,X_s) )\ds-\int_t^T (\nabla_x u \sigma)(s,X_s)  \dW_s,$$
therefore proving that the solution $(Y,Z)$ can be represented as
$$(Y_{t},Z_{t})=(u(t,X_t),(\nabla_{x} u\sigma) (t,X_{t})).$$
This is the well-known Feynman-Kac representation, making connection between semi-linear PDEs and BSDEs.  One may prefer one or the other representations depending on the application at hand, on the numerical scheme to use, and on the theoretical estimates available (of probabilistic or PDE nature). As a consequence,  solving numerically $(Y,Z)$ for any realization of $X$ or computing $(u,\nabla u)$ in the full space $\R^d$ gives, in both cases, essentially access to the same solution. The above reasoning holds for Cauchy boundary conditions and parabolic PDEs as presented above, however extensions in the directions of various boundary conditions (Dirichlet, Neumann, free-boundary) or stationary problems (elliptic PDEs) are also available, see \cite[Sections 3.8, 5.4, 5.6, 5.7 and 5.8]{pard:rasc:14} and references therein.

In view of \eqref{eq:pde}, the link with stochastic control appears  clearly, by analogy with the Hamilton-Jacobi-Bellman equation. Since the non-linearity in $f$ is related only to the first derivative of $u$ (the $Z$ component), the underlying control problem is associated to a control  in the drift only (not in the diffusion coefficient). See \cite{elka:peng:quen:97,elka:hama:mato:08} for some examples in finance.

\paragraph{Applications.} We now conclude the presentation of connection between BSDEs and PDEs by giving  a broader overview of applications where these semi-linear PDEs come into play. We follow the presentation of \cite[Section 7.1]{gobe:16}, some examples can be found in \cite[Chapter 2]{henr:81} and \cite[Chapter 14]{smol:94}.

\begin{description}
\item \emph{Ecology.} In a region of the plane ($d=2$), suppose the existence of  2 interacting species in the same environment and denote $u_k(t,x)$ ($k=1,2$) the population density of the individuals of each species at a point of the space $x$ and at a  time $t$. These population densities obey to the system
\begin{equation}
\label{eq:reac:diff:ecology}
\begin{cases}
\partial_t u_1(t,x) =\alpha_1 \Delta  u_1(t,x) +f_1(t,x,u_1,u_2),\\
\partial_t u_2(t,x) =\alpha_2 \Delta  u_2(t,x) +f_2(t,x,u_1,u_2), \quad t>0, x\in \R^2,\\
u_i(0,\cdot)\text{ {given}},
\end{cases}
\end{equation}
where the two Laplacian operators model the diffusive movement (as a Brownian motion) of each species. Note that the initial condition at $t=0$ can be turned into a terminal condition at $t=T$ as in our previous setting, this is just a simple 
time reversal  $t\leftrightarrow T-t$.

The  function $f=(f_1,f_2)$ plays the role of a growth rate of the population, describing the available local  resources and interactions between the species: the case $\partial_{u_2} f_1<0$ and $\partial_{u_1} f_2>0$ corresponds to the predator-prey model, where the growth rate of the species 1 (prey) decreases in the case of high density of the species 2 (predator) and conversely for the growth rate of predators; the case $\partial_{u_2} f_1>0$ and $\partial_{u_1} f_2>0$ models symbiosis, where each species benefits from the other;  $\partial_{u_2} f_1<0$ and $\partial_{u_1} f_2<0$ describes competition between species. This kind of model, introduced in the 30s by Fisher, Kolmogorov, Petrovsky and Piskunov, is discussed with extra references in \cite{shig:kawa:97}.

\item \emph{Chemistry.} Suppose that a container contains $N$ chemical compounds taking part in $R$  independent reactions. Denote $c_i$ the concentration of the $i$-th {compound} and $\theta$ the temperature. Then their evolution  follows the $N+1$ equations (in $\R^+\times \R^3$)
\begin{equation}
\label{eq:reac:diff:chimie}
\begin{cases}
\varepsilon_p \partial_t c_i=D_i \Delta c_i+\sum_{j=1}^R \nu_{ij} g_j(c_1,\dots,c_N,\theta),\quad i=1,\dots,N,\\
\rho c_p \partial_t \theta = k \Delta \theta - \sum_{j=1}^R \sum_{i=1}^N\nu_{ij} H_i g_j(c_1,\dots,c_N,\theta),\\
c_i(0,\cdot)\text{ and }\theta(0,\cdot)\text{ {given}},
\end{cases}
\end{equation}
where $g_j$ is the speed of the $j$-th reaction and $H_i$ is the partial molar enthalpy of the $i$-th compound. For more details, see \cite{gava:68}. For reaction-diffusion equations, see \cite{smol:94}.

\item \emph{Neuroscience.} The famous model of Hodgkin and Huxley is a set of equations describing the psychological phenomenon of signal transmission in the axon (nerve fiber), showing the dependences between electrical excitability and various chemical ion concentrations. The PDE system is of  size $4$, the first unknown $u$ represents the electric potential and the three other unknowns $(v_1,v_2,v_3)$ are chemical concentrations: the system writes for $(t,x)\in \R^+\times \R $ (linear neuron)
\begin{equation}
\label{eq:reac:diff:neurosience}
\begin{cases}
c_0\partial_t u =\frac{ 1 }{R }  \partial^2_{xx }u +\kappa_1 v_1^3 v_2(c_1-u)+\kappa_2 v_3^4 (c_2-u)+\kappa_3 (c_3-u),\\
\partial_t v_1 =\varepsilon_1 \partial^2_{xx }  v_1 +g_1(u)(h_1(u)-v_1),\\
\partial_t v_2 =\varepsilon_2 \partial^2_{xx }  v_2 +g_2(u)(h_2(u)-v_2),\\
\partial_t v_3 =\varepsilon_3 \partial^2_{xx }  v_3 +g_3(u)(h_3(u)-v_3),\\
u(0,\cdot)\text{ and }v_i(0,\cdot)\text{ given},
\end{cases}
\end{equation}
for different positive constants $c_i, R, \kappa_i, \varepsilon_i$ and different functions $g_i,h_i$. There also exists a simplification of this model, known as the FitzHugh-Nagumo model. 

\item \emph{Materials physics.} The Allen-Cahn equation is a prototype model of  phase transition with a diffusive interface, used to model, for example, a solid/liquid phase transition \cite{visi:96}. The system is one-dimensional, and takes the form (in $\R^+\times \R^3$)
\begin{equation}
\label{eq:reac:diff:materiau}
\begin{cases}
\partial_t u =\varepsilon \Delta u + u(1-u^2),\\ 
u(0,\cdot) \text{ given.}
\end{cases}
\end{equation}
The solution $u$ thus represents an order parameter defining the arrangement of atoms in a crystal lattice.

\item \emph{Economy and finance.} There are numerous applications of BSDEs (and semi-linear PDEs) in finance, we refer the
reader to \cite{cvit:ma:96,elka:peng:quen:97,crep:13}. Let us mention:  valuation/hedging of contingent claim in complete markets; some market imperfections, such as higher interest rate for borrowing
\cite{berg:95}; incomplete markets \cite{pham:09} or super-replication \cite{elka:quen:95};  connections with recursive utilities in \cite{duff:epst:92}; non-linear pricing rule \cite{peng:03:a}; second order BSDEs (2BSDEs) \cite{cheri:mete:touz:vict:2006,soner:touz:zhang:2011} which provide a stochastic representation for solutions of fully non-linear parabolic PDEs through a driver depending also on another process directly connected with the Hessian of the solution of the corresponding PDE; reflected BSDEs, where the solution $Y$ is constrained to stay above a given process called obstacle, were introduced by El Karoui et al in \cite{karoui:kapo:pard:peng:quen97}, and combined with jumps are able to connect with partial-integro differential equations \cite{barles:buck:pard:97,delong:13}, etc. See \cite{elka:hama:mato:08,crep:13} for extra references of applications in finance.
\end{description}

\paragraph{Numerical approximations.} In the two past decades, there have been several attempts to design efficient approximation schemes for BSDEs. First, the four step algorithm was proposed by Ma, Proter and Yong \cite{ma:protter:yong:94} in 1994 to solve general FBSDEs. Their method is based on the finite difference approximation of the associated quasilinear parabolic PDE, and therefore becomes unaffordable in high dimension. In 1997, a discrete scheme based on the approximation of the Brownian motion by some discrete process was proposed by Chevance in \cite{chev:97}, while Bally considered a random time scheme in \cite{ball:97}, although this method requires the computation of conditional expectations to become a fully implementable solver.  At that time, only time-discretization issue was considered. Since then, optimal discretization strategies have been derived: see \cite{geis:geis:gobe:12} and references therein. All probabilistic approaches proposed later roughly follow three steps: time discretization, dynamic programming principle and computation of conditional expectations. The problem of discretization and simulation of the forward component $X$ is well-understood, one could use standard Euler scheme or other choices like Milstein scheme, see \cite{kloeden:platen:92}. The difficulty  lies in the backward component $Y$, since a naive discretization of the backward SDE in $Y$ is not guaranteed to be adapted. In order to obtain a computationally viable  backward induction scheme one should take conditional expectations. If the time grid is $\{t_0:=0,\dots,t_i,\dots,t_N=:T\}$, it leads to a nested expectation problem of the form  (also called Dynamic Programming Equation, or DPE for short)
\begin{align}
\label{eq:ODP:intro:a}
Z_{\ti}&\approx\Esp{ Y_{t_{i+1}} \frac{(W_\tip-W_\ti)^\top}{ \tip-\ti}\mid \cF_{\ti}},\\
\label{eq:ODP:intro:b}
Y_{T} = g(X_T), \quad Y_{\ti}&\approx\Esp{ Y_{t_{i+1}} + f(t_i,X_{\ti},Y_{\tip},Z_\ti) (\tip-\ti)\mid \cF_{\ti}},
\end{align}
for $i=N-1,\ldots,0$. We use the symbol $\approx$ to emphasize the heuristics, a rigorous justification can be done using the aforementioned references. 
The scheme \eqref{eq:ODP:intro:a}-\eqref{eq:ODP:intro:b}, which is known as one-step forward dynamic programming scheme, could be updated taking into account a resimulation of the forward process till terminal time $T$, thus resulting in the so-called multi-step forward dynamic programming scheme, which is written as follows
\begin{align}
\label{eq:MDP:intro:0a}
Z_{\ti}&\approx\Esp{ \bigg(g(X_T) +\sum_{j=i+1}^{N-1}f(t_j,X_{\tj},Y_{\tjp},Z_\tj) (\tjp-\tj)\bigg) \frac{(W_\tip-W_\ti)^\top}{ \tip-\ti}\mid \cF_{\ti}},\\
\label{eq:MDP:intro:0b}
Y_{\ti}&\approx\Esp{ g(X_T) +\sum_{j=i}^{N-1}f(t_j,X_{\tj},Y_{\tjp},Z_\tj) (\tjp-\tj)\mid \cF_{\ti}},
\end{align}
for $i=N-1,\ldots,0$. 
\\Analytically, both schemes \eqref{eq:ODP:intro:a}-\eqref{eq:ODP:intro:b} and \eqref{eq:MDP:intro:0a}-\eqref{eq:MDP:intro:0b} are equivalent. However, differences arise when conditional expectations are numerically approximated. In fact, the multi-step scheme leads to better error estimates of the solutions computed by the developed algorithms, see \cite{bend:denk:07,gobe:turk:2016}. The above backward dynamic programming equations make clearly appear that at each of $N$ grid times, one has to compute a conditional expectation function, that is a function from $\R^d$ to $\R$, this is in a way equivalent to solving the PDE at each time point $t_i$. This is a difficult problem, with increasing difficulty as the dimension gets large and/or as the unknown function gets less smooth (\emph{curse of dimensionality}), like for finite differences schemes  for PDEs. Among the numerous approaches, let us mention the quantization method of Bally and Pag\`es \cite{ball:page:03} (and generalized to coupled FBSDEs by Delarue and Menozzi \cite{dela:meno:06,dela:meno:08}), which is an optimal space discretization of the underlying dynamic programming equation, suitable to compute easily the conditional expectations on the space grid; in  \cite{bouc:touz:04} the authors compute the regression function with Malliavin calculus integration by parts combined with Monte Carlo simulations; the use of empirical regression methods and machine learning techniques is developed in \cite{gobe:lemo:wari:05,lemo:gobe:wari:06} (also called \emph{regression Monte-Carlo method}). Optimal convergence rates with regression methods are derived in \cite{gobe:turk:2016}, variance reduction using importance sampling is designed and analysed in \cite{bend:mose:10,gobe:turk:17}.

Recently, Henry-Labordere et al. \cite{henr:tan:touz:14} and Bouchard et al. \cite{bouc:tan:warin:zou:2017} proposed a branching method and a time step randomization applied to the Feynman-Kac representation of the PDE. Although branching techniques really overcome the curse of dimensionality, unfortunately they are only limited to small maturities, some small non-linearities and mainly to non-linearities that are polynomial in the solution and its gradient. Lately, in \cite{wei:jie:jentzen:2017} the authors propose a deep learning technique, referred as deep BSDE, to solve semi-linear PDEs or the corresponding BSDEs. They interpret the BSDE as a stochastic control problem with the gradient of the solution being the policy function, which is approximated or learned by a deep neural network, as has been done in deep reinforcement learning. More precisely, after discretizing the forward SDE in $X$ using Euler discretization, the algorithm reads the BSDE for $Y$ as a forward SDE, and the method tries to learn the value of $Y$ and $Z$ at each time step of the Euler scheme by minimizing a global loss function between the forward simulation of $Y$ until the maturity $T$ and the goal $g(X_T)$. In the recent work \cite{cwn:mik:warin:2019}, the authors give an improved version of the deep BSDE method, where a fixed point problem is solved using deep learning techniques. Numerical results show that these kind of algorithms are able to solve  problems with some accuracy in dimension $50$ and even above. However, these methods suffer so far from the lack of rigorous error estimates and the matter that convergence is not guaranteed: as in linear regression methods \cite{gobe:turk:2016}, there must be a tricky interplay between the number of time steps, the number of simulations, the number of neurons  and layers, to ensure a good accuracy. Moreover, one has to account for the likely possibility that the stochastic gradient descend method used to minimize the loss function is trapped in a local minimum.

\paragraph{Parallel computations.} Most of the above methods do not allow parallel computations because of the non-linearity of the equations. Recently in \cite{gltv}, the authors have proposed  a novel approach inspired from regression Monte-Carlo schemes suitable for massive parallelization, with excellent results on GPUs. It takes advantage of local approximations of the regression function in different strata of the state space $\R^d$ (partitioning estimates) and of a new stratified sampling scheme to guarantee that enough simulations are available in each stratum to compute accurately the regression functions. The parallelization step is made with respect to the strata.

Our current work lies in this vein and is aimed at overcoming some drawbacks of the stratified regression scheme of \cite{gltv}. In the previous scheme, the solution $u$ to the PDE is represented as a piecewise polynomial function (like in a discontinuous Galerkin method \cite{dipi:ern:11}) and therefore, even if the exact solution is continuous, the approximation may be not, especially at the interface of the different strata. In the current work, we propose a different scheme, where the approximated solution is decomposed on basis functions that are smooth and that have natural orthonormal properties with respect to the sampling distribution, thus avoiding the possible discontinuity at the interfaces. Moreover, the sampling distribution is required to have some nice stability properties, useful for  the propagation of error in the DPE
, see Theorem \ref{th:main}; we prove that the Student's t-distribution meets this requirement, see Theorem \ref{theo:uses}. One possible set of smooth and orthonormal functions is given by the trigonometric basis composed with a space transformation related to the Cumulative Distribution Function (CDF for short) of the Student's t-distribution. We call such a basis the \emph{Student-cosine basis}.
%
%
Instead of a regression Monte Carlo scheme, we design a \emph{quasi-regression Monte-Carlo scheme}: we take advantage of the fact that  the basis functions enjoy orthonormal properties, in order to compute the basis coefficients as simple Monte-Carlo averages, instead of solutions of least-squares problems. On the one hand, it leads to a simplified and quicker procedure (because we avoid inverting linear systems). On the other hand, we obtain a globally smooth approximation which is mainly suitable for smooth solutions $u$ \cite{funaro,Canuto-Hussaini-Quarteroni-Zang}: this is a significant difference with partitioning estimates which can approximate well locally (ir-)regular functions. In this new scheme, the notion of strata is lost and the parallelization is now made with respect to the numerous required simulations and basis coefficients.
Moreover, we establish tight non-asymptotic error estimates (as a function of the size of the function basis and the number of Monte-Carlo simulations). 

In recent years, GPUs have become increasingly attractive for general purpose parallel computations. GPUs are not only impressive graphics tools but also highly parallel programmable processors that can deliver order-of-magnitude performance gains over optimized CPU applications. CPU hardware is designed to handle complex tasks, instead, GPUs only do one thing well, they handle millions of repetitive low level tasks. These high performance computing devices offer an arithmetic capability and a memory bandwidth that substantially outpace their CPU counterpart. In fact, a single NVIDIA Titan V GPU can sustain over 13.8 tera-floating-point operations per second (Tflops) and a memory bandwidth of 652.8 GB/s. In the last years GPUs have been heavily used in scientific computation in order to achieve dramatic runtime reductions over traditional CPU codes. Particularly successful applications of GPUs include Monte Carlo methods like the one we target in the present work. Indeed, Monte Carlo algorithms follow all design principles of GPU computing, they give rise to embarrassingly parallel not divergent codes that can exploit the GPU arithmetic intensity and its streaming memory bandwidth. In the last part of this article our goal is to offer an efficient implementation of the proposed algorithm taking advantage of several GPUs, which will allow us to address the solution of semi-linear PDEs in high dimensions.

\paragraph{Organization of the paper.} In Section \ref{section:Mathematical framework and basic properties}, we set the
framework of our study, define some notations used throughout the
paper and describe our algorithm based on the approximation of
conditional expectations by a projection on a finite orthonormal basis of
functions. Compared to the usual DPE 
to solve \eqref{eq:MDP:intro}, we propose an equivalent modification that includes a weighted factor: this weight is related to the sampling distribution and its crucial purpose is to lead to higher order approximation of the unknown function under the condition it is smooth (see the Fourier-based analysis in Section \ref{sec:approError}). 
Proofs of the propagation of error are postponed to Section \ref{section:error}. This new numerical scheme has been tested on GPU, which is exposed in Section \ref{section:GPU implementations}. Numerical results are collected in Section \ref{sec:numExperiments}.

\paragraph{Notations of the paper.} 
\begin{enumerate}[(i)]
 \item $\log(x)$ stands for the natural logarithm of $x\in \mathbb{R}_+$.
 \item $|\mb{x}|$ stands for the Euclidean norm of the vector $\mb{x}=(x_1,\dots, x_n)\in \R^n$.
 \item $\overline{\mb x}$ stands for the $l_1$-norm of the vector $\mb{x}=(x_1,\dots,x_n)\in\R^n$, i.e.\\$\overline{\mb x} = \norm{\mb{x}}_{1} = \sum_{i=1}^n |x_i|$. It will be mainly used when $x$ is a multi-index of integers used for differentiation.
 \item $(x)_+$ means the positive part of $x\in\R$ defined by $(x)_+ = \max(x,0)$.
 \item {$\Re{z}$ denotes the real part of the complex $z\in \mathbb{C}$; we set $\mb{i}:=\sqrt{-1}$.}
 \item For any finite $L>0$ and $\mb{x} = (x_1,\dots,x_n)\in \R^n$, define the soft truncation function
\begin{equation}
\nonumber
\cT_L(\mb{x}) := (-L\vee x_1 \wedge L, \dots, -L \vee x_n \wedge L).
\end{equation}
{Moreover, if $L:\R^n\to \R^+$ is a measurable function, then the above definition of $\cT_L$ is extended by replacing $L$ by $L(\mb{x})$:
\begin{equation}
\label{eq:TL:2}
\cT_L(\mb{x}) := (-L(\mb{x})\vee x_1 \wedge L(\mb{x}), \dots, -L(\mb{x}) \vee x_n \wedge L(\mb{x})).
\end{equation}}
{ \item Let $\mb{j}$, $\mb{k} \in \N^d$ be two multi-indices.  $\mb{j} \leq \mbk$ means that $j_l\leq k_l$ for all $l=1,\ldots,d$.
 \item  The finite multi-indices set $\Lambda\subset \N^d$ is downward closed if $$(\mbk\in\Lambda \mbox{ and } \mb{j} \leq \mbk) \Rightarrow \mb{j} \in\Lambda.$$
 \item Let ${\Lambda}\subset\N^d$ be an arbitrary downward closed multi-index set and $f:\R^d\rightarrow\R$. $f$ is ${\Lambda}$-smooth if ${\partial^{\mb{n}}_x} f$ exist and are continuous for all $\mb n\in {\Lambda}$.
 }
\end{enumerate}
 

\section{Mathematical framework and basic properties}
\label{section:Mathematical framework and basic properties}
\subsection{Model}

\subsubsection{Dynamic Programming Equation}
From now on, we aim at solving the dynamic programming equation of the form \eqref{eq:MDP:intro:0a}-\eqref{eq:MDP:intro:0b}. First, for the sake of simplicity, we intentionally consider the case where the non-linearity does not depend on the gradient (the $Z$ component). Second, we consider the case of $N$ equidistant grid times of $[0,T]$, i.e. we set $t_{k}=k\Delta$ with $\Delta=T/N$. Third, the forward component $X$ is approximated by an Euler scheme with time-step $\Delta$, see \eqref{alg:markov:euler} below. These two last approximations are known to be accurate enough when data ($f,g,b,\sigma$) are sufficiently regular (see \cite{geis:geis:gobe:12} for a precise statement). To simplify notation, we write $X_{i}$ and $Y_{i}$  for an approximation of  $X_{t_i}$ and $Y_{t_{i}}$. 

All in all, our aim now is to solve 
\begin{align}
\label{eq:MDP:intro}
Y_i=\Esp{ g(X_N) +\sum_{j=i}^{N-1}f_j(X_j,Y_{j+1}) \Delta\mid \cF_\ti}
\quad \text{for }i \in \{N-1,\ldots,0\},
\end{align}
where $f_j(x,y):=f(t_j,x,y)$, $f$ being the driver in \eqref{eq:fbsde}. In other words, our subsequent scheme will approximate the solutions to 
\begin{align}\label{eq:edsWithoutZ}
\begin{split}
X_t & =  x + \int_{ 0} ^t b(s,X_s) \ds + \int_{0} ^t \sigma(s,X_s) \dW_s,\\ Y_t &=  \Esp{g(X_T) + \int_t^T f(s,X_s,Y_s) \ds\mid \cF_t}, 
\end{split}
\end{align}
and 
\begin{equation}
\label{eq:pdeWithoutZ}
\partial_{t}u(t,x)+{\cal A}u(t,x)+f(t,x,u(t,x))=0\ \text{ for $t<T$ and } u(T,.)=g(.).
\end{equation}

Our assumptions on the functions $b,\sigma,f, g$ are as follows.
\begin{description}
\item [\HFG] $g$ is a polynomially bounded measurable function from $\R^d$ to $\R$: there exist two finite non-negative constants $C_g$ and $\eta_g$ such that
\begin{align*}
|g(\mb{x})|&\leq C_g(1+|\mb{x}|^2)^{\eta_g/2}, \quad\forall \mb{x}\in \R^d.
\end{align*}
For every {$i<N$}, $f_i(\mb{x},y)$ is a measurable function $\R^d\times \R$ to $\R$, and 
{there exist three finite non-negative constants $L_f, C_f, \eta_f$ such that, for every $i<N$,}
\begin{align*}
|f_i(\mb{x},y)-f_i(\mb{x},y')|&\leq L_f |y-y'|, \quad\forall (\mb{x},y,y')\in \R^d\times(\R)^2,\\
|f_i(\mb{x},0)|&\leq C_f {(1+|\mb{x}|^2)^{\eta_f/2}}, \quad \forall \mb{x}\in \R^d.
\end{align*}
\item [\HBS]  {The drift coefficient $b:[0,T]\times \R^d\mapsto \R^d$ and the diffusion coefficient $\sigma:[0,T]\times \R^d\mapsto \R^d\otimes\R^q $
are bounded measurable functions.}
\end{description}
One important observation is that, due to the Markov property of the Euler scheme, for every $i$, there exist measurable deterministic functions $y_i : \R^d \to \R$, such that $$Y_i = y_i(X_i),$$ almost surely. See the justification in \cite[Section 4.1]{gobe:turk:2016}. 
Moreover, we have an uniform absolute bound for the function $y_i(\cdot)$, the proof is postponed in Section \ref{subsectionProof of Proposition prop:bound}.
\begin{proposition}
\label{prop:bound}
Assume \HFG\ and \HBS. Then, for any $\mb{x}\in \R^d$ and any $0\leq i \leq N$, we have
\begin{align}
\label{eq:prop:bound}
\quad |y_{i}(\mb{x})|    \le C_\eta(C_g +TC_f )e^{C_\eta L_f T}(1+|\mb{x}|^2)^{\frac{\eta_g\lor \eta_f}2}=:L^\star(\mb{x})
\end{align}
with 
\begin{align}
\label{eq:prop:bound:Ceta}
C_{\eta}:=\sup_{0\leq i\leq j\leq N,\mb{x}\in \R^d}\dfrac{\Esp{(1+|X_j|^2)^{\frac{\eta_g\lor \eta_f}2}\mid X_i=\mb{x}}}{(1+|\mb{x}|^2)^{\frac{\eta_g\lor \eta_f}2}}<+\infty.
\end{align}
\end{proposition}
A second crucial observation is that the value functions $y_i(\cdot)$ are independent of how we initialize the forward component. Our subsequent algorithm takes advantage of this observation. For instance,  let $\Xii$ be a random variable in $\R^d$ with some distribution $\nu$ (more details on this to follow) and let $\Xij$ be the Euler scheme evolution of $X_{j}$ starting from $X_{i}$; it writes
\begin{equation}
\label{alg:markov:euler}
\Xijp  = \Xij + b(\tj,\Xij)\Delta + \sigma (\tj, \Xij) (W_{\tjp}-W_{\tj}), \quad j\geq i.
\end{equation}
This flexibility property w.r.t. the initialization then writes
\begin{align}
 \label{eq:MDP:fcs}
y_i(\Xii)&:=\Esp{  g(\XiN) +\sum_{j=i}^{N-1}f_j\left(\Xij,y_{j+1}(\Xijp)\right) \Delta \mid \Xii}.
\end{align}
Approximating the solution to \eqref{eq:MDP:intro} is actually  achieved by approximating the functions $y_i(\cdot)$. In this way, we are directly approximating the solution to the semi-linear PDE \eqref{eq:pdeWithoutZ}. For some reasons that will appear clear later, we define a weighted modification of $y_i$ by 
\begin{equation}\yq_i(\mbx)=\dfrac{y_i(\mbx)}{(1+|\mbx|^2)^{q/2}}
\label{eq:yqi}\end{equation}
 for a damping exponent $q\geq 0$. For $q=0$, $\yqi$ and $y_{i}$ coincide. The previous DPE \eqref{eq:MDP:fcs} becomes
\begin{align}
 \label{eq:MDP:fcs:q}
\yqi(\Xii)&:=\E\Bigg[  \dfrac{g(\XiN)}{ (1+|\Xii|^2)^{q/2}} \\
\nonumber&\qquad\quad+\sum_{j=i}^{N-1}\dfrac{f_j\left(\Xij,\yqjp(\Xijp)(1+|\Xijp|^2)^{q/2}\right) }{(1+|\Xii|^2)^{q/2}}\Delta
\mid \Xii\Bigg].
\end{align}
The introduction of the polynomial factor $(1+|\Xii|^2)^{q/2}$ gives higher flexibility in the error analysis, it 
ensures that $\yqi$ decreases faster at infinity, which will provide nicer estimates on the approximation error when dealing with Fourier-basis, see Section \ref{sec:approError}.

\subsubsection{Randomizing the initial value $\Xii$}
For the subsequent algorithm, $X_i$ will be sampled randomly according to the distribution $\nu$, which we list the properties below. 
\begin{description}
\item [\Hnu] The distribution $\nu(\dxx)$ of $X_i^i$ has a density with respect to the Lebesgue measure on $\R^d$,  of the form
\begin{equation}
\nu(\mb{x}) = \prod_{l=1}^d \nu_{l}(x_{l}), \quad \mb{x}=(x_1,\dots,x_d)\in \R^d,
\label{eq:hnu}
\end{equation}
for some probability density functions $\nu_{l}$, which are assumed to be strictly positive.\\
The function $\mb{x}\in \R^d\mapsto \dfrac{L^\star(\mb{x})}{ (1+|\mb{x}|^2)^{q/2}}$ (where $L^\star$ is defined in \eqref{eq:prop:bound}) is square-integrable on $\R^d$ with respect to $\nu(\mb{x})\dd\mb{x}$.\\ 
We denote by $\Fnul$ the cumulative distribution function (CDF) of the one-dimensional marginal of $\nu_{l}$, $$\Fnul(x) := \int_{-\infty}^x \nu_{l}(x')\dx',$$ and by $\Fnulinv{l}$ its inverse {(which is standardly defined since $\Fnul$ is continuous and strictly increasing).}
\end{description}
Note that the above integrability condition ensures that $\yqi$ (for any $i$) is square integrable on $\R^d$ with respect to $\nu(\mb{x})\dd\mb{x}$.\\
In this work, we deal essentially with one type of distribution.
%
%

\begin{example}[Student's t-distribution]\label{example:student}
We deal with the non-standardized Student's t-distribution on $\R^d$ with parameter $\mu>0$ defined by its probability density function of the form \eqref{eq:hnu} with
\begin{align}
\nu_{l}(x_{l})=\dfrac{c_\mu}{(1+x_l^2)^{(\mu+1)/2}}, \label{eq:studentDensity}
\end{align}
where $c_\mu \in\R^+$.
This is the density of $ X/\sqrt{\mu}$ where $X$ is a standard Student's t-distribution with  $\mu$ degrees of freedom. It leads to numerous ways to simulate $\nu_l$, as acceptance rejection or ratio-of-uniforms methods, for example.
From \cite[p.~507]{jackman:2009} we know that 
$$c_\mu=\dfrac{\Gamma((\mu+1)/2)\sqrt \mu}{\Gamma(\mu/2)\sqrt{\pi \mu}}=\dfrac{\Gamma((\mu+1)/2)}{\Gamma(\mu/2)\sqrt{\pi }},
$$
where $\Gamma$ is the Gamma function defined in \cite[p.~501]{jackman:2009}.
Note that the marginal CDF and its inverse are not explicit in general for the Student's t-distribution, however their values (for usual parameter $\mu$) are tabulated in most of books in statistics and computable using standard software libraries. Moreover, a priori estimates on $\Fnul$, $\Fnulinv{l}$ and $\dnFnulinvapp{n}{l}$, which will be used in the error analysis of the algorithm, are computed in the Appendix \ref{appendix:estimatesStudent}.

Note that the square integrability of $\mb{x}\in \R^d\mapsto \dfrac{L^\star(\mb{x})}{ (1+|\mb{x}|^2)^{q/2}}$ (see Assumption \Hnu) is guaranteed by the condition $\mu  > \eta_g\lor \eta_f-q$.

\end{example}

The distribution we have considered so far has mean zero, this is suitable for problems where 0 has a natural interpretation as the center in the application at hand. In order to solve the problem far away from the neighborhood of zero we should  allow non zero means (say $\bar x\in \R^d$): this is easy to handle by shifting the above distribution by $\bar x$. For instance, in the case of the Student's t-distribution (Example \ref{example:student}), we would take
\begin{equation}
\nu_{l}(x_{l})=c_\mu (1+(x_l-\bar x_l)^2)^{-(\mu+1)/2}. \label{eq:studentDensity2}
\end{equation}
Since this centering adjustment is quite obvious, we mainly discuss the case of Example \ref{example:student} in what follows.\medskip

More importantly, the probability measure $\nu(\dd \mb{x})$ is expected to satisfy another assumption stated below.

\begin{description}
\item [$\bf H_{\rho}(q)$] $\exists c_\nu>0$ (depending on $q$, $b$, $\sigma$, $\nu$, $T$) such that  for any measurable function $\gamma:\R^d\rightarrow \R$ with $\gamma\in L^2_{\nu}(\R^d)$, any $N\geq j\geq i\geq 0$, we have \begin{equation}
 \Esp{\dfrac{|\gamma(\Xij)|^2(1+|\Xij|^2)^{q}}{(1+|\Xii|^2)^{q} }}\leq c_\nu \Esp{|\gamma(X^i_i)|^2}.\label{eq:prop:uses}
 \end{equation}
\end{description}
This is a stability property about expectation of weighted functionals of  the Markov chain $(X^{i}_j)_{i\le j\le N}$. This property will be crucial for analyzing the propagation of error in the quasi-regression scheme (see the proof in Section \ref{section:error}).  We now illustrate this property in relation with previous Example \ref{example:student}, the proof is postponed to Section \ref{subsection Proof of Theorem theo:uses}.

\begin{theorem}[Norm-stability] \label{theo:uses}
Assume that $X$ is  the Euler scheme \eqref{alg:markov:euler} with time step $\Delta$ of  a SDE \eqref{eq:edsWithoutZ} which drift coefficient $b:[0,T]\times \R^d\mapsto \R^d$ is bounded measurable, and which diffusion coefficient $\sigma:[0,T]\times \R^d\mapsto \R^d\otimes\R^q $ is bounded measurable and $\eta$-H\"older continuous in space uniformly in time (for some $\eta\in (0,1])$, 
$$\sup_{0\leq t \leq T}\sup_{\mb{x}\neq \mb{x'}}\frac{|\sigma(t,\mb{x})-\sigma(t,\mb{x'})|}{|\mb{x}-\mb{x'}|^\eta }<+\infty.$$
In addition, we assume that $\sigma$ fulfills a uniform ellipticity condition: 
$$\inf_{0\leq t \leq T, \mb{x}\in \R^d}\ \inf_{\zeta\in \R^d s.t. |\zeta|=1}\zeta.\sigma\sigma^\top(t,\mb{x})\zeta >0.$$
Let   $X^i_i$ be sampled from distribution $\nu$ in the case of {Student's t}-distribution (Example \ref{example:student}): then $\bf H_{\rho}(q)$ holds for any $q\geq 0$.
 \end{theorem}

\subsection{Approximation spaces and basis functions}
The aim of this section is to define some proper basis functions which satisfy orthogonality properties in $\R^d$ and which span some $L^2$ space, see Proposition \ref{proposition:base}. It turns out that the choice of measure for defining the $L^2$ space has to coincide with the sampling measure of $\Xii$ defined in \Hnu: this plays an important role for designing Algorithm \ref{alg:grmdp}. Our strategy for defining such basis functions is to start from trigonometric basis on $[0,1]^d$ and then to apply appropriate changes of variable: later, this transform will allow to easily quantify the approximation error when truncating the basis, see Propositions \ref{proposition:approError1} and \ref{proposition:approErrord}.

\paragraph{Notation.}
We start by setting notations to deal with appropriate $L^2$ space on $\R^d$ or $[0,1]^d$.
\begin{enumerate}[(i)]
 \item Let $\nu:\R^d\rightarrow \R$ be a probability density satisfying \Hnu. We will consider the space of measurable functions 
 \begin{align*}
 \LNU &:= \left\{ v: \R^d \rightarrow \R \mbox{ such that } \displaystyle\int_{\R^d} v^2 \mbox{ } \nu \dxx < \infty \right\},\\
 \LU &:= \left\{ w: [0,1]^d \rightarrow \R \mbox{ such that } \displaystyle\int_{[0,1]^d} w^2 \mbox{ } \duu < \infty \right\}.
 \end{align*}
 In these spaces we define the inner products,
\begin{align*}(v_1,v_2)_{\LNU} &= \displaystyle \int_{\R^d} v_1 v_2 \mbox{ } \nu \dxx, \quad \forall v_1,v_2 \in \LNU,\\(w_1,w_2)_{\LU} &= \displaystyle \int_{[0,1]^d} w_1 w_2 \mbox{ } \duu, \quad \forall w_1,w_2 \in \LU,
\intertext{and the norms,}
\norm{v}_{\LNU} &= \sqrt{(v,v)_{\LNU}} = \left(\displaystyle\int_{\R^d} v^2\mbox{ } \nu \dxx \right)^{\frac{1}{2}}, \quad \forall v \in \LNU,\\
\norm{w}_{\LU} &= \sqrt{(w,w)_{\LU}} = \left(\displaystyle\int_{[0,1]^d} w^2\mbox{ } \duu \right)^{\frac{1}{2}}, \quad \forall w \in \LU.
\end{align*}
\end{enumerate}

Before the general multidimensional case, we start with an easy computation in dimension 1.

\begin{lemma}[Orthonormal basis functions in dimension 1]  \label{lemma:ortho:dim1}Assume that $\nu$ satisfies \Hnu\ in dimension $d=1$.
Let $x\in\R$ and $k\in\N$. The basis functions
\begin{equation} \label{eq:nu_basis_functions}
\phi_k(x):= \left\{ \begin{array}{lcc}
             1, &   \mbox{if}  & k = 0, \\
             \\ \sqrt{2} \cos\left(k \pi \Fnuo(x)\right), &  \mbox{if}  & k > 0,
             \end{array}
   \right. 
\end{equation}
are orthonormal in $\R$ with respect to the probability density weight $\nu=\nu_{1}$.
\end{lemma}
\begin{proof}
 Let $i,j\in\N$. It is sufficient to see that
 \begin{align*}
  \displaystyle\int_{-\infty}^{\infty} \cos\left(i \pi  \Fnuo(x)\right) \cos\left(j \pi  \Fnuo(x)\right) \nu_{1}(x) \dx = \displaystyle\int_{0}^{1} \cos(i \pi u) \cos(j \pi u) \du,
\end{align*}
where we have used the change of variable $u = \Fnuo(x)$.
Clearly the cosine basis functions
\begin{equation} \label{eq:one_basis_functions}
C_k(u)= \left\{ \begin{array}{lcc}
             1, &   \mbox{if}  & k = 0, \\
             \\ \sqrt{2} \cos(k \pi u), &  \mbox{if}  & k > 0,
             \end{array}
   \right. 
\end{equation}
are orthonormal in $[0,1]$ with respect to the Lebesgue measure, therefore $(\phi_i,\phi_j)_{L^2_\nu(\R)}=\1_{i=j}$.\qed \end{proof}
Note that $(C_k \circ \Fnuo)(x)  = \phi_k(x)$ and $(\phi_k \circ \Fnuo^{-1})(u) = C_k(u)$. This change of variable will be useful in order to analyze the truncation error of $\nu_{1}$ series in terms of the truncation error of Fourier cosine series (see Section \ref{sec:approError}). 

Now that the reader gets acquainted with the one-dimensional case, we move to the more general result for any dimension $d\geq 1$. First, the natural way to build an orthonormal multidimensional set of basis functions is to take tensor products of orthonormal one-dimensional basis functions. To this end we need to add the superscript $(l)$ to the notations of the basis functions in dimension 1. Therefore, from now on, $\phi^{(l)}_{k_l}$ and $C^{(l)}_{k_l}$ stand for the one-dimensional basis functions \eqref{eq:nu_basis_functions} and \eqref{eq:one_basis_functions} in the  $l$-th coordinate, respectively. For any multiindex $\mb{k} = (k_1,\ldots,k_d)$, set
\begin{align}\label{eq:phi:dim:d}
\phi_{\mb{k}}(\mb{x})&:= \prod_{l=1}^{d} \phi_{k_l}^{(l)}(x_l), \quad  \mb{x} := (x_1,\ldots,x_d)\in \R^d,\\
C_{\mb{k}}(\mb{u})&:= \prod_{l=1}^{d} C_{k_l}^{(l)}(u_l), \quad  \mb{u} := (u_1,\ldots,u_d)\in [0,1]^d.
\label{eq:C:dim:d}
\end{align}
The transformation function $F_{\nu}: \R^d \rightarrow [0,1]^d$ and its inverse are defined by
$$F_{\nu}(\mb{x}) := (\Fnuo(x_1),\ldots,\Fnuo(x_d)), \qquad F_{\nu}^{-1}(\mb{u}) := (\Fnuoinv(u_1),\ldots,\Fnudinv(u_d)).$$ Therefore, from the above and from \eqref{eq:phi:dim:d}-\eqref{eq:C:dim:d}, 
we readily check the relations
\begin{align}\label{eq:relation:1}
(C_{\mb{k}} \circ F_{\nu})(\mb{x})&:= \displaystyle\prod_{l=1}^{d} (C_{k_l}^{(l)} \circ F_{\nu_l})(x_l)= \phi_{\mb{k}}(\mb{x}),\\
(\phi_{\mb{k}} \circ F^{-1}_{\nu})(\mb{u})&:= \displaystyle\prod_{l=1}^{d} (\phi_{k_l}^{(l)} \circ F^{-1}_{\nu_l})(u_l)= C_{\mb{k}}(\mb{u}).\label{eq:relation:2}
\end{align}
To analyze orthogonality properties, we  repeatedly use the following lemma, which proof is quite straightforward and  left to the reader (use the same change-of-variable arguments as in the proof of Lemma \ref{lemma:ortho:dim1}).
\begin{lemma} \label{lemma:eqInnerNorms} Assume $\nu$ satisfies \Hnu. Then, for all functions $v_1,v_2 \in \LNU$, we have the following identities
$$(v_1,v_2)_{\LNU} = (v_1\circ F_{\nu}^{-1},v_2\circ F_{\nu}^{-1})_{\LU},\qquad\norm{v_1}_{\LNU} = \norm{v_1 \circ F_{\nu}^{-1}}_{\LU}.$$
\end{lemma}

As a consequence of Lemma \ref{lemma:eqInnerNorms} and Equality \eqref{eq:relation:2}, since the tensored cosine functions $(C_{\mb{k}}(.):\mb{k}\in \N^d)$ are orthonormal in $[0,1]^d$ with respect to the Lebesgue measure, it implies that  $(\phi_{\mb{k}}(.):\mb{k}\in \N^d)$ are orthonormal in $\R^d$ with respect to the probability measure $\nu(\dd \mb{x})$. Moreover, the cosine functions $(C_{\mb{k}}(.):\mb{k}\in \N^d)$ form a complete family in $\LU$: this is a standard result in Fourier analysis (observe that here, the sine functions are not needed to get a complete basis since the scalar product $(.,.)_{\LU}$ is defined on $[0,1]^d$ and not $[-1,1]^d$). We summarize our findings in a proposition ready to be used.
\begin{proposition} \label{proposition:base}Assume $\nu$ satisfies \Hnu. The family of functions $(\phi_{\mb{k}}(.):\mb{k}\in \N^d)$ forms a complete orthonomal basis of $\R^d$ with respect to the measure  $\nu$.
\end{proposition}

\subsection{Global Quasi-Regression algorithm and convergence results} 
We now describe our main algorithm which is able to compute accurately the function $\yq_i$ defined in \eqref{eq:MDP:fcs:q}. For this, define
\begin{align} \label{eq:Sqi:M}
 \Sq_i(x^i_{i:N}) &:=  \frac{g(x^i_N)}{(1+|x^i_i|^2)^{q/2} } + \sum_{j=i}^{N-1} \frac{f_j\left(x^i_j, \yq_{j+1}(x^i_{j+1})(1+|x^i_{j+1}|^2)^{q/2}\right)}{(1+|x_i^i|^2)^{q/2} }\Delta,\\
 x^i_{i:N} & := (x^i_i,\dots,x^i_N) \in (\R^d)^{N-i+1}, \nonumber
\end{align}
which allows us to rewrite the exact solution as
$$\yq_i(\mb{x}) = \E\left[ \Resq{} | X_i^i = \mb{x}\right], \quad \mbx \in \R^d.$$
Note that it is easy to justify that $\Sq_i(X^i_{i:N})$ is square-integrable under \HFG, \Hnu\ and $\bf H_{\rho}(q)$ (remind of the bound \eqref{eq:prop:bound} and the integrability condition in \Hnu). Therefore,  the exact solution $\yq_i$ is in $\LNU$, and consequently (Proposition \ref{proposition:base}) it can be decomposed in
its $\phi$ series: 
\begin{align}
\label{yq:decomp:base}\yq_i(\mb{x}) = \sum_{\mb{k}\in\N^d} \Aq \phi_{\mb{k}}(\mb{x})
\end{align}
 for some coefficients $(\Aq:\mb{k}\in\N^d)$.  In practice, the coefficients $\Aq$ are  computed only in a finite set of multi-indices $\Gamma \subset \N^d$. Different choices of $\Gamma$ are discussed later in Section \ref{subsubsection:SelectionOfTheMulti-indices}.

 The orthonormality property of the $\phi_{\mb{k}}$'s (Proposition \ref{proposition:base}) combined with \eqref{yq:decomp:base} yields 
\begin{equation}
\label{def:Aq}
\Aq = (\yq_i, \phi_{\mb{k}})_{\LNU}.
\end{equation}
Thus, using the tower property of conditional expectation, we obtain, $\forall \mb{k} \in \N^d$,
 \begin{align}
\Aq   &= \Esp{\yq_i(X_i^i) \phi_{\mb{k}}(X_i^i)} = \Esp{ \Esp{\Resq{} | X_i^i } \phi_{\mb{k}}(X_i^i)}\nonumber\\
&=\Esp{ \Resq{} \phi_{\mb{k}}(X_i^i)}. \quad \label{eq:Aqi}
 \end{align}
The above expectations in \eqref{eq:Aqi} can be computed using Monte-Carlo simulations and evaluations of $\Sq_i$ along different sampled paths $x^i_{i:N}$. Since  $\Sq_i$ depends on the unknown solution, we have to replace it by approximations, using previous steps in the dynamic programming equation \eqref{eq:MDP:fcs:q}. It gives rise to the following Algorithm \ref{alg:grmdp}, which we will call \emph{Global Quasi-Regression Multistep-forward Dynamical Programming} (GQRMDP) algorithm. To maintain good integrability properties along the iteration, we use  the truncation function $\cT_{L^\star}$ \eqref{eq:TL:2} with the upper bound \eqref{eq:prop:bound}. 
\begin{algorithm}[GQRMDP algorithm]
\label{alg:grmdp}
$ $
\begin{description}
\item [\bf Initialization.]
Set $\bar{y}^{(q,M)}_{N}(x_{N}) := \frac{g(x_{N})}{(1+|x_{N}|^2)^{q/2} }$.

\item [\bf Backward iteration for $i=N-1$ to $i=0$,] 
\begin{equation}
\label{eq:yqi:bar:M} \byqiM(\cdot) := \sum_{\mb{k}\in\Gamma} \bar{\alpha}_{i,\mb{k}}^{(q,M)} \phi_{\mb{k}}(\cdot),
\end{equation}
where for all $\mb{k}\in\Gamma$,
\begin{equation} \label{eq:Aqi:bar:M}
 \AqbM := \dfrac{1}{M}\sum_{m=1}^M \SqM_{i}(X^{i,m}_{i:N})\phi_{\mb{k}}(X_i^{i,m}),
\end{equation}
and
\begin{align} \label{eq:Sqi:bar:M}
\SqM_{i}(x^i_{i:N}) & := \frac{g(x^i_{N})}{(1+|x^i_{i}|^2)^{q/2} } 
\\&\qquad+ \sum_{j=i}^{N-1} 
\dfrac{f_j\left(x_j^i,\cT_{L^\star} \left(\bar{y}^{(q,M)}_{j+1}(x^i_{j+1})(1+|x^i_{j+1}|^2)^{q/2}\right)\right)}{(1+|x^i_{i}|^2)^{q/2}}\Delta.
 \nonumber
\end{align}
\end{description}
\end{algorithm}
In the algorithm, in order to compute $\byqiM(\cdot)$ we have considered $M$ independent copies of 
$$X_i^i \xrightarrow{Euler\ scheme} X_{i+1}^i \xrightarrow{Euler\ scheme} \ldots \xrightarrow{Euler\ scheme} X_{N}^i,$$
which form a cloud of simulations $\cloud_{i} = \left\{ X_{i:N}^{i,m}: m=1,\ldots,M \right\}.$ We assume that the clouds of simulations $\left(\cloud_{i}: 0\leq i < N\right)$ are independently generated.

Our main result is the following error propagation theorem, which proof is postponed to Section \ref{section:error}. 
\begin{theorem}\label{th:main} Assume \HFG, \HBS, \Hnu\ and $\bf H_{\rho}(q)$. Let $\mb{k}\in \N^d$ be a multi-index and let $\Gamma \subset \N^d$ be a finite set of multi-indices. Let $(\phi_{\mb{k}}(.):\mb{k}\in \N^d)$ be the basis functions defined in \eqref{eq:relation:1}, as an orthonormal basis in $\LNU$, and let $\LG$ be the related Christoffel number \cite{neva:86} restricted to $\Gamma$:
\begin{equation}
\label{eq:christoffel}\LG = \sum_{\mb{k}\in\Gamma} \norm{\phi_{\mb{k}}(\cdot)}^2_\infty\leq 2^{d}\#\Gamma.
\end{equation}
For each $i\in\{0,\ldots,N-1\}$, define the following local error term
\begin{equation}
\label{eq:cEi}\cE_{i}:=\sum_{\mb{k}\in (\N^d - \Gamma)} (\Aq)^2 + \dfrac{\LG}{M} \E\left[ \left( \Sq_i(\Xi_{i:N}) \right)^2\right].\end{equation}
There is a finite positive constant $C$ (depending on model parameters but not on $N, M, \Gamma$)
 such that
\begin{align*}
\Esp{ \norm{ \byqiM(\cdot) - \yqi(\cdot) }_\nu^2 } \leq C\left( \cE_i +  \sum_{j=i}^{N-1} \cE_{j+1}\Delta\right).
\end{align*}
\end{theorem}
Observe that provided that $\dfrac{\LG}{M}$ remains bounded, the constant $C$ can be taken as uniform w.r.t. the algorithm parameters. This is the situation in practice since for getting accurate approximations, we need $\dfrac{\LG}{M}$ to be small. 

The convergence rate of the term $\displaystyle\sum_{\mb{k} \in (\N^d-\Gamma)}(\Aq)^2$ is discussed in Section \ref{sec:approError}  considering $\nu$ as the Student's t-distribution (Example \ref{example:student}). This analysis will be undertaken assuming that $y_i(\mbx)$ has a polynomial growth in $\mbx$ at infinity (coherently with the a priori estimate of Proposition \ref{prop:bound}) as well as its derivatives.

\subsection{Truncation error for the Student-cosine basis}\label{sec:approError}
From Equations \eqref{yq:decomp:base}-\eqref{def:Aq} and Lemma \ref{lemma:eqInnerNorms} we recall that  $\yqi = \displaystyle\sum_{\mb{k}\in \N^d} \Aq \phi_{\mb{k}},$ where 
\begin{align}
\label{eq:fourier}
\Aq = (\yqi,\phi_{\mb{k}})_{\LNU} = (\yqi\circ F_{\nu}^{-1},C_{\mb{k}})_{\LU}.
\end{align}
Define the truncation of the $\nu$-series of $\yqi$ by $$P_\Gamma \yqi := \displaystyle \sum_{\mb{k}\in \Gamma} \Aq \phi_{\mb{k}}.$$
Therefore, from orthonormal properties of the basis functions, we have
$$\norm{\yqi - P_\Gamma \yqi }^2_{\LNU} = \norm{\yqi \circ F_{\nu}^{-1} - \sum_{\mb{k}\in \Gamma} \Aq C_{\mb{k}} }^2_{\LU} = \sum_{\mb{k}\in (\N^d - \Gamma)} |\Aq|^2.$$ 
In other words, $\Aq$ are the Fourier coefficients of the projection of 
\begin{equation}
\label{eq:hqi}\hqi := \yqi \circ F_\nu^{-1} \in \LU
\end{equation} 
over cosine basis functions. In the following section we will take advantage of this property to derive  a tight bound for the truncation error $\sum_{\mb{k}\in (\N^d - \Gamma)} |\Aq|^2$, using assumptions on $\yqi$ and its derivatives only (and not on $\yqi \circ F_{\nu}^{-1}$). The specific choice of the Student's t-distribution measure $\nu$ of Example \ref{example:student} plays an important role to get tractable conditions.

Now our aim is twofold. Firstly we prove that if  $y_i$ is $r$ times ($r\geq 1$) continuously differentiable on $\R^d$, with (for some $ p\in \N$),
\begin{align*}
\limsup_{|\mbx| \to + \infty} \dfrac{|\pdyi{n}{l}(\mbx)|}{|\mb x|^{p-n}}<+\infty,
\quad \text{for} \qquad n=0,\ldots,r, \quad l=1,\dots, d,
\end{align*}
and if the parameters are such that 
\begin{equation*} 
p+\left(r-\frac{1}{2}\right)\mu < q,
\end{equation*}    
then the truncation error is controlled as
\begin{align*}
\sum_{\mbk \in \N^d: \exists k_l > K_l} |\Aq|^2& \leq \sum_{l=1}^d \dfrac{C}{{(K_l+1)}^{2r}}
\end{align*}
for some constant $C>0$. In the following, we first establish the above result in dimension 1 (Proposition \ref{proposition:approError1}) and then in dimension $d$  (Proposition \ref{proposition:approErrord}) .

Our second goal is improving this result when mixed derivatives of $y_i$ are also supposed smooth. Thus, we prove that if $y_i$ is $\Lambda$-smooth, where $\Lambda(r) = \{0,\dots,r\}^d$, with (for some $p \in \N$),
\begin{align*}
\limsup_{|\mbx| \to + \infty} \dfrac{|{\partial^{\mb{n}}_x} y_i(\mbx)|}{|\mb x|^{p-\overline{\mb{n}}}}<+\infty,
\quad \text{for} \quad \mb{n}\in {\Lambda},
\end{align*}
and if the parameters are such that $$p+\left(d\ r-\dfrac{1}{2}\right)\mu<q,$$ 
then the truncation error satisfies the tail inequality
\begin{align*}
\sum_{\mbk \in \bar{\Gamma}_H(\DEG)} |\Aq|^2& \leq  \dfrac{C}{{(\DEG+1)^{2r}}},
\end{align*}
for some constant $C>0$ (Proposition \ref{proposition:approErrord:Multivariate}).

%

\subsubsection{One-dimensional case}
In this section, we let $\nu=\nu_1$ be the Student density \eqref{eq:studentDensity}.

\begin{lemma} \label{lemma:lr_lims_zero}
Assume that $y_i$ is $r$ times ($r\in \N$) continuously differentiable on $\R$, with the existence of a polynomial growth degree $p\in\N$ such that 
\begin{equation}
\limsup_{x\to \pm \infty} \dfrac{|\dyi{n}(x)|}{|x|^{p-n}}<+\infty, \quad \text{for} \quad n=0,\ldots,r.
\label{eq:poly:growth:dyi}
\end{equation}
Assume the relation
\begin{equation} 
\label{eq:lr_lims_zero} 
p+r\mu<q 
\end{equation}    
between $r$, $p$, the Student parameter $\mu>0$ and the damping exponent $q\geq0$. Then $\hqi$ is $r$ times continuously differentiable on $(0,1)$,  and its  first $r$ derivatives vanish at $u=0$ and $u=1$:
$$\lim_{u\to0^+,1^-}\ddhqi{n}(u) = 0, \quad n=0,\ldots,r.$$
\end{lemma}

\begin{proof}
The regularities of $\yqi$ and $\hqi$ are straightforward in view of \eqref{eq:yqi}-\eqref{eq:hqi} and Lemma \ref{lemma:dFnulinv:app}. The estimate $\lim_{u\to0^+,1^-}\hqi(u) = 0$ is clearly true since it is equivalent to $0=\lim_{x\to\pm\infty}\yqi(x)$, which holds owing to the condition $q>p$.

From now on, we may consider the case $n \geq 1$ and $r\geq 1$. We start showing that for $1\leq n\leq r$,
\begin{equation}
\ddyqi{n}(x)=
O(x^{p-n-q}) \mbox{ as } x\to\pm\infty.\label{eq:dyqi}
\end{equation}
In fact, using Leibniz rule, $$\ddyqi{n}(x) = \sum_{m=0}^n {n\choose m} \mbox{ } \dyi{n-m}(x) \mbox{ }{\rm d}^m_x \dfrac{1}{(1+x^2)^{q/2}}.$$ Besides,
\begin{align*}
\dyi{n-m}(x) &= 
O(x^{p-n+m}), \qquad
{\rm d}^m_x \dfrac{1}{(1+x^2)^{q/2}} = O(x^{-q-m}), 
\qquad\mbox{ as } x\to \pm \infty, 
\end{align*}
which leads directly to the desired result \eqref{eq:dyqi}. 

We now estimate $\ddhqi{n}(u)$. In order to compute the $n$-\textit{th} derivative of $\hqi$ we take advantage of the following Fa\`a di Bruno's formula \cite[pp.~224--226]{Johnson02thecurious}, 
\begin{equation} \label{eq:1dFaaDiBruno}
\begin{split}
\ddhqi{n}(u) &= {\rm d}_u^n \yqi(\Fnuoinv(u)) \\
&= \sum_{m=1}^n \frac1{m!}\ddyqi{m}(x)\big|_{x=\Fnuoinv(u)} \sum_{\mb{j}\in J_{m,n}} \dfrac{n!}{j_1!j_2!\cdots j_m!} \prod_{i=1}^m \dnFnudinv{j_i},\\
J_{m,n}&=\{\mb{j}=(j_1,\ldots,j_m)\in \N^m_+:j_1+\ldots+j_m = n\}.
\end{split}
\end{equation}
As $u\to 0^+$, using also \eqref{eq:dFnulinv:app}, 
\begin{align}\nonumber
 \ddhqi{n}(u) &= O\left(\max_{1\leq m \leq n,\ \mb{j}\in J_{m,n}} \left|\left( \Fnuoinv(u)\right)^{p-m-q} \dnFnudinv{j_1}\cdots \dnFnudinv{j_m} \right|\right) \\
\nonumber           &= O\left( \max_{1\leq m \leq n,\ \mb{j}\in J_{m,n}}
u^{-\frac{p-m-q}{\mu}} u^{-\frac{1+j_1\mu}{\mu}} \cdots u^{-\frac{1+j_m\mu}{\mu}} \right) 
           \\
           &= O\left( u^{-\frac{p+n\mu-q}{\mu}} \right).\label{eq:limit:ddhqi:0}
\end{align}
The above bound  goes to 0 as $u\to 0^+$ since  ${p}+n\mu-q\leq  {p}+r\mu -q< 0$.\\
Analogously, as $u\to 1^-$, we have
\begin{align}
 \ddhqi{n}(u) &= O\left( (1-u)^{-\frac{p+n\mu-q}{\mu}} \right)\label{eq:limit:ddhqi:1}
\end{align}
which converges to 0 as  $u\to 1^-$ since  $p+n\mu < q$.
\qed \end{proof}

\begin{lemma}\label{lemma:lr_norms_bounded}
Assume the hypotheses of Lemma \ref{lemma:lr_lims_zero} and replace the condition \eqref{eq:lr_lims_zero} by the weaker condition
\begin{equation} 
\label{eq:lr_lims_zero:weaker} p+\left(r-\frac{1}{2}\right)\mu < q.
\end{equation}    
Then the $r$ first derivatives of $\hqi$ are square integrable on $(0,1)$:
$$\norm{\ddhqi{n}}^2_{L^2_U([0,1])} < +\infty, \quad n=0,\ldots,r.$$
\end{lemma}

\begin{proof}
The integrability of $\ddhqi{n}$ on any compact set of the form $[\varepsilon,1-\varepsilon]$ ($\varepsilon>0$) is obvious owing to the continuity of the function to integrate. As a consequence one has only to consider what happens around 0 and 1. 
As we have just seen in the proof of Lemma \ref{lemma:lr_lims_zero} (see \eqref{eq:limit:ddhqi:0}-\eqref{eq:limit:ddhqi:1}), 
$$\ddhqi{n}(u) =              O_{u\to0^+}\left(u^{-\frac{p+n\mu-q}{\mu}}\right), \qquad
\ddhqi{n}(u) =              O_{u\to1^-}\left((1-u)^{-\frac{p+n\mu-q}{\mu}}\right);$$
therefore it is square-integrable around $u=0$ and $u=1$ provided that \\\mbox{$-\frac{p+n\mu-q}{\mu}>-\frac{1}{2 }$}. Since $n\leq r$ the condition \eqref{eq:lr_lims_zero:weaker} is sufficient.
\qed \end{proof}

\begin{proposition} \label{proposition:approError1} 
Assume that $y_i$ is $r$ times ($r \geq 1$) continuously differentiable on $\R$, with the existence of a polynomial growth degree $p\in\N$ such that \eqref{eq:poly:growth:dyi} holds. Assume the relation \eqref{eq:lr_lims_zero:weaker} holds. Then, the coefficients  $(\Aqo,k\in\N)$
defined in \eqref{eq:fourier} satisfy the tail inequality (for any $K\geq 0$)
\begin{equation}\displaystyle\sum_{k=K+1}^\infty |\Aqo|^2 \leq \dfrac{C^{(r)}_{\eqref{eq:tail:inequality:coeff:dim1}}}{{(K+1)}^{2r}},
\label{eq:tail:inequality:coeff:dim1}
\quad\text{where}\quad C^{(r)}_{\eqref{eq:tail:inequality:coeff:dim1}} := \dfrac{1}{\pi^{2r}}\norm{\ddhqi{r}}^2_{L^2_U([0,1])}<+\infty.
\end{equation}
\end{proposition}

\begin{proof} 
Under \eqref{eq:lr_lims_zero:weaker}, we can invoke Lemma \ref{lemma:lr_lims_zero} (with $r-1$ derivatives) and Lemma \ref{lemma:lr_norms_bounded} (with $r$ derivatives), to get
\begin{equation}\label{eq:controle:derivee}
\begin{split}
 & \lim_{u\to0^+,1^-}\ddhqi{n}(u) = 0, \quad \text{for } n=0,\ldots,r-1,\\
 & \norm{\ddhqi{n}}^2_{L^2_U([0,1])} <+ \infty,\quad \text{for } n=0,\ldots,r.
\end{split}
\end{equation}
The  coefficients $(\Aqo,k\in\N)$  coincide with the Fourier coefficients of the decomposition of $\hqi$ with respect to the orthonormal basis functions $\{C_k,k\in\N\}$ in $L^2_{U}([0,1])$ (see the relations \eqref{eq:fourier}-\eqref{eq:hqi}). The $C_k$-cosine series of $\hqi$ writes
$$\hqi(u) = \displaystyle\sum_{k=0}^\infty \Aqo C_k(u) = \alpha_{i,0}^{(q)} + \sum_{k=1}^\infty \Aqo C_k(u),$$ where 
\begin{align*}
\alpha_{i,0}^{(q)} = \displaystyle\int_0^1 \hqi(u) \du\quad\text{ and }\quad \Aqo = \sqrt{2}\displaystyle\int_0^1 \hqi(u) \cos(k \pi u)\du \text{ for }
 k>0.\end{align*}
Applying integration by parts ($\hqi$ is continuously differentiable) and using that \\$\lim_{u\to 0^+,1^-} \hqi(u) = 0$ we get (for $k\neq 0$)
\begin{align} \label{eq:oneDimTrunErrorIntParts1}
\Aqo &= 
\Re{\sqrt{2} \lim_{\varepsilon\rightarrow 0} \int_{\varepsilon}^{1-\varepsilon} \hqi(u) e^{\mb{i} k \pi u}\du } \nonumber \\
&=\Re{\dfrac{\sqrt{2}}{\mb{i} k\pi} \left[\hqi(u)e^{\mb{i} k \pi u} \right]_{0^+}^{1^-} - \dfrac{\sqrt{2}}{\mb{i} k\pi}\displaystyle\int_0^1 \dhqi(u)e^{\mb{i} k \pi u}\du} \nonumber \\
&=\Re{\dfrac{\sqrt{2}}{-\mb{i}k\pi}\displaystyle\int_0^1 \dhqi(u)e^{\mb{i} k \pi u}\du}.
\end{align}
Thanks to \eqref{eq:controle:derivee}, we can repeat $r-1$ times the integration by parts as above, and we obtain
\begin{align} \label{eq:oneDimTrunErrorIntPartsr}
\Aqo  
&=\Re{\dfrac{\sqrt{2}}{(-\mb{i}k\pi)^r}\displaystyle\int_0^1 \ddhqi{r}(u)e^{\mb{i} k \pi u}\du} \nonumber \\
& =\dfrac{1}{(k\pi)^r}
\displaystyle\int_0^1 \ddhqi{r}(u)\Re{\sqrt{2}\ \mb{i}^re^{\mb{i} k \pi u}}\du.
\end{align}
Observe that the sets  $$\{u\mapsto\sqrt 2 \cos(\pi k u),k\geq 1 \},$$
 and 
 $$\{u\mapsto\sqrt 2 \sin(\pi k u),k\geq 1 \},$$
are made of orthonormal functions in $L^2_U([0,1])$. Consequently, for any value of $r$, $\{\Re{\sqrt{2}\ \mb{i}^re^{\mb{i} k \pi u}},k\geq 1 \}$ is a set of orthonormal functions and the equality \eqref{eq:oneDimTrunErrorIntPartsr} means that $\{(k\pi)^r\Aqo:k\geq 1\}$ are the projection coefficients of $\pm \ddhqi{r}$ on these orthonormal functions. As such, the Bessel inequality gives
\begin{align*}
\sum_{k=1}^{\infty}(k\pi )^{2r} |\Aqo|^2\leq \norm{\ddhqi{r}}^2_{L^2_U([0,1])}=\pi^{2r}C^{{(r)}}_{\eqref{eq:tail:inequality:coeff:dim1}}.
\end{align*}
It readily follows 
$$\sum_{k=K+1}^\infty |\Aqo|^2 \leq \dfrac{1}{\pi^{2r}}\sum_{k=K+1}^\infty \dfrac{(k\pi)^{2r}}{(K+1)^{2r}}|\Aqo|^2 \leq \dfrac{C^{{(r)}}_{\eqref{eq:tail:inequality:coeff:dim1}} }{(K+1)^{2r}}.$$
\qed \end{proof}

\subsubsection{Multi-indices sets} \label{subsubsection:SelectionOfTheMulti-indices}
In this work only arbitrary downward closed multi-indices sets are considered. {As already stated in the introduction}, the finite multi-indices set $\Gamma\subset \N^d$ is downward closed if $$(\mbk\in\Gamma \mbox{ and } \mb{j} \leq \mbk) \Rightarrow \mb{j} \in\Gamma,$$
where $\mb{j} \leq \mbk$ means that $j_l\leq k_l$ for all $l=1,\ldots,d$. The most well-known of such multi-indices sets is the full or tensor product multi-indices set. For any $\mb{K}=(K_1,\ldots,K_d)$ where $K_l\in \N$ for all $l=1,\ldots,d$, the full anisotropic multi-indices set and its complement are defined by 
\begin{align*}
 \Gamma_F(K_1,\ldots,K_d) &:= \{\mbk \in \N^d: k_l \leq K_l, \quad \forall l=1\ldots d \}, \\
 \bar{\Gamma}_F(K_1,\ldots,K_d) &:= \{ \mbk\in\N^d : \mbk\notin \Gamma_F(K_1,\ldots,K_d)\}.
\end{align*}

In high dimensions the full multi-indices set suffers from the curse of dimensionality. In order to effectively compute the projection of $\yqi$ over a multidimensional set of of basis functions we propose to consider the following two sparse sets of downward closed multi-indices sets, where $\DEG$ denotes a nonnegative integer:
\begin{itemize}
 \item Total degree set defined by $$\Gamma_{T}(\DEG) := \left\{\mb{k} \in \mathbb{N}^d: \displaystyle\sum_{l=1}^{d} k_l\leq \DEG \right\},$$ where $\#\Gamma_{T}(\DEG) = \binom{\DEG + d}{d}$.
 \item Hyperbolic cross index set defined by
$$\Gamma_{H}(\DEG) := \left\{\mb{k} \in \mathbb{N}^d: \prod_{l=1}^d \max(k_l,1) \leq \DEG \right\}.$$
In order to compute the number of elements in $\Gamma_{H}(\DEG)$ we split this set in the following two disjoint sets:
\begin{itemize}
 \item $\Gamma^d_1(\DEG):= \left\{\mb{k} \in \mathbb{N}^d, \forall l, k_l \neq 0: \prod_{l=1}^d k_l \leq \DEG \right\},$
 \item $\Gamma_2(\DEG):= \left\{\mb{k} \in \mathbb{N}^d, \exists l, k_l = 0: \prod_{l=1}^d \max(k_l,1) \leq \DEG \right\}.$
\end{itemize}
On the one hand, let us compute $\#\Gamma^d_1(\DEG)$. Firstly we define $$\Upsilon^d(\DEG) := \left\{\mb{k} \in \mathbb{N}^d, \forall l, k_l \neq 0: \prod_{l=1}^d k_l = \DEG \right\}.$$ Secondly we divide $\Gamma^d_1(\DEG)$ in the following disjoint sets, $$\Gamma^d_1(\DEG) = \bigcup\limits_{g=2}^{\DEG} \Upsilon^d(g) \cup \{ \mb{1}\}.$$ For {any integer $g\in \{ 2,\dots, \DEG\}$ having the prime} factorization $$g = \prod_i p_i^{v_i},$$ $\#\Upsilon^d(g)$ is equal to the number of possible ways to write $g$ as a product of exactly $d$ of its factors, $$\#\Upsilon^d(g) = \prod_i \binom{v_i + d - 1}{d-1}.$$ 
Therefore, $$\#\Gamma^d_1(\DEG) = \sum_{g=2}^{\DEG} \#\Upsilon^d(g) + 1.$$

On the other hand, let us compute $\#\Gamma_2(\DEG)$. Firstly we define $\varUpsilon_c(\DEG)$ as the subset of $\Gamma_2(\DEG)$ containing the multi-indices of $\Gamma_2(\DEG)$ with exactly $c$ of the $d$ indexes equal to zero, $0<c<d$. Note that $\#\varUpsilon_c(\DEG) = {d \choose c} \#\Gamma^{d-c}_1(\DEG)$. Therefore, since  $$\Gamma_2(\DEG) = \bigcup\limits_{c=1}^{d-1} \varUpsilon_c(\DEG) \cup \{ \mb{0}\},$$ $$\#\Gamma_2(\DEG) = \sum_{c=1}^{d-1} \binom{d}{c} \#\Gamma^{d-c}_1(\DEG) + 1.$$

Finally, we {conclude using} that $\#\Gamma_{H}(\DEG) = \#\Gamma^d_1(\DEG) + \#\Gamma_2(\DEG).$
\end{itemize}

Note that $\forall g \in \N, g \leq (\DEG-d+1)_+, \Gamma_H(g) \subset \Gamma_T(\DEG)$.
See Figure \ref{fig:totalHyperbolic} and Table \ref{tab:totalHyperbolic}.

\begin{center}
\begin{figure}[!htb]
  \subfigure{\includegraphics[height=3.0cm]{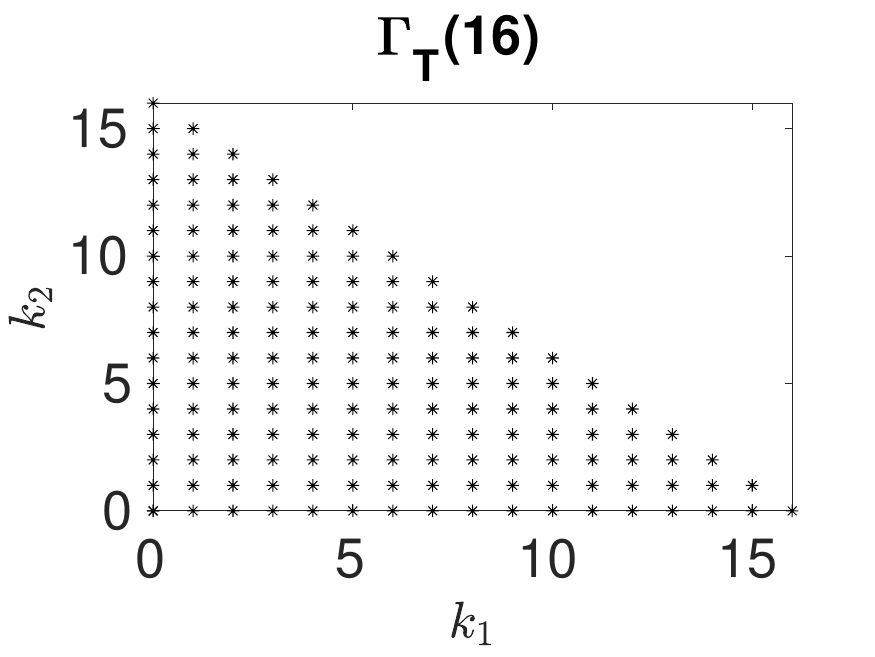}}
  \subfigure{\includegraphics[height=3.0cm]{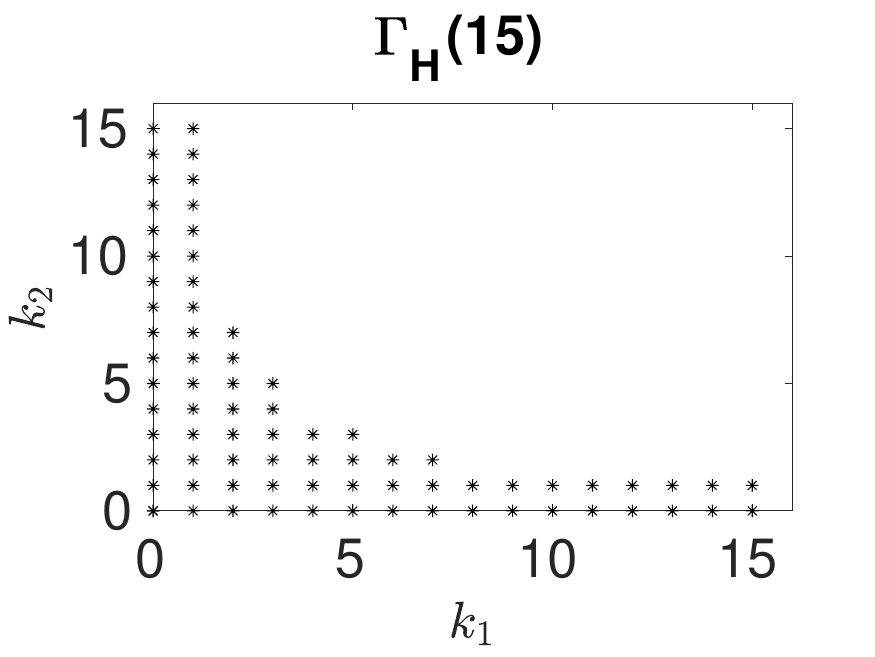}}
  \subfigure{\includegraphics[height=3.0cm]{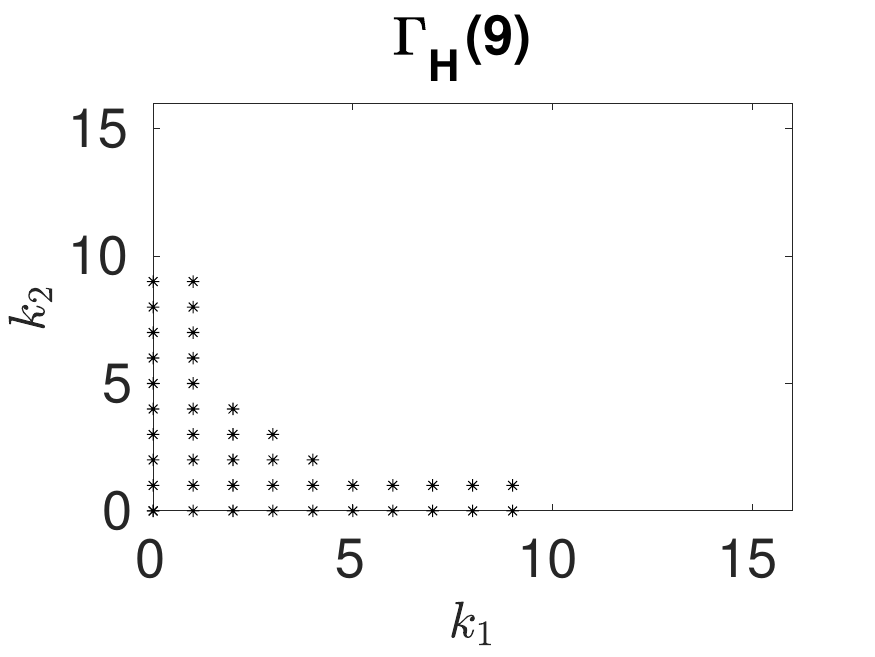}}
  \caption{For $d=2$ examples of total and hyperbolic sets. Note that $\Gamma_H(9) \subset \Gamma_H(15) \subset \Gamma_T(16)$.}
  \label{fig:totalHyperbolic}
\end{figure}

\begin{table}[ht]
\caption{Cardinalities of $\Gamma_T$ and $\Gamma_H$.}
\label{tab:totalHyperbolic}
\begin{center}\begin{tabular}{|c|c|c||c|c|c|}
    \hline
    \multicolumn{3}{|c||}{$d=3$} & \multicolumn{3}{|c|}{$d=4$} \\
    \hline
    $\DEG$ & $\#\Gamma_T(\DEG)$ & $\#\Gamma_H(\DEG-2)$ & $\DEG$ & $\#\Gamma_T(\DEG)$ & $\#\Gamma_H(\DEG-3)$ \\
    \hline    
    $6$ & $84$ & $50$ & $5$ & $126$ & $48$ \\
    \hline
    $8$ & $165$ & $86$ & $7$ & $330$ & $136$ \\
    \hline
    $10$ & $286$ & $123$ & $9$ & $715$ & $248$ \\
    \hline
    $12$ & $455$ & $165$ & $11$ & $1365$ & $368$ \\
    \hline        
\end{tabular}
\end{center}
\end{table}
\end{center}

\subsubsection{Multidimensional case} \label{sec:approError:Multidimensional}
In this section, we let $\nu$ be the Student density \eqref{eq:hnu}-\eqref{eq:studentDensity} in dimension $d$. In the first part of the section only $y_i$ and its coordinate-wise derivatives are assumed to be smooth. In the second part, on top of that mixed derivatives of $y_i$ will be supposed smooth. Let us start with the first case.
\begin{lemma} \label{lemma:lr_lims_zero:d}
Assume that $y_i$ is $r$ times ($r\in\N$) coordinate-wise continuously differentiable on $\R^d$, with the existence of a polynomial growth degree $p\in\N$ such that 
\begin{align}
\limsup_{|\mbx| \to + \infty} \dfrac{|\pdyi{n}{l}(\mbx)|}{|\mb x|^{p-n}}<+\infty,
\quad \text{for} \qquad n=0,\ldots,r, \quad l=1,\dots, d.
\label{eq:poly:growth:dyi:d}
\end{align}
Assume the relation 
\begin{equation}\label{eq:lr_lims_zero:d}   p+r\mu < q \end{equation}                                 
between $r$, $p$, the Student parameter $\mu>0$ and the damping exponent $q\geq 0$. Then $\hqi$ is $r$ times continuously differentiable on $(0,1)^d$, and its first $r$ coordinate-wise derivatives vanish at $\partial[0,1]^d$:
$$\lim_{\mbu\to\partial[0,1]^d} \pdhqi{n}{u_l}(\mbu) = 0, \qquad n=0,\ldots,r, \quad  l=1,\ldots,d.$$
\end{lemma}

\begin{proof} As in dimension 1, the regularities of $\yqi$ and $\hqi$ are straightforward to obtain. 
The statement $\lim_{\mbu\to\partial[0,1]^d} \hqi(\mbu) = 0$ is clear since  the limit coincides with that of 
$\yqi(\mbx)$ as $|\mbx|\to +\infty$, which equals to 0 in view of  $q>p$. Therefore, from now on, we may consider $n\geq 1$ and let $l=1,\dots,d$ be fixed.\\ 
In the following, to deal with comparison of functions  as $|\mbx|$ tends to $+ \infty$, 
we use the same notation $C$ for any generic constants, that do not depend on $\mbx$, but that may depend on the order $n,r,\dots$ of derivatives under consideration. We will keep the same $C$ from line to line, in order the alleviate the computations, although its value changes.

We start by showing that for $n=1,\dots, r$, 
\begin{equation}\label{eq:dyqi:d} |\pddyqi{n}{l}(\mbx)| \leq C (1+|\mbx|^2)^{p/2-n/2-q/2}.\end{equation}
By applying the Leibniz rule to the product \eqref{eq:yqi} (as in dimension 1, since we differentiate $n$ times with respect to the same variable $x_l$) we get 
\begin{align}
\label{eq:leibniz}\pddyqi{n}{l}(\mbx) = \sum_{m=0}^n {n\choose m} \mbox{ } \pdyi{n-m}{l}(\mbx) \mbox{ }\partial^m_{x_l} \dfrac{1}{(1+|\mbx|^2)^{q/2}}.
\end{align}
Our standing assumptions write
\begin{align}
\label{eq:standing:assumptions}
|\pdyi{n-m}{l}(\mbx)| \leq C(1+|\mbx|)^{p-(n-m)}.
\end{align}
From \eqref{eq:1overDumpingFactor:app} we know that
$$\left|\partial^m_{x_l} \dfrac{1}{(1+|\mbx|^2)^{q/2}}\right|\leq C( 1+|\mbx|^2)^{-q/2-m/2}.$$
By plugging the previous upper bound and \eqref{eq:standing:assumptions} into the Leibniz formula \eqref{eq:leibniz}, we get
\begin{align*}
\left|\pddyqi{n}{l}(\mbx)\right| &\leq  C\max_{0\leq m \leq n}(1+|\mbx|)^{p-(n-m)}(1+|\mbx|^2)^{-q/2-m/2} \\
&\leq C (1+|\mbx|^2)^{p/2-q/2-n/2},
\end{align*}
we are done with \eqref{eq:dyqi:d}.


In order to compute the $n$-\textit{th} derivative of $\hqi(\mbu)= \yqi(\Fnulinv{1}(u_1),\ldots,\Fnulinv{d}(u_d))$ we again take advantage of the Fa\`a di Bruno's formula (as in \eqref{eq:1dFaaDiBruno})
$$
\pddlhqi{n}{l}(\mbu) = \sum_{m=1}^n \frac1{m!}\pddyqi{m}{l}(\mbx)\big|_{\mbx=(\Fnulinv{1}(u_1),\ldots,\Fnulinv{d}(u_d))} \sum_{\mb{j} \in J_{m,n} }\dfrac{n!}{j_1!j_2!\cdots j_m!}  \prod_{i=1}^m \dnFnulinv{j_i}{l}.
$$
Combining the above with Lemma \ref{lemma:Fnui:Fnulinv:app}  and the estimates \eqref{eq:dFnulinv:app}, we obtain
\begin{align}
&\left| \pddlhqi{n}{l}(\mbu)\right| \leq C \max_{1\leq m \leq n,\ \mb{j}\in J_{m,n}} \left(1+\sum_{i=1}^d| \Fnulinv{i}(u_i)|^2\right)^{p/2-q/2-m/2}
\prod_{i=1}^m \left|\dnFnulinv{j_i}{l}\right|  \nonumber \\
           &
           \leq C \max_{1\leq m \leq n,\ \mb{j}\in J_{m,n}} \left(1+\sum_{i=1}^d \frac{1}{u_i^{1/\mu}(1-u_i)^{1/\mu} }\right)^{p-q-m}
\left(\frac{1}{u_l^{1/\mu}(1-u_l)^{1/\mu} }\right)^{m+\mu\sum_{i=1}^m j_i}\nonumber\\
&           \leq C \max_{1\leq m \leq n} \left(1+\sum_{i=1}^d \frac{1}{u_i^{1/\mu}(1-u_i)^{1/\mu} }\right)^{p-q-m}
\left(1+\sum_{i=1}^d \frac{1}{u_i^{1/\mu}(1-u_i)^{1/\mu} }\right)^{m+n\mu}\nonumber\\
&           \leq C  \left(1+\sum_{i=1}^d \frac{1}{u_i^{1/\mu}(1-u_i)^{1/\mu} }\right)^{p-q+n\mu}.
\label{eq:limit:ddhqi:0:d}
\end{align}
Since $p+n\mu<q$, we get that $\lim_{\mbu \to \partial [0,1]^d} \pddlhqi{n}{l}(\mbu)=0$.
\qed \end{proof}

\begin{lemma}\label{lemma:lr_norms_bounded:d}
Assume the hypotheses of Lemma \ref{lemma:lr_lims_zero:d} and replace the condition \eqref{eq:lr_lims_zero:d} by the weaker condition
\begin{equation}
\label{eq:lr_lims_zero:weaker:d} 
p+\left(r-\frac{1}{2}\right) \mu< q.
\end{equation}
Then the $r$ first coordinate-wise derivatives of $\hqi$ are square integrable on $(0,1)^d$:
$$\norm{\pdlhqi{n}{l}(\mbu)}^2_{\LU} < +\infty, \qquad n=0,\ldots,r, \quad  l=1,\ldots,d.$$
\end{lemma}

\begin{proof}
In view of the bound \eqref{eq:limit:ddhqi:0:d}, it is sufficient to check the  integrability on $(0,1)^d$ of 
$$l(\mbu):=\left(1+\sum_{i=1}^d \frac{1}{u_i^{1/\mu}(1-u_i)^{1/\mu} }\right)^{2(p-q+r\mu)}.$$
Observe that the function is symmetric with respect to each hyperplan $\{\mbu\in(0,1)^d:u_l=\frac{1}{2 }\}$, $l=1,\dots, d$. Therefore, only the integrability on $(0,1/2)^d$ has to be established, which is equivalent to prove 
$$\int_{(0,1/2)^d} \left(1+\sum_{i=1}^d \frac{1}{u_i^{1/\mu} }\right)^{2(p-q+r\mu)} \duu<+\infty$$
Use the change of variables $z_i=u_i^{-1/{2\mu}}$, then the above integral writes (up to a non-zero factor) 
$${\cal I}:=\int_{(2^{1/(2\mu)},+\infty)^d} \left(1+\sum_{i=1}^d z^2_i \right)^{\mu-\varepsilon} \left(\prod_{i=1}^d z^{-2\mu-1}_i \right)\dd {\bf z},$$
where  $\varepsilon>0$ is given by $2(p-q+r\mu)=\mu-\varepsilon$. To show that ${\cal I}<+\infty
$, we may decrease $\varepsilon\downarrow 0$ as much as needed to get  $\mu-\varepsilon>0$. Once done, using  the inequality 
$$(x_1+x_2)^\alpha \leq 2^{(\alpha-1)_+}(x_1^\alpha+x_2^\alpha),$$
available for any $x_1\geq 0,x_2\geq 0,\alpha>0$, it yields  (by direct induction) the existence of a constant $C_{\alpha,d}$ such that 
$$\left(1+\sum_{i=1}^d x_i\right)^\alpha \leq C_{\alpha,d} \left(1+\sum_{i=1}^d x_i^\alpha\right)$$
for any $x_1\geq 0,\dots, x_d\geq 0$.  Consequently,
$${\cal I}\leq C_{\mu-\varepsilon,d}\int_{(2^{1/(2\mu)},+\infty)^d} \left(1+\sum_{i=1}^d z_i^{2(\mu-\varepsilon)}\right) \left(\prod_{i=1}^d z^{-2\mu-1}_i \right)\dd {\bf z}<+\infty.$$ We are done.
\qed \end{proof}

\begin{proposition} \label{proposition:approErrord} 
Assume that $y_i$ is $r$ times ($r\geq 1$) coordinate-wise continuously differentiable on $\R^d$ with the existence of a polynomial growth degree $p \in \N$ such that \eqref{eq:poly:growth:dyi:d} holds. Assume the relation \eqref{eq:lr_lims_zero:weaker:d} holds. Then, the coefficients $(\Aq, \mbk\in\N^d)$ defined in \eqref{eq:fourier} satisfy the tail inequality 
\begin{align}\sum_{\mbk \in \bar{\Gamma}_F(K_1,\ldots,K_d)} |\Aq|^2& \leq \sum_{l=1}^d \dfrac{C^{(l,r)}_{\eqref{eq:tail:inequality:coeff:dimd}}}{{(K_l+1)}^{2r}}
\label{eq:tail:inequality:coeff:dimd}
\end{align}
where $\displaystyle C^{(l,r)}_{\eqref{eq:tail:inequality:coeff:dimd}} := \dfrac{1}{\pi^{2r}}\norm{\pdlhqi{r}{l}}^2_{\LU}<+\infty.$
 \end{proposition}
\begin{proof} By definition of ${\Gamma}_F(K_1,\ldots,K_d)$, we have the union bound
$$
\{\mbk \in \bar{\Gamma}_F(K_1,\ldots,K_d)\}=\bigcup_{l=1}^d\{\mbk \in \N^d:k_l>K_l\},$$
therefore the advertised inequality \eqref{eq:tail:inequality:coeff:dimd} directly follows from 
\begin{align}
\label{eq:eq:tail:inequality:coeff:dimd:coordonnee1}
\sum_{\mbk \in \N^d:k_l>K_l} |\Aq|^2 \leq \frac{C^{(l,r)}_{\eqref{eq:tail:inequality:coeff:dimd}}}{(K_l+1)^{2r}}, \qquad l=1,\dots,d.
\end{align}
We only prove the case $l=1$, the arguments apply similarly for other $l$.

Under \eqref{eq:lr_lims_zero:weaker:d}, we can invoke Lemma \ref{lemma:lr_lims_zero:d} (with $r-1$ derivatives) and Lemma \ref{lemma:lr_norms_bounded:d} (with $r$ derivatives), to get
\begin{equation}\label{eq:controle:derivee:d}
\begin{split}
 & \lim_{\mbu\to\partial[0,1]^d}\pdlhqi{n}{1}(\mbu) = 0, \quad \text{for } n=0,\ldots,r-1,\\
 & \norm{\pdlhqi{n}{1}}^2_{\LU} < +\infty,\quad \text{for } n=0,\ldots,r.
\end{split}
\end{equation}
In view of \eqref{eq:fourier}-\eqref{eq:hqi}, the Fourier coefficients of the decomposition of $\hqi$ with respect to the orthonormal basis functions $\{C_{\mbk}, \mbk \in\N^d\}$ in $\LU$ are given by $(\Aq,\mbk\in\N^d)$n and the $C_{\mbk}$-cosine series of $\hqi$ writes $$\hqi(u) = \sum_{\mbk \in \N^d} \Aq C_{\mbk}(u),$$ where
\begin{align*}
\Aq &=\beta_{\mbk}\int_{[0,1]^d} \hqi(\mbu) \cos(k_1\pi u_1) \cos(k_2\pi u_2) \cdots \cos(k_d\pi u_d) \duu \\
    & = \beta_{\mbk}\Re{\int_{[0,1]^d} \hqi(\mbu) e^{\mb{i}k_1\pi u_1} \cos(k_2\pi u_2) \cdots \cos(k_d\pi u_d) \duu},\\
  \beta_{\mbk}&:= \prod_{l=1}^d (\1_{k_l=0}+\1_{k_l>0}\sqrt 2 ).
\end{align*}
Applying integration by parts in $u_1$ ($\hqi$ is uniformly bounded on $(0,1)^d$, continuously differentiable and the derivative with respect to $u_1$ is  integrable) and using that $\lim_{u_1\to 0^+,1^-}\hqi(\mbu) = 0$, we get for $k_1> 0$
\begin{align*} 
\Aq 
 & = {\rm{Re}}\Bigg( \dfrac{\beta_{\mbk}}{\mb{i}k_1\pi} \left[ \int_{[0,1]^{d-1}} \hqi(\mbu) e^{\mb{i}k_1\pi u_1} \cos(k_2\pi u_2) \cdots \cos(k_d\pi u_d)  \du_2 \ldots \du_d \right]^{u_1=1^-}_{u_1=0^+}   \\
 & \qquad \qquad - \dfrac{\beta_{\mbk}}{\mb{i}k_1\pi} \int_{[0,1]^d} \pdlhqi{}{1}(\mbu) e^{\mb{i}k_1\pi u_1} \cos(k_2\pi u_2) \cdots \cos(k_d\pi u_d) \duu \Bigg)  \\
& =\Re{\dfrac{\beta_{\mbk}}{-\mb{i}k_1\pi}\int_{[0,1]^d} \pdlhqi{}{1}(\mbu) e^{\mb{i}k_1\pi u_1} \cos(k_2\pi u_2) \cdots \cos(k_d\pi u_d) \duu}.
\end{align*}
Thanks to \eqref{eq:controle:derivee:d}, we can repeat $r-1$ times the integration by parts as above, and we obtain
\begin{align} \label{eq:multiDimTrunErrorIntPartsr}
\Aq  
&=\Re{\dfrac{\beta_{\mbk}}{(-\mb{i}k_1\pi)^r}\displaystyle\int_{[0,1]^d} \pdlhqi{r}{1}(\mbu) e^{\mb{i}k_1\pi u_1} \cos(k_2\pi u_2) \cdots \cos(k_d\pi u_d) \duu}  \nonumber \\
& =\dfrac{1}{(k_1\pi)^r}
\displaystyle\int_{[0,1]^d} \pdlhqi{r}{1}(\mbu)\Re{\beta_{\mbk} \mb{i}^r  e^{\mb{i}k_1\pi u_1} \cos(k_2\pi u_2) \cdots \cos(k_d\pi u_d)}\duu.
\end{align}
Observe that the sets  
\begin{align*}
 & \{\mbu\mapsto \beta_{\mbk} \cos(k_1 \pi u_1)\cos(k_2 \pi u_2)\cdots \cos(k_d \pi u_d),\mbk \in \N^d \}, \\
 & \{\mbu\mapsto \beta_{\mbk} \sin(k_1 \pi u_1)\cos(k_2 \pi u_2)\cdots \cos(k_d \pi u_d),\mbk \in \N^d \}, 
\end{align*}
are made of orthonormal functions in $\LU$. Consequently, $$\left\{\Re{\beta_{\mbk} \mb{i}^r e^{\mb{i}k_1\pi u_1} \cos(k_2\pi u_2) \cdots \cos(k_d\pi u_d)},\mbk \in \N^d \right\}$$ is (for any value of $r$) a set of orthonormal functions and the equality \eqref{eq:multiDimTrunErrorIntPartsr} means that $\{(k_1\pi)^r\Aq:\mbk \in \N^d\}$ are the projection coefficients of $\pm\pdlhqi{r}{1}$ on these orthonormal functions. As such, the Bessel inequality gives
\begin{align*}
\sum_{\mbk \in \N^d}(k_1\pi )^{2r} |\Aq|^2\leq \norm{\pdlhqi{r}{1}}^2_{\LU}=\pi^{2r}C^{{(1,r)}}_{\eqref{eq:tail:inequality:coeff:dimd}}.
\end{align*}
It readily follows
$$\sum_{\mbk \in \N^d: k_1>K_1} |\Aq|^2 \leq\frac{1}{\pi^{2r}} \sum_{\mbk \in \N^d: k_1>K_1} \dfrac{(k_1\pi)^{2r}}{(K_1+1)^{2r}}|\Aq|^2 \leq \dfrac{C^{{(1,r)}}_{\eqref{eq:tail:inequality:coeff:dimd}}}{(K_1+1)^{2r}}.$$
The proof of \eqref{eq:eq:tail:inequality:coeff:dimd:coordonnee1} is now finished and Proposition \ref{proposition:approErrord} is proved. \qed
\end{proof}

Now we consider that not only $y_i$ and its coordinate-wise derivatives are smooth but also its mixed derivatives.
\begin{lemma} \label{lemma:lr_lims_zero:d:Multivariate}
Let ${\Lambda}\subset\N^d$ be an arbitrary downward closed multi-index set {and set $r =\max_{\mb{r}\in {\Lambda}}(\overline{\mb{r}})$.}
Assume that $y_i$ is ${\Lambda}$-smooth with the existence of a polynomial growth degree $p\in\N$ such that 
\begin{align}
\limsup_{|\mbx| \to + \infty} \dfrac{|{\partial^{\mb{n}}_x} y_i(\mbx)|}{|\mb x|^{p-\overline{\mb{n}}}}<+\infty,
\quad \text{for} \quad \mb{n}\in {\Lambda}.
\label{eq:poly:growth:dyi:d:Multivariate}
\end{align}
Assume the relation 
\begin{equation}\label{eq:lr_lims_zero:d:Multivariate}   p+r\mu < q \end{equation}                                 
between $r$, $p$, the Student parameter $\mu>0$ and the damping exponent $q\geq 0$. Then $\hqi$ is {${\Lambda}$-smooth on $(0,1)^d$ and its derivatives} vanish at $\partial[0,1]^d$:
$$\lim_{\mbu\to\partial[0,1]^d} {\partial^{\mb{n}}_u} \hqi(\mbu) = 0, \qquad \mb{n}\in {\Lambda}.$$
\end{lemma}

\begin{proof}
We start by showing that for $\overline{ \mb{n}}> 0$, $\mb{n}\in {\Lambda}$, 
\begin{equation}\label{eq:dyqi:d:Multivariate} |{\partial^{\mb{n}}_x}\yqi(\mbx)| \leq C (1+|\mbx|^2)^{p/2-\overline{\mb{n}}/2-q/2}.\end{equation}
By applying the multivariate Leibniz rule to the product \eqref{eq:yqi}  we get 
\begin{align}
\label{eq:leibnizMultivariate}{\partial^{\mb{n}}_x} \yqi(\mbx) = \sum_{\{\mb{m}\in\N^d: \mb{m}\leq \mb{n}\}} C_{\mb{n},\mb{m}} \mbox{ } {\partial^{\mb{n}-\mb{m}}_x} y_i(\mb{x})\ {\partial^{\mb{m}}_x} \dfrac{1}{(1+|\mbx|^2)^{q/2}}.
\end{align}
Our standing assumptions write
\begin{align}
\label{eq:standing:assumptions:Multivariate}
|{\partial^{\mb{n}-\mb{m}}_x} y_i(\mbx)| \leq C(1+|\mbx|)^{p-\overline{(\mb{n}-\mb{m})}} {\leq} C(1+|\mbx|^2)^{p/2-\overline{\mb{n}}/2+\overline{\mb{m}}/2}.
\end{align}
From \eqref{eq:1overDumpingFactorMultivariate:app} we know that
$$\left|{\partial^{\mb{m}}_x} \dfrac{1}{(1+|\mbx|^2)^{q/2}}\right|\leq C( 1+|\mbx|^2)^{-q/2-\overline{\mb{m}}/2}.$$
By plugging the previous upper bound and \eqref{eq:standing:assumptions:Multivariate} into the Leibniz formula \eqref{eq:leibnizMultivariate}, we get
\begin{align*}
\left|{\partial^{\mb{n}}_x}\yqi(\mbx)\right| &\leq  C\max_{\mb{0}\leq \mb{m}\leq \mb{n}}(1+|\mbx|^2)^{p/2-\overline{\mb{n}}/2+\overline{\mb{m}}/2}(1+|\mbx|^2)^{-q/2-\overline{\mb{m}}/2} \\
&\leq C (1+|\mbx|^2)^{p/2-q/2-\overline{\mb{n}}/2},
\end{align*}
we are done with \eqref{eq:dyqi:d:Multivariate}.

{
Let $[*,\ldots,*]$ denotes a list. In order to compute the $\mb{n}$ ($\overline{\mb{n}}>0$, $\mb{n}\in\Lambda$ ) derivative of $\hqi(\mbu)= \yqi(\Fnulinv{1}(u_1),\ldots,\Fnulinv{d}(u_d))$ we take advantage of the multivariate Fa\`a di Bruno's formula derived from \cite[Theorem 2.1]{constantine:savits:96}:
\begin{align*}
{\partial^{\mb{n}}_u} \hqi(\mbu) &= \sum_{\mb{m} \in M_{\mb{n}}} {\partial^{\mb{m}}_x}\yqi(\mbx)\big|_{\mbx=(\Fnulinv{1}(u_1),\ldots,\Fnulinv{d}(u_d))} \sum_{J \in A_{\mb{m},\mb{n}} } C_{\mb{m},J} \prod_{{i=1}}^d \prod_{k=1}^{{m_i}} \dnFnulinv{j_{ik}}{i}, \\
M_{\mb{n}} &= \{ \mb{m}= (m_1,\ldots,m_d)\in\N^d_+: \forall i=1,\ldots,d,  {0}\leq m_i\leq n_i,{\mb m\neq 0 } \}, \\
A_{\mb{m},\mb{n}} &= \{ J = [\mb{j_1}\in\N_+^{m_1},\ldots,\mb{j_d}\in\N_+^{m_d}] : {\overline{\mb{j_i}}=n_i\,\forall i=1,\ldots,d\}.} 
\end{align*}
Combining the above with Lemma \ref{lemma:Fnui:Fnulinv:app}  and the estimates \eqref{eq:dFnulinv:app}, we obtain
\begin{align}
&\left| {\partial^{\mb{n}}_u} \hqi(\mbu)\right| \leq C \max_{\mb{m}\in M_{\mb{n}},\ J\in {A}_{\mb{m},\mb{n}}} \left(1+\sum_{i=1}^d| \Fnulinv{i}(u_i)|^2\right)^{p/2-q/2-\overline{\mb{m}}/2}
\prod_{i=1}^d\prod_{k=1}^{{m_i}} \left|\dnFnulinv{j_{ik}}{i}\right|  \nonumber \\
           &
           \leq C \max_{\mb{m}\in M_{\mb{n}},\ J\in {A}_{\mb{m},\mb{n}}} \left(1+\sum_{i=1}^d \frac{1}{u_i^{1/\mu}(1-u_i)^{1/\mu} }\right)^{p-q-\overline{\mb{m}}}
\prod_{i=1}^d \left(\frac{1}{u_i^{1/\mu}(1-u_i)^{1/\mu} }\right)^{m_i+\mu\sum_{l=1}^{m_i} j_{il}}\nonumber\\
&           \leq C \max_{{\mb m \neq 0}, \mb{m} \leq \mb{n}} \left(1+\sum_{i=1}^d \frac{1}{u_i^{1/\mu}(1-u_i)^{1/\mu} }\right)^{p-q-\overline{\mb{m}}}
\prod_{i=1}^d \left(\frac{1}{u_i^{1/\mu}(1-u_i)^{1/\mu} }\right)^{m_i+\mu n_i}\nonumber\\
&           \leq C \max_{{\mb m \neq 0}, \mb{m} \leq \mb{n}} \left(1+\sum_{i=1}^d \frac{1}{u_i^{1/\mu}(1-u_i)^{1/\mu} }\right)^{p-q-\overline{\mb{m}}}
\prod_{i=1}^d \left(1+\sum_{j=1}^d\frac{1}{u_j^{1/\mu}(1-u_j)^{1/\mu} }\right)^{m_i+\mu n_i}\nonumber\\
&           \leq C \max_{{\mb m \neq 0},\mb{m} \leq \mb{n}} \left(1+\sum_{i=1}^d \frac{1}{u_i^{1/\mu}(1-u_i)^{1/\mu} }\right)^{p-q-\overline{\mb{m}}+\overline{\mb{m}}+\mu\overline{\mb{n}}}\\
&           \leq C  \left(1+\sum_{i=1}^d \frac{1}{u_i^{1/\mu}(1-u_i)^{1/\mu} }\right)^{p-q+\overline{\mb{n}}\mu}.
\label{eq:limit:ddhqi:0:d:Multivariate}
\end{align} 
Since $p+\overline{\mb{n}}\mu<q$, we get that $\lim_{\mbu \to \partial [0,1]^d}{\partial^{\mb{n}}_u}\hqi(\mbu)=0$.
 \qed}
\end{proof}

\begin{lemma}\label{lemma:lr_norms_bounded:d:Multivariate}
Assume the hypotheses of Lemma \ref{lemma:lr_lims_zero:d:Multivariate} and replace the condition \eqref{eq:lr_lims_zero:d:Multivariate} by the weaker condition
\begin{equation}
\label{eq:lr_lims_zero:weaker:d:Multivariate} 
p+\left(r-\frac{1}{2}\right) \mu< q.
\end{equation}
Then the $\Lambda$ derivatives of $\hqi$ are square integrable on $(0,1)^d$:
$$\norm{\partial^{\mb{n}}_u\hqi(\mbu)}^2_{\LU} < +\infty, \qquad  \mb{n}\in\Lambda.$$
\end{lemma}

\begin{proof}
 The proof follows readily from the proof of Lemma \ref{lemma:lr_norms_bounded:d} by replacing $r$ by $r=\max_{\mb{r}\in\Lambda}(\overline{\mb{r}})$.
\end{proof}


\begin{proposition} \label{proposition:approErrord:Multivariate}
Let $\Lambda(r) = \{0,\dots,r\}^d$. Assume that $y_i$ is $\Lambda$-smooth with the existence of a polynomial growth degree $p \in \N$ such that \eqref{eq:poly:growth:dyi:d:Multivariate} holds. Assume the relation \eqref{eq:lr_lims_zero:weaker:d:Multivariate} holds, i.e. $$p+(d\ r-1/2)\mu<q.$$ Then, the coefficients $(\Aq, \mbk\in\N^d)$ defined in \eqref{eq:fourier} satisfy the tail inequality 
\begin{align}
\sum_{\mbk \in \bar{\Gamma}_H(\DEG)} |\Aq|^2& \leq  \dfrac{C_{\eqref{eq:tail:inequality:coeff:dimd:Multivariate}}}{{(\DEG+1)^{2r}}},
\label{eq:tail:inequality:coeff:dimd:Multivariate}
\end{align}
where $\displaystyle C_{\eqref{eq:tail:inequality:coeff:dimd:Multivariate}} := \dfrac{1}{\pi^{2dr}}\norm{\partial^{(r,\ldots,r)}_u\hqi}^2_{\LU}<+\infty.$
\end{proposition}

\begin{proof}
Under \eqref{eq:lr_lims_zero:weaker:d:Multivariate}, we can invoke Lemma \ref{lemma:lr_lims_zero:d:Multivariate} (with $\#\Lambda(r)-1$ derivatives) (condition $p+(dr-1)\mu<q$) and Lemma \ref{lemma:lr_norms_bounded:d:Multivariate} (with $\#\Lambda(r)$ derivatives) (condition $p+(dr-1/2)\mu<q$), to get
\begin{equation}\label{eq:controle:derivee}
\begin{split}
 & \lim_{\mbu\to\partial[0,1]^d} \partial^{\mb{n}}_u \hqi(\mbu) = 0, \qquad \mb{n}\in\Lambda(r),\, \mb{n}\neq(r,\ldots,r),\\
 & \norm{\partial_u^{\mb{n}}\hqi(\mbu)}^2_{\LU} < +\infty, \qquad  \mb{n}\in\Lambda(r).
\end{split}
\end{equation}
By doing integration by parts on the Fourier coefficients of $\hqi$ as we have already done in Proposition \ref{proposition:approErrord} we readily obtain 
 \begin{align} 
\nonumber  \Aq &= 
\displaystyle\int_{[0,1]^d} \hqi(\mbu) \beta_{\mbk} \cos(k_1\pi u_1) \cdots \cos(k_d\pi u_d)\duu,\\ 
 &= \dfrac{1}{\pi^{d r}(k_1 \cdots k_d)^r}\displaystyle\int_{[0,1]^d} \partial^{(r,\ldots,r)}_u\hqi(\mbu)\beta_{\mbk}\cos^{\int,r}(k_1\pi u_1) \cdots \cos^{\int,r}(k_d\pi u_d)\duu,\label{eq:multiDimTrunErrorIntPartsrMultivariate}
 \end{align}
where $ \cos^{\int,r}(.)$ is the $r$-th antiderivative of $\cos(\cdot)$ such that 
\begin{align*}
&\cos^{\int,4 m}(\cdot)=\cos(\cdot),&&  \cos^{\int,4 m+1}(\cdot)=\sin(\cdot),\\
&\cos^{\int,4 m+2}(\cdot)=-\cos(\cdot), &&
 \cos^{\int,4 m+3}(\cdot)=-\sin(\cdot),\quad m\in \N.
\end{align*}
Equation \eqref{eq:multiDimTrunErrorIntPartsrMultivariate} tells us that $\hqi$ is a function in a Korobov space \cite[Definition 6.5.6, p. 225]{krommerUeberhuber:94}.
Observe that the sets  
\begin{align*}
 & \{\mbu\mapsto \beta_{\mbk} \cos(k_1 \pi u_1)\cos(k_2 \pi u_2)\cdots \cos(k_d \pi u_d),\mbk \in \N^d \}, \\
 & \{\mbu\mapsto \beta_{\mbk} \sin(k_1 \pi u_1)\cos(k_2 \pi u_2)\cdots \cos(k_d \pi u_d),\mbk \in \N^d \},  \\
 &  \hspace{4.5cm}\vdots\\ 
 & \{\mbu\mapsto \beta_{\mbk} \sin(k_1 \pi u_1)\sin(k_2 \pi u_2)\cdots \sin(k_d \pi u_d),\mbk \in \N^d \},
\end{align*}
are made of orthonormal functions in $\LU$.
Consequently, $$\left\{\beta_{\mbk}\cos^{\int,r}(k_1\pi u_1) \cdots \cos^{\int,r}(k_d\pi u_d)
,\mbk \in \N^d \right\}$$ is (for any value of $r$) a set of orthonormal functions and the equality \eqref{eq:multiDimTrunErrorIntPartsrMultivariate} means that $\{(k_1\pi)^{r}\cdots (k_d\pi)^{r}\Aq:\mbk \in \N^d\}$ are the projection coefficients of $\partial^{(r,\ldots,r)}_u\hqi$ on these orthonormal functions. As such, the Bessel inequality gives
\begin{align*}
\sum_{\mbk \in \N^d}\pi^{2dr}(k_1\cdots k_d)^{2r}  |\Aq|^2\leq \norm{\partial^{(r,\ldots,r)}_u\hqi}^2_{\LU}=\pi^{2dr}C_{\eqref{eq:tail:inequality:coeff:dimd:Multivariate}}.
\end{align*}
It readily follows
\begin{align*}
\sum_{\substack{\mbk \in \N^d,\\ \mbk \notin \Gamma_H(\DEG)}} |\Aq|^2 &\leq\frac{1}{\pi^{2 d r}} \sum_{\substack{\mbk \in \N^d, \\ \mb{k}\notin \varGamma_H(\DEG)}} \dfrac{\pi^{2d r}(k_1\cdots k_d)^{2r}}{(\DEG+1)^{2r}}|\Aq|^2 \leq \frac{C_{\eqref{eq:tail:inequality:coeff:dimd:Multivariate}}}{(\DEG+1)^{2r}}.
\end{align*}
\qed
\end{proof}

Equation \eqref{eq:multiDimTrunErrorIntPartsrMultivariate} means that the most significant coefficients $\Aq$ are the ones for which the product $k_1\cdots k_d$ is as small as possible. Under this setting we prefer to consider the hyperbolic cross index set (Proposition \ref{proposition:approErrord:Multivariate}) rather than the total degree set (Proposition \ref{proposition:approErrord}). In fact, although the hyperbolic truncation error with $\Gamma_H((\DEG-d+1)_+)$ will be greater than the total truncation error with $\Gamma_T(\DEG)$, if $\hqi$ belongs to a Korobov space, the hyperbolic truncation error will be near to the total truncation error (here we are assuming that $\DEG$ is big enough to take into account all the significant Fourier coefficients). Besides, having in mind that $\#\Gamma_H(DEG-d+1) \ll \#\Gamma_T(DEG)$ (specially in high dimensions), the statistical error, driven by $\frac{2^d\#\Gamma}{M}$, will be smaller for the hyperbolic multi-indices set, considering the same number of simulations $M$. Last but not least, the algorithm executed with the hyperbolic multi-indices set will be less time consuming than with the total degree set. This behavior is shown in the numerical experiments of Section \ref{sec:numExperiments}.

 
 
\section{Proofs related to error analysis}

\subsection{Proof of Theorem \ref{theo:uses}}
\label{subsection Proof of Theorem theo:uses}
First observe that we can restrict in the proof of \eqref{eq:prop:uses} to the case of bounded measurable functions $\gamma$: indeed, for $\gamma\in L^2_{\nu}(\R^d)$ (possibly unbounded), apply the inequality \eqref{eq:prop:uses} to  the bounded function  $\mb{x} \mapsto \cT_L(\gamma(\mb{x}))$ where $L>0$ to get (for any $0\leq i \leq j\leq N$)
\begin{equation}
 \Esp{\dfrac{| \cT_L(\gamma(\Xij))|^2(1+|\Xij|^2)^{q}}{(1+|\Xii|^2)^{q} }}\leq c_\nu \Esp{|\cT_L(\gamma(X^i_i))|^2},
 \label{eq:prop:uses:3}
 \end{equation}
 and take the limit $L\to+\infty$. On the one hand, the dominated convergence theorem (because $\gamma\in L^2_\nu$) ensures that the r.h.s. of \eqref{eq:prop:uses:3} converges to $c_\nu \Esp{|\gamma(X^i_i)|^2}$. On the other hand, by Fatou's lemma, the liminf of the l.h.s. of  \eqref{eq:prop:uses:3}  is bounded from below by $\Esp{\dfrac{|\gamma(\Xij)|^2(1+|\Xij|^2)^{q}}{(1+|\Xii|^2)^{q} }}$, which all in all, gives the result \eqref{eq:prop:uses} for the square integrable function $\gamma$.

Now we establish \eqref{eq:prop:uses} for bounded functions $\gamma$. 
 We follow and adapt the arguments of \cite[Proposition 3.1 and Proposition 3.3 item (b)]{gobe:turk:17} which paves the way for $q=0$ (although the distribution $\nu$ is different). Under the current assumptions on $b$ and $\sigma$, the Euler scheme $X$ is a non-homogeneous Markov chain which transition kernel has a probability density w.r.t. the Lebesgue measure on $\R^d$. This density is   denoted by $p_{i,j}(\mb{x},\mb{y})$ where $i<j$ stands for the time indices and $\mb{x}$ (resp. $\mb{y}$) stands for the initial (resp. final) space point. This probability density satisfies to a Gaussian upper bound of Aronson type (see \cite[Theorem 2.1]{lema:meno:10}): there is a finite  constant $C_{\eqref{eq:lemma:aronson:xh}}> 0$ (depending  on $d$, $q$, $T$, $b$, $\sigma$ but not on $N$) such that
\begin{equation}
\label{eq:lemma:aronson:xh}
p_{i,j}(\mb{x},\mb{y})\leq C_{\eqref{eq:lemma:aronson:xh}} \frac{\exp\left({-\frac{|\mb{y}-\mb{x}|^2  }{2 C_{\eqref{eq:lemma:aronson:xh}} (t_j-t_i)  }}\right)  }{(2\pi (t_j-t_i))^{d/2} },
\end{equation}
for any $0\leq i <j \leq N$ and any $\mb{x},\mb{y}\in \R^{ d}$. Assume that $\nu$ is such that, for any $\Lambda\geq 0$, there is a finite constant $C_\nu(\Lambda)>0$ (depending on $\mu$ and $q$) such that, for  any  $\lambda\in [0,\Lambda]$ and any $\mb{y}\in \R^{{d}}$, we have
\begin{equation}
\label{eq:rho:1} 
\int_{\R^{{ d}}} \nu(\mb{y}+\mb{z}\sqrt \lambda) \dfrac{(1+|\mb{y}|^2)^q }{(1+|\mb{y}+\mb{z}\sqrt \lambda|^2)^q}\frac{e^{-|\mb{z}|^2/2}  }{(2\pi)^{{{ d}}/2} }\dd \mb{z} \leq C_\nu(\Lambda)  \nu(\mb{y}).
\end{equation}
Then we claim that the inequality \eqref{eq:prop:uses} holds, still considering bounded $\gamma$ (which clearly ensures that both sides of \eqref{eq:prop:uses} are finite). Indeed, combining \eqref{eq:lemma:aronson:xh} and \eqref{eq:rho:1}, we get
\begin{align*}
&\Esp{\dfrac{|\gamma(\Xij)|^2(1+|\Xij|^2)^{q}}{(1+|\Xii|^2)^{q} }}\\
  &= \int_{\R^{ d}}\int_{\R^{ d}}\nu(\mb{x})p_{i,j}(\mb{x},\mb{y}) \frac{(1+|\mb{y}|^2)^{q}}{(1+|\mb{x}|^2)^{q} } \gamma^2(\mb{y}) \dd\mb{x} \dd \mb{y}\\
 \intertext{(use \eqref{eq:lemma:aronson:xh} and the change of variable $(\mb{x},\mb{y})\to (\mb{y}+\mb{z}\sqrt{C_{\eqref{eq:lemma:aronson:xh}}(t_j-t_i)},\mb{y})$)}
   &\leq C_{\eqref{eq:lemma:aronson:xh}}\int_{\R^{ d}}\int_{\R^{ d}}\nu(\mb{x}) \frac{e^{-\frac{|\mb{y}-\mb{x}|^2  }{2 C_{\eqref{eq:lemma:aronson:xh}} (t_j-t_i)  }}  }{(2\pi (t_j-t_i))^{d/2} } \frac{(1+|\mb{y}|^2)^{q}}{(1+|\mb{x}|^2)^{q} } \gamma^2(\mb{y}) \dd\mb{x} \dd \mb{y}\\
&= C_{\eqref{eq:lemma:aronson:xh}}^{1+d/2}
\int_{\R^{ d}}\int_{\R^{ d}} \nu\left(\mb{y}+\mb{z}  \sqrt {C_{\eqref{eq:lemma:aronson:xh}}(t_j-t_i)}\right) \\
& \hspace{3cm} \times\dfrac{(1+|\mb{y}|^2)^q }{(1+|\mb{y}+\mb{z}\sqrt {C_{\eqref{eq:lemma:aronson:xh}}(t_j-t_i)}|^2)^q} \frac{e^{-|\mb{z}|^2 /{ 2}}  }{(2\pi) ^{d/2} } \gamma^2(\mb{y})\dd\mb{z} \dd\mb{y}\\
&\leq  C^{1+d/2}_{\eqref{eq:lemma:aronson:xh}} C_\nu(C_{\eqref{eq:lemma:aronson:xh}} T) \int_{\R^{ d}}\gamma^2(\mb{y}) \nu(\mb{y})\dd \mb{y}
\end{align*}
which gives the announced result \eqref{eq:prop:uses} with $c_\nu :=C^{1+d/2}_{\eqref{eq:lemma:aronson:xh}} C_\nu(C_{\eqref{eq:lemma:aronson:xh}} T)$.

Therefore, it remains to establish \eqref{eq:rho:1}; since 
$$
(1+z^2)^{1/2}\leq (1+z)\leq \sqrt 2 (1+z^2)^{1/2}, \qquad \forall z\geq 0,
$$
for getting \eqref{eq:rho:1} it is enough to prove (for any $\mb{y}=(y_1,\dots,y_d)$)
\begin{align}
\nonumber
&\int_{\R^{{ d}}} \left(\prod_{l=1}^d 
\frac{1}{(1+|y_l+z_l \sqrt \lambda |)^{\mu+1}}\right)
 \dfrac{(1+|\mb{y}|^2)^q }{(1+|\mb{y}+\mb{z}\sqrt \lambda|^2)^q}\frac{e^{-|\mb{z}|^2/2}  }{(2\pi)^{{{ d}}/2} }\dd \mb{z}\\
   &\qquad\leq\dfrac{C_\nu(\Lambda)}{2^{(\mu+1)d/2}}  \prod_{l=1}^d \frac{1}{(1+|y_l|)^{\mu+1}}.\label{eq:rho:1:bis} 
\end{align}
Set ${\cal A}_\lambda:=\{\mb{z}\in \R^d:|\mb{z}|\sqrt \lambda\leq |\mb{y}|/2\}$. Denote by ${\cal I}_1$ (resp. ${\cal I}_2$) the  above integral restricted to ${\cal A}_\lambda$ (resp. ${\cal A}_\lambda^c$). On ${\cal A}_\lambda$, use that 
$1+|\mb{y}+\mb{z}\sqrt \lambda|^2\geq 1+\frac 12 |\mb{y}|^2-|\mb{z}\sqrt \lambda|^2\geq \frac{1}{4 }+\frac 14 |\mb{y}|^2$ to obtain 
\begin{align}
\nonumber
{\cal I}_1&\leq 4^q \int_{{\cal A}_\lambda} \left(\prod_{l=1}^d 
\frac{1}{(1+|y_l+z_l \sqrt \lambda |)^{\mu+1}}\right)
\frac{e^{-|\mb{z}|^2/2}  }{(2\pi)^{{{ d}}/2} }\dd \mb{z} \\
&\leq 4^q \prod_{l=1}^d  \int_{\R} 
\frac{1}{(1+|y_l+z_l \sqrt \lambda |)^{\mu+1}}
\frac{e^{-|z_l|^2/2}  }{\sqrt{2\pi}}\dz_l.
\nonumber
\end{align}
Split each integral according to the set ${\cal A}_{\lambda,l}:=\{z_l\in \R:|z_l|\sqrt \lambda\leq |y_l|/2\}$ and use the bound $1+|y_l+z_l \sqrt \lambda |\geq \1_{{\cal A}_{\lambda,l}^c}+\frac{1}{ 2}(1+|y_l|)\1_{{\cal A}_{\lambda,l}}$: it gives
\begin{align*}
& \hspace{-1cm} \int_{\R} 
\frac{1}{(1+|y_l+z_l \sqrt \lambda |)^{\mu+1}}
\frac{e^{-|z_l|^2/2}  }{\sqrt{2\pi}}\dz_l\\
&\leq   \int_{{\cal A}_{\lambda,l}^c} 
\frac{e^{-|z_l|^2/2}  }{\sqrt{2\pi}}\dz_l
+2^{\mu+1}\int_{{\cal A}_{\lambda,l}^c} 
\frac{1}{(1+|y_l)^{\mu+1}}
\frac{e^{-|z_l|^2/2}  }{\sqrt{2\pi}}\dz_l\\
&\leq   2{\cal N}\left(-\dfrac{|y_l|}{2\sqrt \Lambda}\right)
+2^{\mu+1}\frac{1}{(1+|y_l)^{\mu+1}}\\
&\leq   \frac{C_{\mu,\Lambda}}{(1+|y_l)^{\mu+1}}
\end{align*}
where ${\cal N}(\cdot)$ is the cumulative density function of the standard normal distribution, where $C_{\mu,\Lambda}$ is a finite positive constant depending only on $\mu$ and $\Lambda$, and where at the last line, we have used that ${\cal N}(-x)$ is decreasing when $x\to+\infty$ faster than any polynomial. Therefore, we have proved
\begin{align}
\label{eq:rho:1:bis:1a2} 
{\cal I}_1\leq 4^q \prod_{l=1}^d  \frac{C_{\mu,\Lambda}}{(1+|y_l)^{\mu+1}}.
\end{align}
Now we deal with ${\cal I}_2$, i.e. the part of the integral of  \eqref{eq:rho:1:bis} restricted to 
${\cal A}_\lambda^c$. On this set, $|\mb{z}|^2/2\geq |\mb{z}|^2/4+|\mb{y}|^2/(16 \Lambda)$, therefore we obtain
\begin{align} 
\nonumber{\cal I}_2\leq (1+|\mb{y}|^2)^q \int_{{\cal A}_\lambda^c} 
\frac{e^{-|\mb{z}|^2/2}  }{(2\pi)^{{{ d}}/2} }\dd\mb{z} 
&\leq (1+|\mb{y}|^2)^q 
e^{-|\mb{y}|^2/(16 \Lambda)}   \int_{{\cal A}_\lambda^c} 
\frac{e^{-|\mb{z}|^2/4}  }{(2\pi)^{{{ d}}/2} }\dd \mb{z}\\
&\leq 2^{d/2}(1+|\mb{y}|^2)^q 
e^{-|\mb{y}|^2/(16 \Lambda)} \nonumber\\
& \leq C_{\mu,q,d,\Lambda}\prod_{l=1}^d \frac{1}{(1+|y_l|)^{\mu+1}}\label{eq:rho:1:bis:b}
\end{align}
where the last inequality follows from usual comparisons between exponential and polynomial growth.
Gathering \eqref{eq:rho:1:bis:1a2} and \eqref{eq:rho:1:bis:b} gives the announced estimate \eqref{eq:rho:1:bis} for some constant $C_\nu(\Lambda)$. Theorem \ref{theo:uses} is proved. \qed

\subsection{Proof of Theorem \ref{th:main}} \label{section:error}
In all this section, we assume the assumptions of Theorem \ref{th:main} are in force, without further reference. 

In addition to the coefficients $\AqbM$ calculated with the responses computed using the erroneous functions $\left(\byqM_j\right)_{i< j < N}$ (see the definitions \eqref{eq:yqi:bar:M} and \eqref{eq:Aqi:bar:M}), we define the following coefficients $\Aqb$ calculated with the responses computed with the exact functions $\left(\yq_{j}\right)_{i< j < N}$,
\begin{equation} \label{eq:Aqi:bar}
\Aqb := \displaystyle\frac{1}{M}\sum_{m=1}^{M}\Resq{,m} \phi_{\mb{k}}(X_i^{i,m}),
\end{equation}
and thus 
\begin{equation} \label{eq:yqi:bar}
\byqi(\cdot):= \sum_{\mb{k}\in\Gamma} \Aqb \phi_{\mb{k}}(\cdot).
\end{equation}

\subsubsection{Preliminary results}
The next result states a decomposition of the local error as a superposition of an approximation error (truncation over the functions basis) and a statistical error.
\begin{lemma}[Bias/variance decomposition] \label{lemma:empQEE}
For each $i\in\{0,\ldots,N-1\}$ we have
\begin{align*}
\E\left[ \norm{\byqi(\cdot)-\yq_i(\cdot)}^2_\nu \right] \leq & \sum_{\mb{k}\in (\N^d - \Gamma)} (\Aq)^2 + \dfrac{\LG}{M} \E\left[ \left( \Resq{} \right)^2\right]=:\cE_{i}
\end{align*}
where $\LG$ is the Christoffel number of the functions basis defined in \eqref{eq:christoffel} and the definition of $\cE_{i}$ is stated in \eqref{eq:cEi}.
\end{lemma}

\begin{proof}
 By definition of $\byqi$ and $\yq_i$,
 \begin{align*}
 \Esp{ \norm{\byqi(\cdot)-\yq_i(\cdot)}^2_\nu}  = \Esp{ \norm{-\sum_{\mb{k}\in(\N^d-\Gamma)}\Aq \phi_{\mb{k}}(\cdot) + \sum_{\mb{k}\in \Gamma} \left( \Aqb  - \Aq  \right)\phi_{\mb{k}}(\cdot)}^2_\nu}.
 \end{align*}
 By orthonormality of the $\left\{\phi_{\mb{k}}(\cdot)\right\}$ we obtain
 \begin{align*}
   \Esp{ \norm{\byqi(\cdot) - \yq_i(\cdot)}^2_\nu}
   = & \sum_{\mb{k}\in (\N^d - \Gamma)}(\Aq)^2 + \sum_{\mb{k}\in\Gamma}\Esp{ \left|\Aqb - \Aq  \right|^2}.
 \end{align*}
For $\mb{k}\in \Gamma$, observe that $\Esp{ \Aqb} = \Aq$, therefore
\begin{align*}
 \Esp{\left|\Aqb - \Aq  \right|^2} &=  \Var{\Aqb}
 =  \dfrac{1}{M} \Var{ \Resq{} \phi_{\mb{k}}(X_i^{i})}\nonumber \\ 
 &\leq  \dfrac{1}{M} \norm{ \phi_{\mb{k}}(\cdot)}_\infty^2 \Esp{\left( \Resq{} \right)^2},
\end{align*}
which leads to the desired result.
\qed \end{proof}

Before stating the next preliminary result we define the following $\sigma$-algebras generated by the cloud of simulations, $$\cG_{i}:=\sigma(\cloud_{i}) \quad\mbox{ and }\quad \cG_{i:N}:=\sigma(\cloud_{i}, \dots, \cloud_{N}).$$

\begin{lemma} \label{lemma:secondMoment} We have
 \begin{align}
   &\sum_{\mb{k}\in\Gamma}\left( \E\left[ \left( \ResqM{} - \Resq{}\right) \phi_{\mb{k}}(X_i^i) \right] \right)^2 \nonumber\\
    \label{eq:1:lemma:secondMoment}
&\hspace{3cm}\leq \E\left[ \left( \ResqM{} - \Resq{} \right)^2 \right],\\
&       \sum_{\mb{k}\in\Gamma}\left( \E\left[ \left( \ResqM{} - \Resq{}\right) \phi_{\mb{k}}(X_i^i) \mid \cG_{i+1:N}\right] \right)^2\nonumber\\
 \label{eq:2:lemma:secondMoment}        &\hspace{3cm}\leq \E\left[ \left( \ResqM{} - \Resq{} \right)^2 \mid \cG_{i+1:N}\right].
 \end{align}
\end{lemma}
\begin{proof} We start with the proof of \eqref{eq:1:lemma:secondMoment}.
Let $\beta_{i,\mb{k}}$ be the projection coefficients in $L^2_\nu(\R^d)$ of the function $$h_i(\mb{x}) := \E\left[ \left( \ResqM{} - \Resq{} \right) \mid X_i^i = \mb{x} \right]$$ over the  functions $\left\{\phi_{\mb{k}}(\cdot)\right\}$. Since the latter form an orthonormal basis of $L^2_\nu(\R^d)$ (Proposition \ref{proposition:base}) and owing to the tower property of conditional expectation, we have
\begin{align}\label{eq:formula:betaik}
 \beta_{i,\mb{k}} = & \E\left[ h_i(X^i_{i:N}) \phi_{\mb{k}}(X_i^i) \right] = 
  \E\left[ \left( \ResqM{} - \Resq{} \right) \phi_{\mb{k}}(X_i^i) \right].
\end{align}
By the Parseval identity and the Jensen inequality, we get
\begin{align*}
 \sum_{\mb{k}\in \N^d} \beta_{i,\mb{k}}^2  = \E\left[ h_i^2(X_i^i)\right] 
 & \leq \E\left[  \E\left[ \left( \ResqM{} - \Resq{} \right)^2 \mid X_i^i \right]  \right] \\
 &= \E\left[ \left( \ResqM{} - \Resq{} \right)^2  \right].
\end{align*}
By recalling the representation \eqref{eq:formula:betaik} for $ \beta_{i,\mb{k}}$, the above inequality readily leads to \eqref{eq:1:lemma:secondMoment}.

We now prove \eqref{eq:2:lemma:secondMoment}. Similarly, denote by $\beta_{i,\mb{k}}( \cG_{i+1:N})$  the $L^2_\nu(\R^d)$-projection coefficients of the (random) function $$h_i(\mb{x}, \cG_{i+1:N}) = \E\left[ \left( \ResqM{} - \Resq{} \right)  \mid X_i^i = \mb{x}, \cG_{i+1:N} \right].$$ Remind of the independence of $\cG_{i+1:N}$ and $X^i_{i:N}$. Therefore, using similar arguments as before, we obtain $$\beta_{i,\mb{k}}( \cG_{i+1:N})=\E\left[ \left( \ResqM{} - \Resq{} \right) \phi_{\mb{k}}(X_i^i)\mid \cG_{i+1:N} \right]$$ and 
\begin{align*}
 \sum_{\mb{k}\in \N^d} |\beta_{i,\mb{k}}( \cG_{i+1:N})|^2  
 \leq \E\left[ \left( \ResqM{} - \Resq{} \right)^2\mid \cG_{i+1:N} \right],
\end{align*}
from which \eqref{eq:2:lemma:secondMoment} readily follows. We are done. 
\qed \end{proof}

\subsubsection{Proof of Theorem \ref{th:main}}
Let $\varepsilon_i := \E\left[ \norm{\byqiM(\cdot) - \yq_i(\cdot) }_\nu^2 \right]$. By definition of $\byqiM$, $\yq_i(\cdot)$ and $ \byqi(\cdot)$, 
\begin{align*}
  \varepsilon_i &
  = \Esp{ \norm{ \sum_{\mb{k}\in \Gamma} \AqbM  \phi_{\mb{k}}(\cdot) -\yq_i(\cdot)}_\nu^2} \\
  &= \Esp{ \norm{ \byqi(\cdot) - \yq_i(\cdot) +  \sum_{\mb{k}\in\Gamma} \left( \AqbM - \Aqb \right) \phi_{\mb{k}}(\cdot) }_\nu^2}.
\end{align*}
Then, using the triangular inequality for $\norm{\cdot}_\nu$ norm and the inequality $(a+b)^2 \leq 2a^2 + 2b^2$, we obtain
\begin{align}\label{eq:decomp:ei}
  \varepsilon_i &\leq 2 \Esp{\norm{ \byqi(\cdot) - \yq_i(\cdot)}_\nu^2} + 2\Esp{\norm{\sum_{\mb{k}\in\Gamma} \left( \AqbM - \Aqb \right) \phi_{\mb{k}}(\cdot)}_\nu^2}:=\varepsilon_{i}^{(a)}+\varepsilon_{i}^{(b)}.
\end{align}
The first term $\varepsilon_{i}^{(a)}$ is bounded taking into account Lemma \ref{lemma:empQEE}:
\begin{align}
\label{eq:bound:eia}
\varepsilon_{i}^{(a)} \leq 2 \cE_{i},
\end{align}
Regarding the second term $\varepsilon_{i}^{(b)}$, the orthonormality of $\{\phi_{\mb{k}}(\cdot)\}$ gives
\begin{align}
\varepsilon_{i}^{(b)}
 &= 2 \sum_{\mb{k}\in\Gamma} \E \left[ \left(\AqbM - \Aqb \right)^2\right] \nonumber \\
 & = 2 \sum_{\mb{k}\in\Gamma} \Var{\AqbM - \Aqb} + 2 \sum_{\mb{k}\in\Gamma} \left( \E\left[ \AqbM - \Aqb \right] \right)^2\nonumber\\
 &:=\varepsilon_{i}^{(b,c)}+\varepsilon_{i}^{(b,d)}.\label{eq:eic:eid}
\end{align}
To handle $\varepsilon_{i}^{(b,d)}$ write
$$\AqbM - \Aqb = \displaystyle\frac{1}{M}\sum_{m=1}^{M} \left( \ResqM{,m} - \Resq{,m} \right) \phi_{\mb{k}}(X_i^{i,m}),$$
and apply Lemma \ref{lemma:secondMoment}, thus obtaining
\begin{align}
\nonumber\varepsilon_{i}^{(b,d)}
& = 2 \sum_{\mb{k}\in\Gamma} \left( \E\left[ \left( \ResqM{} - \Resq{} \right) \phi_{\mb{k}}(X_i^{i}) \right] \right)^2  \\
 & \leq 2 \E\left[ \left( \ResqM{} - \Resq{} \right)^2 \right]. \label{eq:eibd}
\end{align}
Now we focus on $\varepsilon_{i}^{(b,c)}$. We recall the conditional variance formula $\Var{X} = \Esp{\Var{X\mid \cG}} + \Var{\Esp{X\mid \cG}}$, available for any scalar random variable $X$ and any sigma-algebra $\cG$: then, 
\begin{align*}
\varepsilon_{i}^{(b,c)}
& =  2 \sum_{\mb{k}\in\Gamma} \Esp{\Var{ \AqbM - \Aqb \mid \cG_{i+1:N}}} +  2\sum_{\mb{k}\in\Gamma} \Var{\Esp{\AqbM - \Aqb\mid\cG_{i+1:N}}}\\
&:=\varepsilon_{i}^{(b,c,e)}+\varepsilon_{i}^{(b,c,f)}.
\end{align*}
We start with $\varepsilon_{i}^{(b,c,e)}$:
\begin{align*}
\varepsilon_{i}^{(b,c,e)} & = 2 \sum_{\mb{k}\in\Gamma} \Esp{ \Var{\frac{1}{M} \sum_{m=1}^{M} \left( \ResqM{,m} - \Resq{,m} \right) \phi_{\mb{k}}(X_i^{i,m}) \mid \cG_{i+1:N}} } \\
 & = 2 \sum_{\mb{k}\in\Gamma} \Esp{ \frac{1}{M^2} \sum_{m=1}^{M} \Var{  \left( \ResqM{,m} - \Resq{,m} \right) \phi_{\mb{k}}(X_i^{i,m}) \mid \cG_{i+1:N}} } \\
 & \leq 2 \sum_{\mb{k}\in\Gamma} \Esp{ \frac{1}{M^2} \sum_{m=1}^{M} \norm{\phi_{\mb{k}}(\cdot)}^2_\infty \Esp{  \left( \ResqM{,m} - \Resq{,m} \right)^2  \mid \cG_{i+1:N}}} \\
 & \leq \frac{2 \LG}{M} \Esp{ \left( \ResqM{} - \Resq{} \right)^2}.
\end{align*}
We continue with $\varepsilon_{i}^{(b,c,f)}$:  use Lemma \ref{lemma:secondMoment} to write
\begin{align*}
\varepsilon_{i}^{(b,c,f)}&= 2 \sum_{\mb{k}\in\Gamma} \Var{\E\left[\AqbM - \Aqb\mid\cG_{i+1:N}\right]} \\
& \leq 2 \sum_{\mb{k}\in\Gamma} \E\left[ \left( \E\left[\AqbM - \Aqb\mid\cG_{i+1:N}\right] \right)^2\right] \\
 & = 2  \Esp{ \sum_{\mb{k}\in\Gamma} \left( \Esp{ \left( \ResqM{} - \Resq{} \right) \phi_{\mb{k}}(X_i^i)\mid\cG_{i+1:N}}\right)^2} \\
 & \leq 2  \Esp{  \Esp {\left( \ResqM{} - \Resq{} \right)^2 \mid\cG_{i+1:N}} } \\
& = 2  \Esp{ \left( \ResqM{} - \Resq{} \right)^2 }. 
\end{align*}
Summarizing the above computations, we get a bound for $\varepsilon_{i}^{(b,c)}$:
\begin{equation}\nonumber
\varepsilon_{i}^{(b,c)} \leq 2 \left( \dfrac{\LG}{M} + 1 \right) \Esp{ \left( \ResqM{} - \Resq{} \right)^2}.
\end{equation}
All in all, in view of \eqref{eq:eic:eid}-\eqref{eq:eibd} and the above, we have derived a bound on $\varepsilon_{i}^{(b)}$:
\begin{equation}\label{eq:eib}
\varepsilon_{i}^{(b)}
\leq 2  \left( \frac{\LG}{M} + 2\right) \E\left[\left(\ResqM{} - \Resq{}\right)^2 \right].
\end{equation}
It remains to estimate the above term with the difference of the responses. Using the Cauchy-Schwarz inequality for the summation in $j$, and  taking into account that $f$ is Lipschitz in $y$, we obtain
\begin{align*}
  &\left(\ResqM{} - \Resq{} \right)^2 \\
  & =\dfrac{1}{(1+|\Xii|^2)^q }\Bigg( \sum_{j=i}^{N-1} \Big( f_j\big(\Xij, \cT_{L^\star}(\byqM_{j+1}(\Xijp)(1+|\Xijp|^2)^{q/2})\big) \\
  &\qquad\qquad\qquad\qquad\qquad- f_j\big( \Xij, \yq_{j+1}(\Xijp)(1+|\Xijp|^2)^{q/2} \big) \Big) \Delta\Bigg)^2  \\
  &\leq \dfrac{(T-t_i)}{(1+|\Xii|^2)^q } \sum_{j=i}^{N-1} \Bigg( f_j\left(X_j^i, \cT_{L^\star}(\byqM_{j+1}(\Xijp)(1+|\Xijp|^2)^{q/2})\right) \\
    &\qquad\qquad\qquad\qquad\qquad- f_j\left(X_j^i, \yq_{j+1}(\Xijp)(1+|\Xijp|^2)^{q/2} \right) \Bigg)^2\Delta \\
\intertext{using that ${\cT_{L^\star}}(\yq_{j+1}(\mb{x})(1+|\mb{x}|^2)^{q/2})=\yq_{j+1}(\mb{x})(1+|\mb{x}|^2)^{q/2}$ and that ${\cT_{L^\star}}$ is $1$-Lipschitz (see Proposition \ref{prop:bound} and Definition \eqref{eq:TL:2})}
& \leq  \dfrac{(T-t_i)}{(1+|\Xii|^2)^q }\sum_{j=i}^{N-1} L_{f}^2\left| \byqM_{j+1}(\Xijp) -  \yq_{j+1}(X_{j+1}^i)\right|^2(1+|\Xijp|^2)^q \Delta.
\end{align*}
Combine this with \eqref{eq:eib} to obtain,
\begin{align*}
 \varepsilon_{i}^{(b)}
 & \leq 2  \left( \frac{\LG}{M} + 2\right) T L_{f}^2 \sum_{j=i}^{N-1} \E\left[\dfrac{(1+|\Xijp|^2)^q}{(1+|\Xii|^2)^q }\left| \byqM_{j+1}(X_{j+1}^i) -  \yq_{j+1}(X_{j+1}^i)\right|^2\right] \Delta.
\end{align*}
Finally, taking advantage of the norm-stability property of the distribution $\nu$ (see \eqref{eq:prop:uses}) we get,
\begin{align} \label{eq:bound:eib}
 \varepsilon_{i}^{(b)} \leq 2 \left( \frac{\LG}{M} + 2\right) T  L_{f}^2 c_\nu \sum_{j=i}^{N-1}  \varepsilon_{j+1} \Delta.
\end{align}
Let us set $C_\varepsilon := 2 ( \dfrac{\LG}{M} + 2) T L_{f}^2 c_\nu$. At this point, in view of \eqref{eq:decomp:ei}, \eqref{eq:bound:eia} and \eqref{eq:bound:eib} we have obtained 
\begin{align*}
 \varepsilon_i \leq & 2\cE_i + C_\varepsilon \sum_{j=i}^{N-1}  \varepsilon_{j+1} \Delta.
\end{align*}
Apply now \cite[Proposition 2.4]{gt2016} with $\alpha=\beta=1/2$ to get the existence of a constant $\cC(C_{\varepsilon},T)$
 (depending only on $C_{\varepsilon},T$) such that 
 \begin{align*}
 \varepsilon_i \leq 2\cE_i + \cC(C_{\varepsilon},T) C_\varepsilon \sum_{j=i}^{N-1} \cE_{j+1}\Delta.
\end{align*}
We are done.
\qed

 \subsection{Proof of Proposition \ref{prop:bound}}
\label{subsectionProof of Proposition prop:bound}
Without loss of generality, we can assume from now that $\eta_f=\eta_g$ because assumptions on $f$ and $g$ are still valid by choosing the maximum parameter between $\eta_f$ and $\eta_g$. Set $\eta:=\eta_f=\eta_g$. First note that the constant $C_\eta$ in \eqref{eq:prop:bound:Ceta} is finite since $b$ and $\sigma$ are bounded, this follows from standard manipulations which are left to the reader. 

Second we prove \eqref{eq:prop:bound} by induction on $i=N,\dots,0$. It is obvious for $i=N$ since $C_y\geq C_g$ (note that $C_\eta\geq 1$). Now assume that
\begin{align}
\label{eq:tmp:1}
\sup_{\mb{x}\in \R^d} \dfrac{|y_{j}(\mb{x})|}  {(1+|\mb{x}|^2)^{\frac\eta 2}}  \le C_j
\end{align}
for $j=N,\dots,i+1$ for some $i\in\{N,\dots,0\}$ and for some finite $C_j$'s, and let us prove that the inequality holds also for $j=i$. In the proof below, we will get a relation  between the $C_j$'s leading to the statement of Proposition \ref{prop:bound}. We start from  \eqref{eq:MDP:fcs} and make use of the growth assumptions on $f$, $g$ and the Lipschitz property of $f$: we have
\begin{align*}
|y_i(\mb{x})|&\leq \E\Bigg[  C_g(1+|X_N|^2)^{\frac\eta 2}\\
&\qquad\quad+\sum_{j=i}^{N-1} \left(C_f(1+|X_j|^2)^{\frac\eta 2}+L_fC_{j+1} (1+|X_{j+1}|^2)^{\frac\eta 2}\right)\Delta\mid X_i=\mb{x}\Bigg]\\
&\leq (1+|\mb{x}|^2)^{\frac\eta 2} \left(C_g C_\eta+TC_f C_\eta+C_\eta L_f \sum_{j=i+1}^{N}  C_j\Delta\right).
\end{align*}
Therefore, \eqref{eq:tmp:1} holds for $j=i$ too, and thus owing to the induction principle, it holds for any $j\in\{0,\dots,N\}$. Moreover, we can take
$$C_i:=C_g C_\eta+TC_f C_\eta+C_\eta L_f \sum_{j=i+1}^{N}  C_j\Delta.$$
The Gronwall lemma implies that $C_i\leq C_\eta(C_g +TC_f )e^{C_\eta L_f T}.$
\qed

 
\section{GPU implementations}
\label{section:GPU implementations}
Two implementations on GPUs of Algorithm \ref{alg:grmdp} are proposed in this section. The first one is a parallelization of the algorithm on a single GPU, we will refer to it as the (mono-)GPU version. This version includes two kernels, one simulates the paths of the forward process and computes the associated responses, the {other one} computes the regression coefficients $(\Aq, \mbk \in \Gamma)$. In the first kernel the initial value of each simulated path of the forward process is stored in a device vector in global memory, it will be read later in the second kernel. In order to minimize the number of memory transactions and therefore maximize performance, all accesses to global memory have been implemented in a coalesced way. The random numbers needed for the path generation of the forward process were generated on the fly (inline generation) taking advantage of the NVIDIA cuRAND library \cite{ref:curand} and the generator MRG32k3a proposed by L'Ecuyer in \cite{ecuyer:99}. Therefore, inside this kernel the random number generator is called as needed. Another approach would be the pre-generation of the random numbers in a separate previous kernel, storing them in GPU global memory and reading them back from this device memory in the next kernel. Both alternatives have advantages and drawbacks. In this work we have chosen inline generation having in mind that this option is faster and saves global memory. Besides, register swapping was not observed on the implementation and the quality of the obtained solutions is similar to the accuracy of pure sequential traditional CPU solutions achieved employing more complex random number generators. Additionally, the pre-generation of random numbers is not feasible when dealing with a huge number of simulations $M$, as required for high dimensional problems, since it will require more memory than the available in the GPU. In the second kernel, in order to compute the regression coefficients, a parallelization not only over the multi-indices $\mbk \in \Gamma$ but also over the simulations $1\leq m \leq M $ was proposed. Thus, blocks of threads parallelize the outer for loop $\forall \mbk \in \Gamma$, whilst the threads inside each block carry out in parallel the inner loop traversing the vectors of the responses and the simulations. We stress that the parallel sum reduction \eqref{eq:Aqi:bar:M} has been implemented in shared memory avoiding bank conflicts in pursuance of getting the most of the GPU.

A host can have several GPUs attached (in our machine, 4 GPUs TITAN BLACK). When dealing with high dimensional problems demanding a huge number of simulations $M$ it is compelling to harness the computing power of these extra GPUs to further reduce execution times. The approach we have followed is to use OpenMP \cite{ref:openmp} to control  all GPUs inside the host node. This strategy involves launching as many CPU threads as GPUs available on the node, so that each CPU thread handles one GPU. We will allude this version as M(ulti)GPU. In order to simulate the forward process and compute the responses, each GPU runs in parallel the first previous mono-GPU kernel, storing in its global memory the initial Student random variable of each path and the computed responses of the paths. Having done that, in order to compute the regression coefficients, the outer loop $\forall \mbk \in \Gamma$ is executed sequentially in CPU, whereas the inner loop on the simulations $1 \leq m \leq M$ is performed in parallel across all GPUs. More precisely, the parallel reduction operation is executed in two levels. The first reduction level is carried out inside each one of the GPUs, it has been performed using the \textit{reduce} primitive of the NVIDIA Thrust library \cite{ref:thrust}. The second level of reduction was made in the plane of CPU threads. This approach of parallelizing  only over the simulations of the forward process behaves optimally when the number of simulations is large enough to fully occupy the GPU. In fact, this is the scenario we will be dealing with when tackling high dimensional problems. Actually, as sketched in Theorem \ref{th:main}, in order to get accurate approximations we need $\frac{2^d \#\Gamma}{M}$ to be small, therefore one need to choose $M$ much bigger than $\#\Gamma$. This fact motivates a small variant of this multi-GPU version, specifically designed for problems where $M$ is so large that we cannot store the initial values of the simulated paths in the set of global memories of the GPUs. In the first kernel in charge of simulating the paths of the forward process and computing the associated responses, now the initial values of the forward process at each time step are not stored, thus saving $d\times M \times 4$ bytes of global memory\footnote{Single precision data type is assumed.}. Instead, we just store the initial random seeds of the launched threads, which will be used in the next kernel in order to resimulate the same initial values of the forward process and allow the computation of the regression coefficients. This strategy is even as efficient as the previous one, on the grounds that GPUs are (most often) faster recomputing data than getting this data back from global memory. This is due to the fact that the version recomputing data is less memory bounded than the first one.


\section{Numerical experiments} \label{sec:numExperiments}

We introduce the following functions 
\begin{align*}
  g(\mbx) &:= 1 + \kappa + \sin\left(\lambda \sum_{l=1}^d x_l\right), \\
  f(t,\mbx,y) &:= \min\left(1, \left[ y - \kappa - 1 -\sin\left( \lambda \sum_{l=1}^d x_l\right)\exp\left(\dfrac{\lambda^2 d(t-T)}{2}\right) \right]^2 \right).
\end{align*}
We aim to solve the PDE:
\begin{align}
 \begin{split} & \partial_t u(t,\mbx) + \dfrac{1}{2} (\Delta_x u)(t,\mbx) + f(t,\mbx,u(t,\mbx)) = 0, \quad t\leq T,\\
 &u(T,\mbx) = g(\mbx). 
 \end{split}\label{eq:pde2OnlyY}
\end{align}
Note that $$u(t,\mbx) = 1+\kappa+\sin\left(\lambda \sum_{l=1}^d x_l\right)\exp\left(\dfrac{\lambda^2 d (t-T)}{2}\right),$$ is solution to \eqref{eq:pde2OnlyY}.
The probabilistic formulation goes as follows:
\begin{align*}
 \dd X_t &= \dd W_t, \\
- \dd Y_t &= f(t,X_t,Y_t)\dt - Z_t\dd W_t,\\
 Y_T &=  g(X_T).
\end{align*}
Then, the processes $(Y,Z)$ satisfy $Y_t = u(t,X_t)$ and $Z_t = \nabla_x u(t,X_t)$. 
For the numerical results the final time $T$ is fixed to 1, $\kappa = \frac{6}{10}$ and $\lambda = \frac{1}{\sqrt{d}}$.

The $\mu$ parameter of the $\nu_l$ Student's t-distribution is chosen as $\mu = 2$. Under this choice $\nu_l(x_l) = 0.5 (1+x_l^2)^{-3/2}$. The marginal CDF and its inverse are given by
\begin{equation}\label{eq:Fnui:Fnulinv}
\begin{split}\Fnul(x)  &= \dfrac{1}{2}\left(\dfrac{x}{\sqrt{x^2+1}}+1\right), \\
    \Fnulinv{l}(u) &= \left\{ \begin{array}{lcc}
             - \sqrt{\dfrac{-u^2+u-0.25}{u(u-1)}} &   \mbox{if}  & 0 < u \leq 0.5, \\ \\
             \sqrt{\dfrac{-u^2+u-0.25}{u(u-1)}} & \mbox{if} & 0.5 < u < 1. \\
             \end{array}
   \right.
   \end{split}
   \end{equation}
   In order to sample according to $\nu$, we draw $d$ independent random variables $(U_1,\ldots,U_d)$ with uniform distribution on $(0,1)$, and we compute
$$X_{i}^i:=\left(\Fnulinv{1}(U_1),\dots,\Fnulinv{d}(U_d)\right)\overset{d}{\sim} \nu.$$

First we present some numerical results in dimension 1. Taking into account that 
\begin{align*}
\dyi{r}(x) &= \pm \lambda\exp\left(\dfrac{t-1}{2}\right) \left[ \1_{r \tiny{\mbox{ even}}} \sin(x) + \1_{r \tiny{\mbox{ odd}}} \cos(x)\right],
\end{align*}
in order to fit the assumptions of Lemma \ref{lemma:lr_lims_zero} we select $p=r$. The values of $q$ are selected using relation \eqref{eq:lr_lims_zero:weaker} under $p=r$ for $r=0,1,2$ and $3$. Therefore, for $q=0$ only the function $\hqi$ is square integrable; for $q=2.1$ the function $\hqi$ and its first spatial derivative  are square integrable and the function $\hqi$ vanish for $u\to 0^+,1^-$; finally, for $q=8.1$ the function $\hqi$ and its first two spatial derivatives are square integrable and vanish for $u\to 0^+,1^-$, its third spatial derivative is square integrable, as well.

In Figure \ref{fig:1D:yqi}, the analytical solution $\yqi$ versus the computed approximation $\byqiM$ are plotted for $t=0.5$ and the values of $q=0,2.1,5.1,8.1$.  The function was approximated on $1000$ independent and identically $\nu$ distributed simulations. As long as $q$ increases the convergence of the algorithm is improved (the truncation  error shrinks to 0, coherently with Proposition \ref{proposition:approError1}). In Figure \ref{fig:1D:hqi} the corresponding results are presented for $\hqi$ and $\bar{h}_{i}^{(q,M)}$; for $q$ large the functions are very flat at the boundaries $0$ and $1$, as expected. The parameters of the GQRMDP algorithm were chosen $\Delta=0.01$, $K=200$ meaning that $\yqi$ was projected over $201$ basis functions, and $M=2\times 10^6$. This selection makes $\frac{2^d \#\Gamma}{M}$ small, in fact it is $2.01\times 10^{-4}$, thus we control well the statistical error of the algorithm.
\begin{figure}[htbp]
\centering
 \subfigure{\includegraphics[scale=0.4]{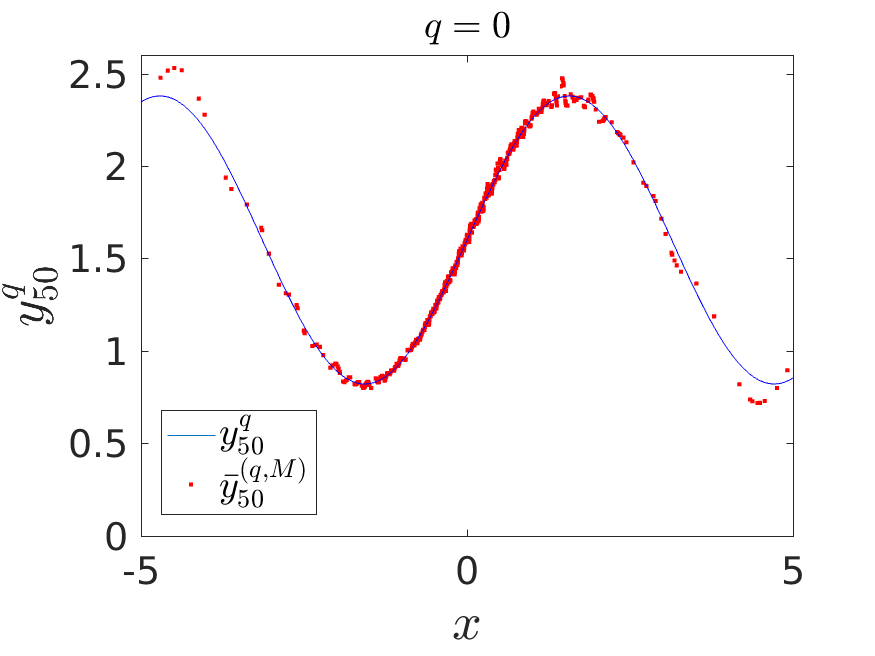}}
 \subfigure{\includegraphics[scale=0.4]{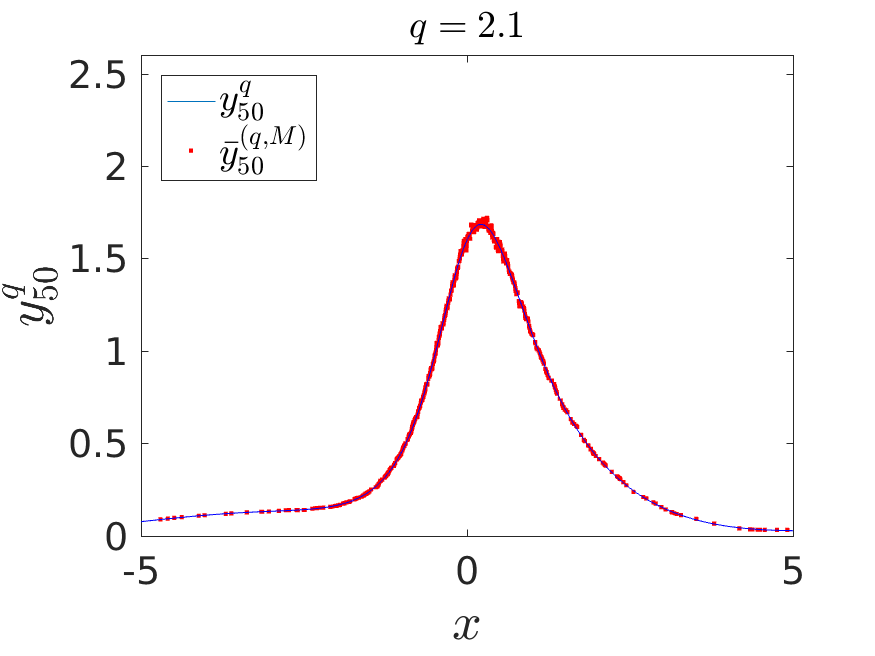}}
 \subfigure{\includegraphics[scale=0.4]{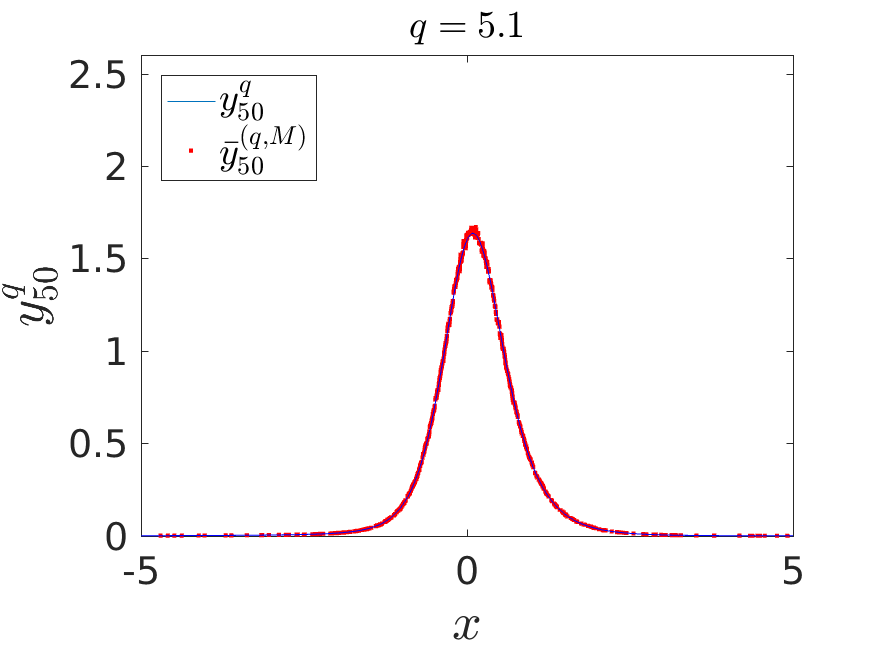}} 
 \subfigure{\includegraphics[scale=0.4]{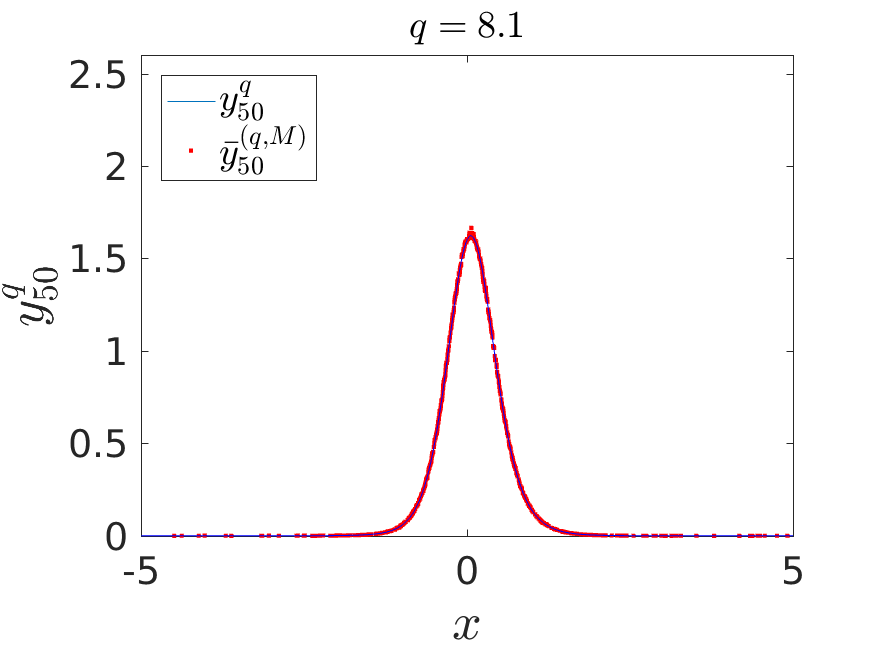}}   
\caption{For $t=0.5$, $\yqi$ versus $\byqiM$ considering $q=0,2.1,5.1,8.1$. $\byqiM$ were computed using $\Delta=0.01$, $K=200$, $M=2\times 10^6$.} 
\label{fig:1D:yqi}
\end{figure}
\begin{figure}[htbp]
\centering
 \subfigure{\includegraphics[scale=0.4]{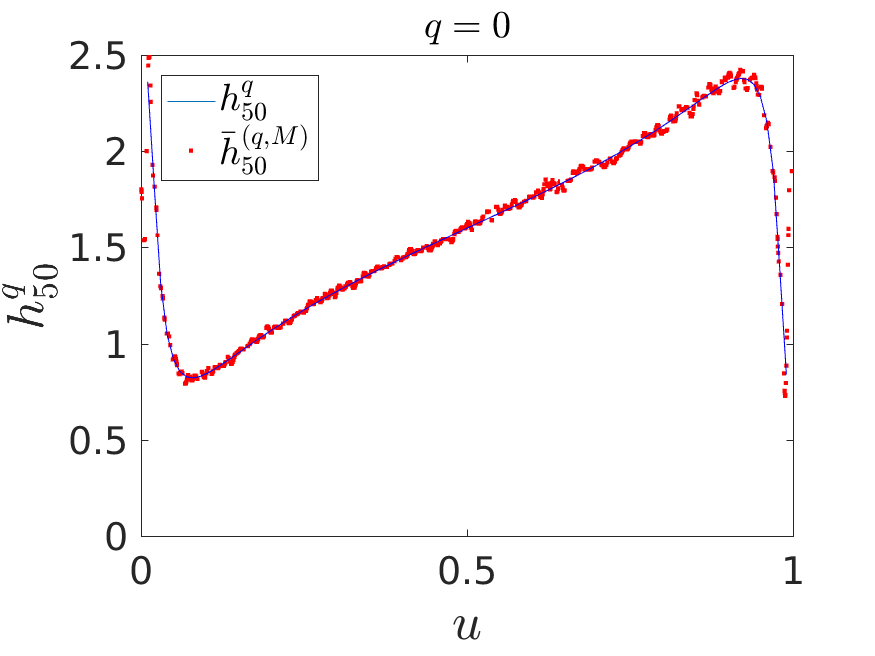}}
 \subfigure{\includegraphics[scale=0.4]{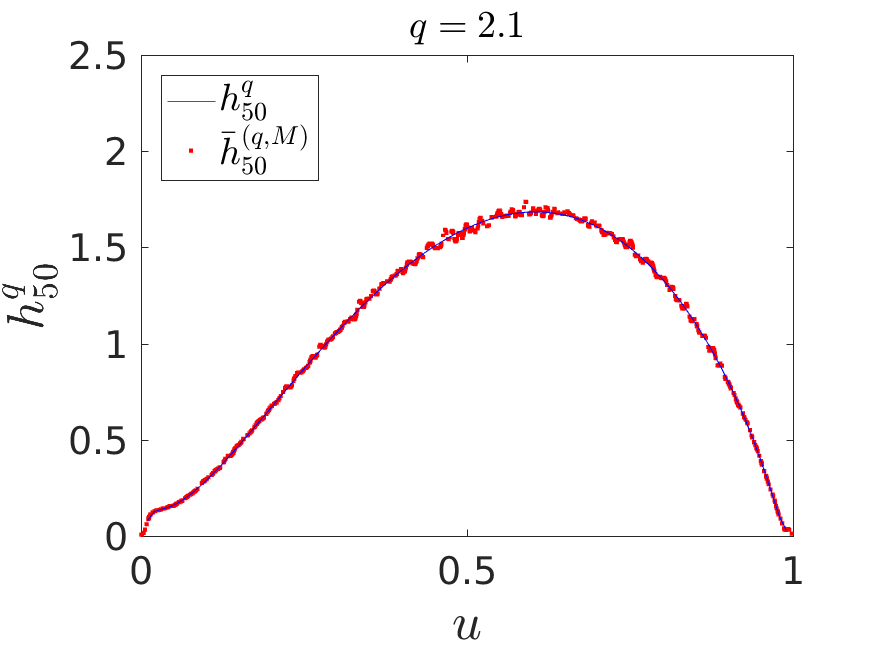}}
  \subfigure{\includegraphics[scale=0.4]{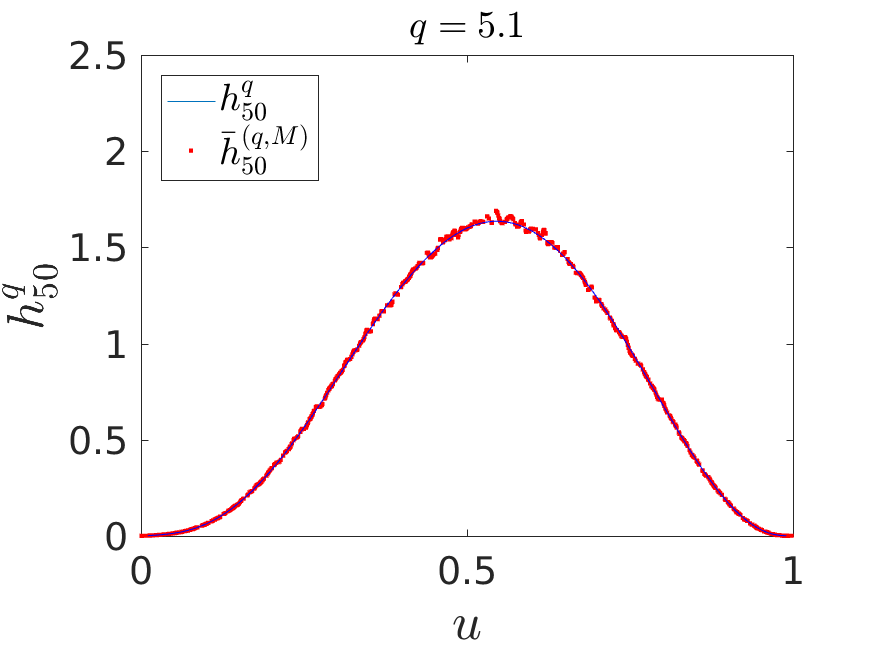}} 
  \subfigure{\includegraphics[scale=0.4]{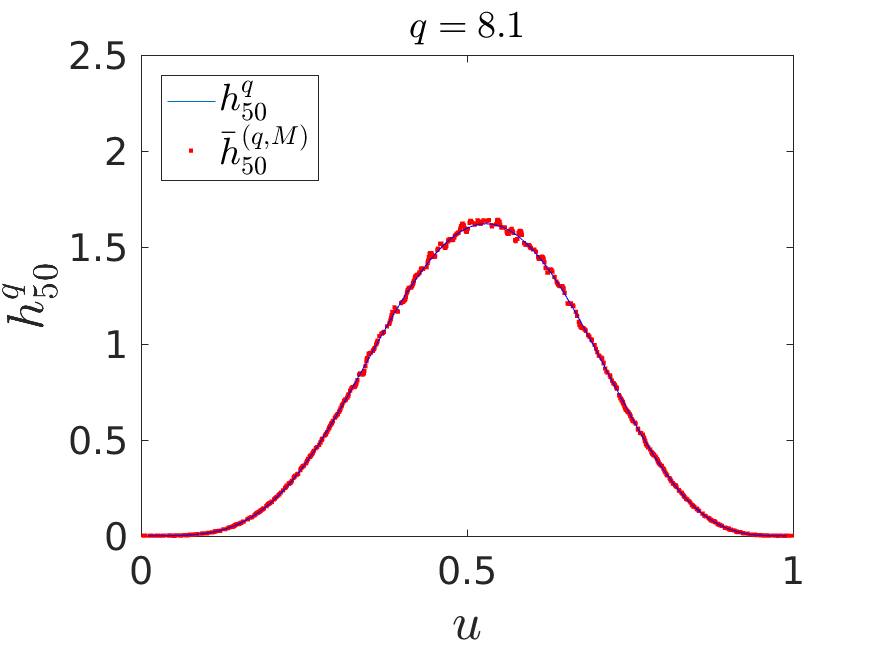}} 
\caption{For $t=0.5$, $\hqi$ versus $\bar{h}_{i}^{(q,M)}$ considering $q=0,2.1,5.1,8.1$. $\bar{h}_{i}^{(q,M)}$ were computed using $\Delta=0.01$, $K=200$, $M=2\times 10^6$.} 
\label{fig:1D:hqi}
\end{figure}

In order to assess  the performance of the algorithm, we compute the average mean squared error ($MSE$) over $10^3$ independent runs of the algorithm for two error indicators:
\begin{equation}
\begin{split}
MSE_{\text{max}}&:= \ln \left\{ 10^{-3} \max_{0\le i \le N-1} \sum_{m=1}^{10^3} \left| \yqi(R_{i,m}) - \byqiM(R_{i,m}) \right|^2 \right\} ,\\
MSE_{\text{av}}&:= \ln \left\{ 10^{-3} N^{-1} \sum_{m=1}^{10^3} \sum_{i=0}^{N-1}  | \yqi(R_{i,m}) - \byqiM(R_{i,m})|^2 \right\} ,
\end{split}
\label{mse}
\end{equation}
where the simulations $\{R_{i,m}; i = 0,\ldots,N-1, \ m=1,\ldots,10^3\}$ are independent and identically $\nu$-distributed, and independently drawn from the simulations used for the GQRMDP scheme.

In order to test the theoretical results we compare the performance of the algorithm according to the choice of the damping exponent $q$ and the multi-indices set. We will show the impact of these choices on the convergence of the approximation of the semi-linear PDE and on the computational performance in terms of computational time and memory consumption.

The numerical experiments have been performed with the following hardware and software configurations: four GPUs GeForce GTX TITAN Black with $6$ GBytes of global memory per device (see \cite{CUDA:keplerArchitecture} for details in the architecture), two multicore Intel Xeon CPUs E5-2620 v2 clocked at $2.10$ GHz ($6$ cores per socket) with $62$ GBytes of RAM, CentOS Linux, NVIDIA CUDA SDK 7.5 and INTEL C compiler 15.0.6. The CPU programs  were optimized and parallelized using OpenMP \cite{ref:openmp}. The multicore CPUs time (CPU) and the GPU time (GPU) will all be measured in seconds in the forthcoming tables. CPU times correspond to executions using $24$ threads so as to take advantage of Intel Hyperthreading. {Unless otherwise indicated}, the results are obtained using single precision, both in CPU and GPU.

All examples will be run using $64$ thread blocks, each with $256$ threads. In the following Table \ref{table:d1_mse} we show results for $d=1$ and the choices of the damping exponent $q=0,2.1,5.1,8.1$. As we have just seen in Figure \ref{fig:1D:yqi}, we expect (see Proposition \ref{proposition:approError1}) that the convergence of the algorithm is improved for large values of the damping exponent $q$. Besides, as long as the time step decreases, our parameter choice is such that $K$ increases and $\frac{2 K}{M}$ decreases so that the algorithm is more and more accurate. The GPU implementation provides a significant reduction in the computational time: the GPU speed-up reaches the value $12.2$. One usually expects a GPU speed-up factor between $3$x-$10$x vs two sockets of CPU \cite{VictorWLee}, therefore both CPU and GPU codes were thoroughly optimized. 

\begin{table}[h]
\caption{$d=1$.}
\label{table:d1_mse}
{\footnotesize
\begin{tabular}{|r|r|r||rr|rr|rr|}
\hline
 & & & \multicolumn{2}{c|}{$q=0$} & \multicolumn{2}{c|}{$q=2.1$} & \multicolumn{2}{c|}{$q=5.1$} \\
\hline         
$\Delta$ & $K$ & $M$ & $MSE_{\text{max}}$ & $MSE_{\text{av}}$ & $MSE_{\text{max}}$ & $MSE_{\text{av}}$ & $MSE_{\text{max}}$ & $MSE_{\text{av}}$ \\
\hline
\hline
 $0.05$ & $100$ & $2\times 10^4$ & $-3.658$ & $-3.868$ & $-4.615$ & $-4.874$ & $-4.517$ & $-5.130$  \\ 
\hline
 $0.02$ & $150$ & $2\times 10^5$ & $-4.427$ & $-4.900$ & $-6.450$ & $-6.740$ & $-6.685$ & $-7.094$  \\ 
\hline
 $0.01$ & $200$ & $2\times 10^6$ & $-4.704$ & $-5.524$ & $-8.465$ & $-8.729$ & $-8.717$ & $-9.096$  \\ 
\hline
\end{tabular}
\begin{tabular}{|r|r|r||rr|r|r|}
\hline
 & & & \multicolumn{2}{c|}{$q=8.1$}  &  &  \\
\hline         
$\Delta$ & $K$ & $M$ & $MSE_{\text{max}}$ & $MSE_{\text{av}}$ & CPU & GPU \\
\hline
\hline
 $0.05$ & $100$ & $2\times 10^4$ & $-4.781$ & $-5.319$ & $0.94$ & $0.09$ \\ 
\hline
 $0.02$ & $150$ & $2\times 10^5$ & $-6.740$ & $-7.221$ & $89.31$ & $7.37$ \\ 
\hline
 $0.01$ & $200$ & $2\times 10^6$ & $-8.830$ & $-9.301$ & $5056.16$ & $414.44$ \\ 
\hline
\end{tabular}
}
\end{table}

In order to assess  the accuracy of the proposed algorithm, in the following Table \ref{table:d1_accuracy}, for $t=0$, $q=0$ and different values of $N$, $K$ and $M$, $99\%$ confidence intervals for $\byqiM(0)$ are shown in the knowledge that $u(0,0) = 1.6$. These confidence intervals were computed using 50 independent runs of the algorithm. Since we are computing the solution at the origin, similar results are obtained for values of $q>0$.
\begin{table}[h]
\caption{$d=1$, $i=0$, $99\%$ confidence intervals for $\byqiM(0)$. $\yqi(0) = 1.6$.}
\label{table:d1_accuracy}
\begin{tabular}{|r|r|r||r|r|}
\hline
$\Delta$ & $K$ & $M$ & $q=0$   \\
\hline
\hline
 $0.05 $ & $100$ & $2\times 10^4$ & $[1.578,1.667]$ \\ %
\hline
$0.02 $ & $150$ & $2\times 10^5$ & $[1.595,1.624]$   \\ %
\hline
 $0.01$ & $300$ & $2\times 10^7$ & $[1.598,1.612]$  \\ %
\hline
\end{tabular}
\end{table}

Now we present results in dimensions greater than one, for which the total degree and the hyperbolic cross index sets are different. {Coherently with Section \ref{sec:approError:Multidimensional}, we start showing the behaviour of the algorithm when the discussed conditions are enforced only on coordinate-wise derivatives of $\hqi$.} Table \ref{table:d2_mse} shows results for $d=2$. Convergence is clearly improved by increasing $q$, considering $q=5.1$ ($p=r=2$), not only $\hqi$ and its first two coordinate-wise derivatives are square integrable but also $\hqi$ and its first coordinate-wise derivatives vanish at the boundaries of $[0,1]^2$. Additionally, convergence is enhanced as long as $\Delta$ decreases: $\DEG$ and  the number of simulations $M$ increase in a way such that $\frac{2^2 \#\Gamma}{M}$ decreases. Finally, the hyperbolic cross index set offers better performance than the total degree set, not only an improved rate of convergence but also much faster computational times. In fact, for $\Delta=0.01$, the hyperbolic cross index set is not only almost $3.4$ times faster than the total degree set (on GPU), but also more accurate achieving \mbox{$MSE_{\text{av}} = -10.137$} instead of \mbox{$MSE_{\text{av}} = -9.107 $}. The GPU speed-up with respect to a fully optimized CPU parallel version reaches the value $11.21$.
\begin{table}[h]
\caption{$d=2$.}
\label{table:d2_mse}
{\footnotesize
\begin{tabular}{|c|r|rr|r||rr|rr|r|r|}
\hline
& & & & & \multicolumn{2}{c|}{$q=0$} & \multicolumn{2}{c|}{$q=5.1$} &  & \\
\hline         
& $\Delta$ & $\DEG$ & $\#\Gamma(\cdot)$ & $M$ & $MSE_{\text{max}}$ & $MSE_{\text{av}}$ & $MSE_{\text{max}}$ & $MSE_{\text{av}}$ & CPU & GPU \\
\hline
\hline
\multirow{3}{*}{\begin{sideways}Total\end{sideways}} & $0.05$ & $20$ & $231$ & $2 \times 10^4$ & $-2.322$ & $-2.773$ & $-4.646$ & $-5.239$ & $3.24$ & $0.33$   \\ 
\cline{2-11}
& $0.02$ & $25$ & $351$ & $2 \times 10^5$ & $-2.869$ & $-3.688$ & $-6.761$ & $-7.186$ & $362.55$ & $35.51$ \\ 
\cline{2-11}
& $0.01$ & $31$ & $528$ & $2\times 10^6$ & $-2.978$ & $-3.966$ & $-8.793$ & $-9.107$ & $25144.70$ & $2243.06$\\ 
\hline
\multirow{3}{*}{\begin{sideways}Hyper.\end{sideways}} & $0.05$ & $19$ & $99$ & $2 \times 10^4$ & $-2.526$ & $-3.062$ & $-5.588$ & $-6.196$ & $1.43$ & $0.16$  \\ 
\cline{2-11}
& $0.02$ & $24$ & $133$ & $2 \times 10^5$ & $-2.745$ & $-3.550$ & $-7.487$ & $-8.191$ & $129.99$ & $12.38$\\ 
\cline{2-11}
& $0.01$ & $30$ & $172$ & $2\times 10^6$ & $-2.898$ & $-3.770$ & $-9.389$ & $-10.137$ & $7221.68$ & $662.54$\\ 
\hline
\end{tabular}
}
\end{table}

Tables \ref{table:d3_mse} and \ref{table:d5_mse} show results for $d=3$ and $5$, respectively. As before, the convergence of the algorithm improves for large values of the damping exponent, so that we benefit from small approximation error (see Proposition \ref{proposition:approErrord} with the conditions  \eqref{eq:lr_lims_zero:d}-\eqref{eq:lr_lims_zero:weaker:d}). Furthermore, the hyperbolic cross index set outperforms the total degree set (which can be expected from the equation \eqref{eq:multiDimTrunErrorIntPartsrMultivariate}). Except for the cases with $\Delta=0.05$, where there are not enough simulations to fully take advantage of four GPUs, the multi-GPU implementation provides a significant reduction in the computational time: the speed-up with respect to the mono-GPU version reaches the value $3.75$. The ideal speed-up would be $4$, although the multi-GPU parallelization introduces some overhead which makes this value unattainable.

\begin{table}[h]
\caption{$d=3$.}
\label{table:d3_mse}
{\footnotesize
\begin{tabular}{|c|r|rr|r||rr|rr|r|r|}
\hline
& & & & & \multicolumn{2}{c|}{$q=0$} & \multicolumn{2}{c|}{$q=5.1$} &  & \\
\hline         
& $\Delta$ & $\DEG$ & $\#\Gamma(\cdot)$ & $M$ & $MSE_{\text{max}}$ & $MSE_{\text{av}}$ & $MSE_{\text{max}}$ & $MSE_{\text{av}}$ & GPU & MGPU \\
\hline
\hline
\multirow{3}{*}{\begin{sideways}Total\end{sideways}} & $0.05$ & $22$ & $2300$ & $2 \times 10^5$ & $ -2.164$ & $-2.785$ & $-5.699$ & $-6.141$  &  $53.89$ & $121.75$ \\ 
\cline{2-11}
& $0.02$ & $24$ & $2925$ & $2 \times 10^6$ & $-2.423$ & $-3.580$ & $-7.693$ & $-8.329$ & $4361.22$ & $1351.36$\\ 
\cline{2-11}
& $0.01$ & $26$ & $3654$ & $2 \times 10^7$ & $-2.484$ & $-3.763$ & $-9.836$ & $-10.487$ & $220687.47$ & $58783.53$\\ 
\hline
\multirow{3}{*}{\begin{sideways}Hyper.\end{sideways}} & $0.05$ & $20$ & $411$ & $2 \times 10^5$ & $-2.255$ & $-3.191$ & $-6.883$ & $-7.662$ &  $8.37$ & $22.99$ \\ 
\cline{2-11}
& $0.02$ & $22$ & $459$ & $2 \times 10^6$ & $-2.249$ & $-3.356$ & $-7.778$ & $-8.875$ & $592.19$ & $191.59$\\ 
\cline{2-11}
& $0.01$ & $24$ & $528$ & $2 \times 10^7$ & $-2.248$ & $-3.380$ & $-8.387$ & $-9.576$ & $27317.65$ & $7306.78$\\ 
\hline
\end{tabular}
}
\end{table}

\begin{table}[h]
\caption{$d=5$.}
\label{table:d5_mse}
{\footnotesize
\begin{tabular}{|c|r|rr|r||rr|rr|r|}
\hline
& & & & & \multicolumn{2}{c|}{$q=0$} & \multicolumn{2}{c|}{$q=5.1$} & \\
\hline         
& $\Delta$ & $\DEG$ & $\#\Gamma(\cdot)$ & $M$ & $MSE_{\text{max}}$ & $MSE_{\text{av}}$ & $MSE_{\text{max}}$ & $MSE_{\text{av}}$ & MGPU \\
\hline
\hline
\multirow{3}{*}{\begin{sideways}$\mbox{ }$Tot.\end{sideways}} & $0.2$ & $20$ & $53130$ & $2 \times 10^6$ & $-1.906$ & $-2.049$ & $-6.537$ & $-6.994$ & $874.36$ \\ 
\cline{2-10}
& $0.1$ & $21$ & $65780$ & $2 \times 10^7$ & $-2.221$ & $-3.356$ & $-8.482$ & $-9.103$ & $19163.72$ \\
\hline
\multirow{3}{*}{\begin{sideways}Hyper.\end{sideways}} & $0.2$ & $16$ & $3042$ & $2 \times 10^6$ & $-2.353$ & $-3.259$ & $-7.312$ & $-7.835$ & $49.27$ \\ 
\cline{2-10}
& $0.1$ & $17$ & $3122$ & $2 \times 10^7$ & $-1.955$ & $-3.278$ & $-7.438$ & $-8.125$ & $773.70$ \\
\cline{2-10}
& $0.1$ & $36$ & $10503$ & $2 \times 10^7$ & $-2.216$ & $-3.494$ & $-8.609$ & $-9.407$ & $2718.07$ \\ 
\cline{2-10}
& $0.1$ & $66$ & $24893$ & $2 \times 10^7$ & $-2.381$ & $-3.603$ & $-8.819$ & $-9.703$ & $6563.77$ \\ 
\hline
\end{tabular}
}
\end{table}

Finally, results for $d=7$, $9$ and $11$ are shown in Table \ref{table:d7_9_11_mse} considering only the scenario performing better, i.e. the multi-GPU version considering a large value of the damping exponent along with the hyperbolic cross index set for the selection of the basis functions on which to project $\yqi$.

\begin{table}[h]
\caption{$d=7, 9, 11$, hyperbolic cross index set.}
\label{table:d7_9_11_mse}
{\footnotesize
\begin{tabular}{|r|r|rr|r||rr|r|}
\hline
& & & & & \multicolumn{2}{c|}{$q=5.1$} &  \\
\hline         
$d$ & $\Delta$ & $\DEG$ & $\#\Gamma_H(\DEG)$ & $M$ & $MSE_{\text{max}}$ & $MSE_{\text{av}}$ & MGPU \\
\hline
\hline
$7$ & $0.1$ & $36$ & $94460$ & $2 \times 10^7$ & $-8.879$ & $-9.574$ & $33284.41$ \\ 
\hline
$9$ & $0.1$ & $12$ & $102656$ & $2 \times 10^7$ & $-8.101$ & $-8.359$ & $44437.40$ \\ 
\hline
$11$ & $0.1$ & $5$ & $75264$ & $2 \times 10^8$ & $-7.828$ & $-8.785$ & $373816.97$ \\ 
\hline
\end{tabular}
}
\end{table}

{Before the end of this section, the behavior of the algorithm is exhibited using larger values of the damping exponent $q$ that account for the case where not only coordinate-wise derivatives of $\hqi$ but also its mixed derivatives are required to vanish at the boundaries of $[0,1]^d$ and to be square integrable. Taking into account that
\begin{align*}
\partial^{(r,\ldots,r)}_x y_i(\mbx) &= \pm \lambda\exp\left(\dfrac{t-1}{2}\right) \left[ \1_{dr \tiny{\mbox{ even}}} \sin(\mbx) + \1_{dr \tiny{\mbox{ odd}}} \cos(\mbx)\right],
\end{align*}
and under $\Lambda=\{0,\ldots,r\}^d$, in order to fit the assumptions of Lemma \ref{lemma:lr_lims_zero:d:Multivariate} we select $p=dr$. In the following numerical experiments $r=2$ or $1$ will be selected. In order to avoid too large values of $q$, which  will hinder the computation of the damping factor $(1+|\mbx|^2)^{q/2}$ when single and even double precision are considered, the $\mu$ parameter of the $\nu_l$ Student's t-distribution is chosen as $\mu=1$ from now on. Under this choice we deal with the Cauchy distribution  $\nu_l(x_l) = \pi^{-1}(1+x^2)^{-1}$, whose CDF and its inverse are given by $F_{\nu_l}(x) = 0.5 + \pi^{-1}\arctan(x)$ and $F^{-1}_{\nu_l}(u) = \tan\left( \pi x - 0.5\pi\right)$, respectively.

Tables \ref{table:d2_mse:mixed} and \ref{table:d3_mse:mixed} show the behavior of the algorithm using a value of $q$ satisfying the condition \eqref{eq:lr_lims_zero:weaker:d:Multivariate} for $r=2$, which boils down to $q>2dr -0.5$ in our benchmark scenario. As expected from equation \eqref{eq:multiDimTrunErrorIntPartsrMultivariate} for the Fourier coefficients of $\hqi$, the hyperbolic cross index set clearly outperforms the total degree set. In order to take the most of our runtime environment, the codes were executed taking advantage of the whole available GPUs and the computations were performed using single precision.

\begin{table}[h]
\caption{$d=2$, $\Lambda=\{0,1,2\}^2$.}
\label{table:d2_mse:mixed}
{\footnotesize
\begin{tabular}{|c|r|rr|r||rr|r|}
\hline
& & & & & \multicolumn{2}{c|}{$q=7.6$} & \\
\hline         
& $\Delta$ & $\DEG$ & $\#\Gamma(\DEG)$ & $M$ & $MSE_{\text{max}}$ & $MSE_{\text{av}}$ & MGPU \\
\hline
\hline
\multirow{3}{*}{\begin{sideways}Total\end{sideways}} & $0.05$ & $31$ & $528$ & $2  \times 10^5$ & $-7.453$ & $-7.871$ & $27.98$ \\ 
\cline{2-8}
& $0.02$ & $41$ & $903$ & $2 \times 10^6$ & $-8.916$ & $-9.598$ & $297.99$  \\ 
\cline{2-8}
& $0.01$ & $51$ & $1378$ & $2\times 10^7$ & $-10.528$ & $-11.416$ &  $13220.48$ \\ 
\hline
\multirow{3}{*}{\begin{sideways}Hyper.\end{sideways}} & $0.05$ & $30$ & $172$ & $2 \times 10^5$ & $-7.710$ & $-8.494$ & $9.94$  \\ 
\cline{2-8}
& $0.02$ & $40$ & $239$ & $2 \times 10^6$ & $-9.019$ & $-10.152$ & $73.66$ \\ 
\cline{2-8}
& $0.01$ & $50$ & $308$ & $2\times 10^7$ & $-10.425$ & $-11.964$ & $2646.07$ \\ 
\hline
\end{tabular}
}
\end{table}

\begin{table}[h]
\caption{$d=3$, $\Lambda = \{0,1,2\}^3$.}
\label{table:d3_mse:mixed}
{\footnotesize
\begin{tabular}{|c|r|rr|r||rr|r|}
\hline
& & & & & \multicolumn{2}{c|}{$q=11.6$} & \\
\hline         
& $\Delta$ & $\DEG$ & $\#\Gamma(\DEG)$ & $M$ & $MSE_{\text{max}}$ & $MSE_{\text{av}}$ & MGPU \\
\hline
\hline
\multirow{3}{*}{\begin{sideways}\mbox{ }Tot.\end{sideways}} & $0.2$ & $142$ & $497640$ & $2 \times 10^6$ & $-5.009$ & $-5.423$ & $6950.81$ \\ 
\cline{2-8}
& $0.2$ & $152$ & $608685$ & $2  \times 10^7$ & $-6.848$ & $-7.085$ & $27347.05$ \\ 
\hline
\multirow{3}{*}{\begin{sideways}\,\,Hyper.\end{sideways}} & $0.2$ & $140$ & $4881$ & $2 \times 10^6$ & $-7.955$ & $-8.693$ & $68.73$  \\ 
\cline{2-8}
& $0.2$ & $150$ & $5322$ & $2 \times 10^7$ & $-8.173$ & $-9.039$ & $205.92$  \\ 
\cline{2-8}
& $0.1$ & $160$ & $5763$ & $2 \times 10^8$ & $-8.319$ & $-9.494$ & $8030.68$ \\ 
\hline
\end{tabular}
}
\end{table}

Ultimately, the semi-linear PDE \eqref{eq:pde2OnlyY}  is solved for dimension $5$ and $7$ in Table  \ref{table:d5_7_mse:mixed}, and for dimension $9$ and $11$ in Table \ref{table:d9_11_mse:mixed}. The GQRMDP algorithm was executed using double precision and considering only the hyperbolic cross index set for the projection of $\yqi$. For $d=5,7$, the so-discussed conditions are imposed for $\Lambda=\{0,1,2\}^d$, while for $d=9,11$, in order to avoid larger values of $q$, these smoothness conditions are only forced for $\Lambda=\{0,1\}^d$. The error indicators and the computational time for $d=7$ inside Table \ref{table:d5_7_mse:mixed} are to be compared with those of Table \ref{table:d7_9_11_mse}: as it was anticipated, for $q=27.6$ the algorithm obtains much better convergence, whereas the execution time is larger due to more accurate computations on double precision. 


\begin{table}[h]
\caption{$d=5, 7$, hyperbolic cross index set, $\Lambda=\{0,1,2\}^d$, double precision.}
\label{table:d5_7_mse:mixed}
{\footnotesize
\begin{tabular}{|r|r|rr|r|r||rr|r|}
\hline         
$d$ & $\Delta$ & $\DEG$ & $\#\Gamma_H(\DEG)$ & $M$ & $q$ & $MSE_{\text{max}}$ & $MSE_{\text{av}}$ & MGPU \\
\hline
\hline
$5$ & $0.1$ & $100$ & $45443$ & $2 \times 10^7 $ & $19.6$ & $-9.065$ & $-10.198$ & $16889.41$ \\ 
\hline
$7$ & $0.1$ & $36$ & $94460$ & $2 \times 10^7$ & $27.6$ & $-15.054$ & $-16.381$ & $48021.18$ \\  
\hline
\end{tabular}
}
\end{table}

\begin{table}[h]
\caption{$d=9, 11$, hyperbolic cross index set, $\Lambda=\{0,1\}^d$, double precision.}
\label{table:d9_11_mse:mixed}
{\footnotesize
\begin{tabular}{|r|r|rr|r|r||rr|r|}
\hline         
$d$ & $\Delta$ & $\DEG$ & $\#\Gamma_H(\DEG)$ & $M$ & $q$ & $MSE_{\text{max}}$ & $MSE_{\text{av}}$ & MGPU \\
\hline
\hline
$9$ & $0.2$ & $5$ & $14336$ & $2 \times 10^7 $ & $17.6$ & $-21.016$ & $-21.852$ & $2466.21$ \\ 
\hline
$11$ & $0.2$ & $2$ & $13312$ & $2 \times 10^7$ & $21.6$ & $-25.559$ & $-27.161$ & $2648.61$ \\  
\hline
\end{tabular}
}
\end{table}

}
\FloatBarrier

\appendix
\section{\appendixname. Technical estimates on Student's t-distribution} \label{appendix:estimatesStudent}
In this appendix, estimates on the CDF of the Student's t-distribution, its inverse and the derivatives of the inverse of the CDF are computed. As we have been doing in the article, we use the same notation $C$ for any generic constants, we will keep the same $C$ from line to line, in order the alleviate the computations, although its value changes. In this appendix $\nu_l$ stands for Student's t-distribution density \eqref{eq:studentDensity} with parameter $\mu$, $\Fnul$ its marginal CDF and $\Fnuinv$ its inverse.

\begin{lemma}[Estimates on $\Fnul$ and $\Fnulinv{l}$] \label{lemma:Fnui:Fnulinv:app} Let $x\in\R$, $u\in(0,1)$. The following estimates hold
\begin{align}\label{eq:Fnui:Fnulinv:app}
\begin{split}
   \Fnul(x)  &\sim_{x\to-\infty} \dfrac{c_\mu}{\mu |x|^{\mu} },  \\
   1-\Fnul(x)  &\sim_{x\to+\infty}  \dfrac{c_\mu}{\mu |x|^{\mu}}  , \\
   \Fnulinv{l}(u) &\sim_{u\to 0^+} -\tilde c_\mu u^{-1/\mu} ,  \\
   \Fnulinv{l}(u) &\sim_{u\to 1^-} \tilde c_\mu (1-u)^{-1/\mu},
   \end{split}
   \end{align}
      with $\tilde c_\mu= \left(\dfrac{c_\mu}{\mu}\right)^{1/\mu}.$ 
\end{lemma}

\begin{proof}In view of \eqref{eq:studentDensity} we know that $\nu_{l}(x_l)\sim_{x_l\to\pm\infty}\dfrac{c_\mu}{|x_l|^{\mu+1}}$. Besides, 
\begin{align*}
 \Fnul(x) = \int_{-\infty}^x \nu_l(x_l)\dd x_l \sim_{x\to-\infty} \int_{-\infty}^x  \dfrac{c_\mu}{|x_l|^{\mu+1}} \dd x_l 
 = \dfrac{c_\mu}{\mu |x|^{\mu} }.
\end{align*}
The distribution $\nu_l$ being symmetric, we have $\Fnul(x)+\Fnul(-x)=1$ and $\Fnulinv{l}(u)=-\Fnulinv{l}(1-u)$. Therefore, the estimate $1- \Fnul(x)$ for $x\to+\infty$ follows from the case $x\to-\infty$.

In order to compute $\Fnulinv{l}$ we use that $\Fnul(\Fnulinv{l}(u))=u$: 
 If $u\to 0^+$, $\Fnulinv{l}(u)\to-\infty$, $u = \Fnul(\Fnulinv{l}(u)) \sim_{u\to0^+} \dfrac{c_\mu}{\mu|\Fnulinv{l}(u)|^\mu}$. Therefore, 
 \begin{equation*}
  |\Fnulinv{l}(u)|^\mu \sim_{u\to0^+}  \dfrac{c_\mu}{\mu u},\qquad |\Fnulinv{l}(u)| \underset{u\to0^+}{=} -\Fnulinv{l}(u) \sim_{u\to0^+}\left(\dfrac{c_\mu}{\mu u}\right)^{1/\mu},
 \end{equation*}
which is the advertised result. The case $u\to 1^-$ follows from the case for $u\to 0^+$ using again the symmetry of the distribution $\nu_l$.
\qed \end{proof}

\begin{lemma}[Estimates on derivatives of $\frac{1}{(1+|\mbx|^2)^{q/2}}$]\label{lemma:1overDumpingFactor:app} Let $\mbx\in\R^d$, $q\in\R^+$ and $n\in\N$. The following upper bound holds
\begin{equation}\label{eq:1overDumpingFactor:app}
\left|\partial^n_{x_l} \dfrac{1}{(1+|\mbx|^2)^{q/2}}\right|\leq C( 1+|\mbx|^2)^{-q/2-n/2}. 
\end{equation} 
\end{lemma}

\begin{proof} We only have to consider the case $n\geq 1$. In order to compute $\partial^n_{x_l} \dfrac{1}{(1+|\mbx|^2)^{q/2}}$ we use the following Fa\`a di Bruno's formula (as long as the partial derivative is only with respect to one variable we can extend in the following way the one-dimensional Fa\`a di Bruno's formula \eqref{eq:1dFaaDiBruno}):
\begin{align}
\begin{split}
\partial_{x_l}^n f(g(\mbx)) &= \sum_{m=1}^n \frac1{m!}{\rm{d}}_{u}^m f(u)\big|_{u=g(\mbx)} \sum_{\mb{j}\in J_{m,n}} \dfrac{n!}{j_1!j_2!\cdots j_m!}
\prod_{i=1}^m \partial_{x_l}^{j_i} g(\mbx),\\
J_{m,n}&=\{\mb{j}=(j_1,\ldots,j_m)\in \N^m_+:j_1+\ldots+j_m = n\}.
\end{split}
\label{eq:FdB:multidim}
\end{align}
Here, consider $f(u)=u^{-q/2}$ and $g(\mbx)=1+|\mbx|^2$: ${\rm d}^m_u f(u) = C_{m,q} u^{-q/2-m}$ and $\partial_{x_l}^{j} g(\mbx)=2 x_l\1_{j=1}+2\1_{j=2}$. 
Observe that in the Fa\`a di Bruno's formula \eqref{eq:FdB:multidim}, the sum over $\mb{j}\in J_{m,n}$ will be made only for $j_i=1$ with $m_1:=\#\{j_i:\mb{j}\in J_{m,n}, j_i=1\}$ terms, and for $j_i=2$ with $m_2:=\#\{j_i:\mb{j}\in J_{m,n}, j_i=2\}$ terms. Owing to this property and by definition of $J_{m,n}$, we have $m_1+2m_2=n$ and $m_1+m_2=m$, which results in $m_1=2m-n$ and $m_2=n-m$. Invoking the form of the derivatives $\partial_{x_l}^{j_i} g(\mbx)$, it readily follows that 
$$\prod_{i=1}^m \partial_{x_l}^{j_i} g(\mbx)=2^m x^{m_1}_l.$$ 
All in all, we deduce
 \begin{align*}
\left|\partial^n_{x_l} \dfrac{1}{(1+|\mbx|^2)^{q/2}}\right| &\leq C \sum_{m=\ceil{n/2}}^n
( 1+|\mbx|^2)^{-q/2-m}  (1+|x_l|)^{2m-n}\\
&\leq C \sum_{m=\ceil{n/2}}^n
( 1+|\mbx|^2)^{-q/2-m}  (1+|\mbx|^2)^{m-n/2}\leq C( 1+|\mbx|^2)^{-q/2-n/2}.
\end{align*}
\qed \end{proof}

\begin{lemma}[Estimates on derivatives of $\frac{1}{(1+|\mbx|^2)^{q/2}}$]\label{lemma:1overDumpingFactorMultivariate:app} Let $\mbx\in\R^d$, $q\in\R^+$, $\mb{n} = (n_1,\ldots,n_d) \in \N_+^d$. The following upper bound holds
\begin{equation}\label{eq:1overDumpingFactorMultivariate:app}
\left|{\partial^{\mb n}_x} \dfrac{1}{(1+|\mbx|^2)^{q/2}}\right|\leq C( 1+|\mbx|^2)^{-q/2-\overline{\mb n}/2}. 
\end{equation} 
\end{lemma}

\begin{proof} {We only have to consider the case $\overline{\mb n}\geq 1$. For such $\mb n$,} we use the following multivariate Fa\`a di Bruno's formula derived from \cite[Corollary 2.10]{constantine:savits:96}
\begin{align}
\begin{split}
{\partial^{\mb n}_x} f(g(\mbx)) &= \sum_{m=1}^{\overline{\mb n}} {\rm{d}}_{u}^m f(u)\big|_{u=g(\mbx)} \sum_{J\in \mb{J}_{m,\mb{n}}} C_{m,J}
\prod_{i=1}^m {\partial^{\mb{j}_i}_{x}} g(\mbx),\\
\mb{J}_{m,\mb{n}}&=\{J= \begin{pmatrix} \mb{j}_1  \\ \vdots \\ \mb{j}_m  \end{pmatrix} \in \N^m \times \N^d : \mb{j}_1+\ldots+\mb{j}_m = \mb{n}, \mb{j}_1,\ldots,\mb{j}_m \neq \mb{0}\},
\end{split}
\label{eq:FdB:multivariate}
\end{align}
where the positive constants $C_{m,J}$ depend on $m$ and the matrix $J$ under consideration.

Note that non coordinate-wise derivatives of $g(\mbx)=1+|\mb{x}|^2$ are all zero. Besides, coordinate-wise derivatives of $g$ are $\partial_{x_l}^{j} g(\mbx)=2 x_l\1_{j=1}+2\1_{j=2}$. For a given $J$, let $a_{l} \in\N\, \forall l=1,\ldots,d$ denote the number of vectors {in $J$} with $1$ in the $l$-th coordinate and zero anywhere else. It holds that 
\begin{align*}
\prod_{i=1}^m \partial^{\mb{j}_i}_x g(\mbx) &= \1_{2m\geq \overline{\mb n}}C x_1^{a_{1}}\cdots x_d^{a_{d}} \mbox{ with } a_{1}+\ldots+a_{d} = 2m-\overline{\mb n},\\
{\left|\prod_{i=1}^m \partial^{\mb{j}_i}_x g(\mbx) \right|}&
{\leq C   (1+|\mbx|^2)^{a_{1}/2}\cdots (1+|\mbx|^2)^{a_{d}/2} \1_{\sum_{l=1}^d a_{l} = 2m-\overline{\mb n}}\leq C   (1+|\mbx|^2)^{m-\overline{\mb n}/2}}
\end{align*}
{for $2m\geq \overline{\mb n}$.}
All in all, we deduce
 \begin{align*}
\left|{\partial^{\mb n}_x} \dfrac{1}{(1+|\mbx|^2)^{q/2}}\right| 
&\leq C \sum_{m=\ceil{\overline{\mb n}/2}}^{\overline{\mb n}}
( 1+|\mbx|^2)^{-q/2-m} (1+|\mbx|^2)^{m-\overline{\mb n}/2} \leq C( 1+|\mbx|^2)^{-q/2-\overline{\mb n}/2}.
\end{align*}
\qed
\end{proof}

\begin{lemma}[Estimates on derivatives of Student's t-distribution]\label{lemma:dnul:app} Let $n\in\N$. The following upper bound holds
\begin{equation}\label{eq:dnul:app}
 |\dnu{n}(x)| \leq C(1+x^2)^{-\frac{\mu+1}{2}-\frac{n}{2}}.
\end{equation}
 \end{lemma}
\begin{proof}
This result follows readily using the previous estimate \eqref{eq:1overDumpingFactor:app}.
\qed \end{proof}

\begin{lemma}[Estimates on derivatives of $\Fnulinv{l}$]\label{lemma:dFnulinv:app}Let $n\in\N$. The following upper bounds hold
\begin{align}\label{eq:dFnulinv:app}
 \left|\dnFnulinvapp{n}{l} \right| \leq_{u\to0^+} C  u^{-\frac{\mu n +1}{\mu}}, \qquad
  \left|\dnFnulinvapp{n}{l} \right| \leq_{u\to1^-} C    (1-u)^{-\frac{\mu n +1}{\mu}},
\end{align}
where $C$ is a constant depending on $\mu$ and $n$.
\end{lemma}

\begin{proof}
Using the combinatorial formula for higher derivatives of inverses \cite{Johnson02thecuriousInverse} and the fact that for $b\in\N$, $b\geq 1$, ${\rm d}^b_{x}\Fnul(x) = \dnu{b-1}(x)$, 
\begin{align*}
 \dnFnulinvapp{n+1}{l} &= \sum_{k=0}^n  \dfrac{(-1)^k}{k!} [\nu_l(x)|_{x=\Fnulinv{l}(u)}]^{-n-k-1} \\
& \qquad \qquad\times\sum_{\substack{b_1+\cdots+b_{k} = n+k \\ b_i\geq 2}} \dfrac{(n+k)!}{b_1!b_2!\cdots b_k!} \prod_{i=1}^{k} \dnu{b_i-1}(x)|_{x=\Fnulinv{l}(u)}.
 \end{align*}
 Therefore, using \eqref{eq:dnul:app},
 \begin{align*}
   |\dnFnulinvapp{n+1}{l}| 
   &\mbox{ } \leq \sum_{k=0}^n C \left[ \left(1+\left(\Fnulinv{l}(u)\right)^2\right)^{-\frac{\mu+1}{2}} \right]^{-n-k-1}  \\
   &\qquad \qquad \times\sum_{\substack{b_1+\cdots+b_{k} = n+k \\ b_i\geq 2}} \quad \prod_{i=1}^{k} \left(1+\left(\Fnulinv{l}(u)\right)^2\right)^{-\frac{\mu+1}{2}-\frac{b_i-1}{2}}  \\
  &\mbox{ } \leq \sum_{k=0}^n C \left(1+\left(\Fnulinv{l}(u)\right)^2\right)^{\frac{\mu+1}{2}(n+k+1)-\frac{k(\mu+1)+n}{2}}\\
  &\mbox{ } \leq C \left(1+\left(\Fnulinv{l}(u)\right)^2\right)^{\frac{\mu(n+1)+1}{2}}.
\end{align*}
Finally, combining the  previous upper bound with \eqref{eq:Fnui:Fnulinv:app} yields
\begin{align*}
 \left|\dnFnulinvapp{n}{l} \right| & \leq_{u\to0^+} C \left|\Fnulinv{l}(u) \right|^{\mu n+1} 
\leq_{u\to0^+} C u^{-\frac{\mu n +1}{\mu}}, \\
  \left|\dnFnulinvapp{n}{l} \right| &\leq_{u\to1^-} C 
   (1-u)^{-\frac{\mu n +1}{\mu}},
\end{align*}
where the constant $C$ may change value along computations.
\qed \end{proof}


\paragraph{\textbf{Compliance with Ethical Standards}}
\begin{itemize}
 \item Conflict of Interest: The authors declare that they have no conflict of interest.
\end{itemize}

\bibliographystyle{spmpsci}      
\bibliography{GlobalRegression_bib}   


\end{document}